\documentclass[twoside,11pt,reqno]{amsart}
\usepackage{amsmath,amssymb,amscd,mathrsfs,epic,wasysym,latexsym,tikz,mathrsfs,cite,hyperref}
\usepackage{pb-diagram}
\usepackage{stmaryrd}

\makeatletter

\numberwithin{equation}{section}

\allowdisplaybreaks

\newtheorem{Proposition}[equation]{Proposition}
\newtheorem{Lemma}[equation]{Lemma}
\newtheorem{Theorem}[equation]{Theorem}
\newtheorem{Corollary}[equation]{Corollary}
\newtheorem{MainTheorem}{Theorem}

\newtheorem{MainCorollary}[MainTheorem]{Corollary}
\theoremstyle{definition}  

\newtheorem{Remark}[equation]{Remark}

\newtheorem{Example}[equation]{Example}

\let\<\langle
\let\>\rangle

\newcommand\Comment[2][\relax]{\space\par\medskip\noindent%
   \fbox{\begin{minipage}{\textwidth}\textbf{Comment\ifx\relax#1\else---#1\fi}\newline%
        #2\end{minipage}}\medskip
}

\def\bi{\text{\boldmath$i$}}
\def\bj{\text{\boldmath$j$}}

\def\bk{\text{\boldmath$k$}}
\def\bl{\text{\boldmath$l$}}

\def\bc{\text{\boldmath$c$}}

\def\b1{\text{\boldmath$1$}}

\def\bd{\text{\boldmath$d$}}

\def\bb{\text{\boldmath$b$}}

\def\bla{\text{\boldmath$\lambda$}}
\def\bmu{\text{\boldmath$\mu$}}
\def\bnu{\text{\boldmath$\nu$}}

\def\ula{\text{\boldmath$\lambda$}}
\def\umu{\text{\boldmath$\mu$}}
\def\unu{\text{\boldmath$\nu$}}

\newcommand{\sm}{\setminus}

\def\pmod#1{\text{ }(\text{\rm mod } #1)\,}

\def\Y#1{\llbracket #1 \rrbracket}

\newcommand{\Hom}{\operatorname{Hom}}

\newcommand{\End}{\operatorname{End}}
\newcommand{\ind}{\operatorname{ind}}
\newcommand{\im}{\operatorname{im}}

\def\sgn{\mathtt{sgn}}
\def\reg{\mathtt{reg}}

\newcommand{\res}{\operatorname{res}}

\newcommand{\cha}{\operatorname{char}}

\newcommand{\noncusp}{\mathrm{nsc}}
\newcommand{\scusp}{\mathrm{sc}}
\newcommand{\sep}{\mathrm{sep}}
\newcommand{\nsep}{\mathrm{nsep}}

\newcommand{\rot}{\mathtt{rot}}

\newcommand{\Z}{\mathbb{Z}}

\newcommand{\N}{\mathbb{Z}_{\geq 0}}
\newcommand{\0}{{\bar 0}}
\newcommand{\1}{{\bar 1}}
\def\eps{{\varepsilon}}
\def\phi{{\varphi}}

\newcommand{\zc}{{\textsf{c}}}

\newcommand{\zz}{{\textsf{z}}}
\newcommand{\ze}{{\textsf{e}}}

\newcommand{\za}{{\textsf{a}}}

\newcommand{\funF}{{\mathcal F}}

\newcommand{\F}{{\mathbb F}}
\newcommand{\Cent}{{\mathcal Z}}

\newcommand{\ga}{\gamma}
\newcommand{\Ga}{\Gamma}
\newcommand{\la}{\lambda}
\newcommand{\La}{\Lambda}
\newcommand{\al}{\alpha}
\newcommand{\be}{\beta}

\def\Si{\mathfrak{S}}
\newcommand{\si}{\sigma}
\newcommand{\om}{\omega}
\newcommand{\Om}{\Omega}

\newcommand{\de}{\delta}
\newcommand{\De}{\Delta}
\newcommand{\ka}{\kappa}
\renewcommand{\th}{\theta}

\newcommand{\Irr}{{\mathrm {Irr}}}

\newcommand{\Ind}{{\mathrm {Ind}}}

\newcommand{\Mat}{{\mathcal {M}}}

\newcommand{\tr}{{\mathrm {tr}}}

\newcommand{\Inv}{\operatorname{{\mathtt{Inv}}}}

\newcommand{\Res}{{\mathrm {Res}}}
\newcommand{\Ann}{{\mathrm {Ann}}}
\newcommand{\C}{{\mathbb C}}
\newcommand{\Q}{{\mathbb Q}}

\newcommand{\D}{{\mathscr D}}

\renewcommand{\mod}{\bmod \,}

\newcommand{\di}{\mathrm{div}}
\newcommand{\Zt}{\tilde Z}

\newcommand{\Zig}{{\sf Z}}

\def\g{{\mathfrak g}}

\def\Par{{\mathscr P}}
\newcommand{\core}{\operatorname{core}}
\newcommand{\Nodes}{\mathsf N}
\newcommand{\Char}{\mathsf{Char}}
\newcommand{\Inc}{\mathsf{Inc}}

\newcommand{\charact}{\operatorname{char}}

\newcommand{\Ru}{\mathsf R}
\newcommand{\Hook}{\mathsf H}
\newcommand{\Gook}{\mathsf G}
\newcommand{\Add}{\mathsf{Add}}
\newcommand{\Rem}{\mathsf{Rem}}

\def\b{\mathfrak{b}}
\def\k{\Bbbk}

\def\T{{\mathtt T}}

\def\height{{\operatorname{ht}}}

\def\wt{{\operatorname{wt}}}

\def\op{{\mathrm{op}}}
\def\re{{\mathrm{re}}}
\def\im{{\mathrm{im}\,}}
\def\onto{{\twoheadrightarrow}}
\def\into{{\hookrightarrow}}

\def\mod#1{#1\!\operatorname{-mod}}

\renewcommand\O{\mathcal O}

\def\iso{\stackrel{\sim}{\longrightarrow}}

\def\alt{{\tt alt}}
\def\col{{\rm col}}
\def\row{{\tt row}}

\def\CH{{\operatorname{ch}_q\,}}
\def\DIM{{\operatorname{dim}_q\,}}
\newcommand{\qdim}{\operatorname{dim}_q}

\def\Car{{\tt C}}

\newcommand{\cont}{\operatorname{cont}}
\newcommand{\quot}{\operatorname{quot}}

\newcommand{\Std}{\mathsf{Std}}
\newcommand{\vn}{\varnothing}

\newcommand{\CT}{\mathsf{CT}}
\newcommand{\Ab}{\operatorname{\mathsf{A}}}

\renewcommand{\t}{\mathtt t}
\newcommand{\s}{\mathtt s}
\newcommand{\itt}{\mathtt{i}}

{\catcode`\|=\active
  \gdef\set#1{\mathinner{\lbrace\,{\mathcode`\|"8000%
  \let|\midvert #1}\,\rbrace}}
}
\def\midvert{\egroup\mid\bgroup}

\colorlet{darkgreen}{green!50!black}
\tikzset{dots/.style={very thick,loosely dotted},
         greendot/.style={fill,circle,color=darkgreen,inner sep=1.5pt,outer sep=0},
         blackdot/.style={fill,circle,color=black,inner sep=1.5pt,outer sep=0},
         graydot/.style={fill,circle,color=gray,inner sep=1.1pt,outer sep=0}
}
\def\greendot(#1,#2){\node[greendot] at(#1,#2){}}
\def\blackdot(#1,#2){\node[blackdot] at(#1,#2){}}
\def\graydot(#1,#2){\node[graydot] at(#1,#2){}}

\newenvironment{braid}{
  \begin{tikzpicture}[baseline=6mm,black,line width=1pt, scale=0.32,
                      draw/.append style={rounded corners},
                      every node/.append style={font=\fontsize{5}{5}\selectfont}]%
  }{\end{tikzpicture}
}

\def\Grid(#1,#2){
  \draw[very thin,gray,step=2mm] (0,0)grid(#1,#2);
  \draw[very thin,darkgreen,step=10mm] (0,0)grid(#1,#2);
}

\newcommand\Tableau[2][\relax]{
  \begin{tikzpicture}[scale=0.5,draw/.append style={thick,black}]
    \ifx\relax#1\relax%
    \else 
      \foreach\box in {#1} { \filldraw[blue!30]\box+(-.5,-.5)rectangle++(.5,.5); }
    \fi
    \newcount\row\newcount\col
    \row=0
    \foreach \Row in {#2} {
       \col=1
       \foreach\k in \Row {
          \draw(\the\col,\the\row)+(-.5,-.5)rectangle++(.5,.5);
          \draw(\the\col,\the\row)node{\k};
          \global\advance\col by 1
       }
       \global\advance\row by -1
    }
  \end{tikzpicture}
}

\newcommand\YoungDiagram[2][\relax]{
  \begin{tikzpicture}[scale=0.5,draw/.append style={thick,black}]
    \ifx\relax#1\relax%
    \else 
    \foreach\box in {#1} {
      \filldraw[blue!30]\box rectangle ++(1,1);
    }
    \fi
    \newcount\row
    \row=0
    \foreach \col in {#2} {
       \draw(1,\the\row)grid ++(\col,1);
       \global\advance\row by -1
    }
  \end{tikzpicture}
}

\begin{document}

\title[Blocks of symmetric groups and semicuspidal KLR algebras]{Blocks of symmetric groups, semicuspidal KLR algebras and zigzag  Schur-Weyl duality}

\author{\sc Anton Evseev}
\address{School of Mathematics, University of Birmingham, Edgbaston, Birmingham B15 2TT, UK}
\email{a.evseev@bham.ac.uk}

\author{\sc Alexander Kleshchev}
\address{Department of Mathematics\\ University of Oregon\\Eugene\\ OR 97403, USA}\email{klesh@uoregon.edu}

\subjclass[2010]{20C08, 20C30, 20G43}

\thanks{The first author is supported by the EPSRC grant EP/L027283/1 and thanks the Max-Planck-Institut for hospitality. The second author is supported by the NSF grant DMS-1161094, the Max-Planck-Institut and the Fulbright Foundation.}

\begin{abstract}
We prove Turner's conjecture, which describes the blocks of the Hecke algebras of the symmetric groups up to derived equivalence as certain explicit Turner double algebras. Turner doubles are Schur-algebra-like `local'  objects, which replace wreath products of Brauer tree algebras 
 in the context of the Brou{\'e} abelian defect group conjecture for blocks of symmetric groups with non-abelian defect groups.
The main tools used in the proof are generalized Schur algebras corresponding to wreath products of zigzag algebras and  imaginary semicuspidal quotients of affine KLR algebras. 
\end{abstract}

\maketitle

\section{Introduction}\label{SIntro}
Let $H_n(q)$ be the Hecke algebra of the symmetric groups $\Si_n$ over a field $\F$ with parameter $q\in \F^\times$. An important special case is $q=1$, when $H_n(q)=\F\Si_n$. Let $e$ be the quantum characteristic of $q$. In this paper we assume that $e>0$, i.e.~there exists $k\in\Z_{> 0}$ such that $1+q+\dots+q^{k-1}=0$, and $e$ is the smallest such $k$. 

Representations of $H_n(q)$ for all $n\geq 0$ categorify the basic module $V(\La_0)$ with highest weight $\La_0$ of the affine Kac-Moody Lie  algebra $\g=\widehat{\mathfrak{sl}}_e(\C)$, see for example \cite{Ariki,Abook,Gr,Kbook}. In particular, each weight space $V(\La_0)_{\La_0-\al}$ for $\al$ in the positive part of the root lattice is identified with the complexified Grothendieck group of the corresponding block $H_\al(q)$ of some $H_n(q)$. 

The Weyl group $W$ of $\g$ acts on the weights of $V(\La_0)$, and the orbits are precisely $\O_d:=W(\La_0-d\de)=W\La_0-d\de$ for all $d\in\Z_{\geq 0}$, where $\de$ is the null-root. Chuang and Rouquier \cite{CR} lift this action of $W$ on the weights to derived equivalences between the corresponding blocks. Therefore, all blocks $H_\al(q)$ with $\La_0-\al\in \O_d$ for a fixed $d$ are derived equivalent.

Moreover, for every $d\in\Z_{\geq 0}$, Rouquier \cite{RoTh} and Chuang and Kessar \cite{CK} identify special representatives $\La_0-\al\in\O_d$ for which the corresponding blocks $H_\al(q)$ have a particularly nice structure. These blocks are known as {\em RoCK blocks}. Thus for any $n$, every block of $H_n(q)$ is derived equivalent to a RoCK block.

Let $H_\al(q)$ be a RoCK block. Turner \cite{Turner, TurnerT, TurnerCat} developed a theory of double algebras and conjectured that $H_\al(q)$ is Morita equivalent to an appropriate double \cite[Conjecture 165]{Turner}. The aim of this paper is to prove Turner's Conjecture. In fact, we prove a slightly more general result stated in terms of cyclotomic KLR algebras over $\Z$. 

To state the result precisely, we now recap Turner's theory as developed in \cite{EK}. Let $Q$ be a type $A_{e-1}$ quiver, and let $P_Q$ be the quotient of the path algebra $\Z Q$ by all paths of length $2$. 
As a $\Z$-module, $P_Q$ has an obvious  basis whose elements are identified with the vertices and the edges of $Q$. We view $P_Q$ as a $\Z$-superalgebra with $P_{Q,\0}$ being the span of the vertices and $P_{Q,\1}$ being the span of the edges. We denote by $\bar x\in \{\0,\1\}$ the degree of a homogeneous element $x$ of any superalgebra.
 Let $n\in \Z_{>0}$, 
 and consider the matrix superalgebra $X:=M_n(P_Q)$. 

For every $d\in\Z_{\geq 0}$ we have a superalgebra structure on $X^{\otimes d}$ induced by that on $X$. So 
$\bigoplus_{d\geq 0} X^{\otimes d}$ is a
superalgebra, 
with the product on each summand $X^{\otimes d}$ being as above, and $xy=0$ for $x\in X^{\otimes d}$ and $y\in X^{\otimes f}$ with $d\neq f$. In  fact, $\bigoplus_{d\geq 0} X^{\otimes d}$ is even a {\em superbialgebra}\, with the coproduct 
\begin{equation*}
\begin{split}
\De \colon 
X^{\otimes d}&\to 
\bigoplus_{0\leq f\leq d}
X^{\otimes f}\otimes X^{\otimes (d-f)},
\\ 
\xi_1\otimes\dots\otimes \xi_d &\mapsto \sum_{0\leq f\leq d} (\xi_1\otimes\dots\otimes \xi_f)\otimes(\xi_{f+1}\otimes\dots\otimes \xi_d).
\end{split}
\end{equation*}
The symmetric group $\Si_d$ acts on $X^{\otimes d}$ by signed place permutations with superalgebra automorphisms, so the set of fixed points 
$
\Inv^dX:=\big(X^{\otimes d}\big)^{\Si_d}
$
 is a subsuperalgebra of $X^{\otimes d}$, and 
$
\Inv X:=\bigoplus_{d\geq 0} \Inv^dX
$
is a subsuperbialgebra of $\bigoplus_{d\geq 0} X^{\otimes d}$.  

There is a superbialgebra structure on $(\Inv X)^*:=\bigoplus_{d\geq 0} (\Inv^dX)^*$ which is dual to that on $\Inv X$. 
We also have an $\Inv X$-bimodule structure on $(\Inv X)^*$ 
defined by
$$
(x\cdot \xi)(\eta)=x(\xi \eta),\ (\xi\cdot x)(\eta)=x( \eta\xi) \qquad(\xi,\eta\in \Inv X,\ x\in (\Inv X)^*).
$$
The {\em Turner double} is the superalgebra 
$DX:=\Inv X\otimes (\Inv X)^*$, with the product defined, using Sweedler's notation for the coproduct $\Delta$, by 
\[
 (\xi\otimes x)(\eta\otimes y)=\sum (-1)^{\bar\xi_{(1)}(\bar\xi_{(2)}+\bar\eta+\bar x)+\bar \eta_{(1)}\bar x} 
 \xi_{(2)}\eta_{(1)}\otimes (x\cdot\eta_{(2)})(\xi_{(1)}\cdot y)
\]
for all homogeneous $\xi,\eta \in \Inv X$ and $x,y\in (\Inv X)^*$. 
We have a decomposition $DX = \bigoplus_{d\ge 0} D^d X$ as a direct sum of superalgebras, where 
\[
 D^d X:= \bigoplus_{0\le f\le d} \Inv^f X \otimes (\Inv^{d-f} X)^*
\]
is a direct sum of $\Z$-supermodules. Each superalgebra $D^d X$ is symmetric. 

The superalgebra $P_Q$ is $\Z$-graded with all vertices of $Q$ being in degree $0$ and all edges in degree $1$. This induces gradings on $M_n (P_Q)$ and $\Inv X=\Inv M_n(P_Q)$ in the natural way. 
We view each $(\Inv^d X)^*$ as a graded $\Z$-module, with the grading induced from $\Inv^d X$ shifted by $2d$, i.e.~for $x\in (\Inv^d X)^*$ we have $\deg x = m$ if and only if $x$ is zero on all graded components of $\Inv^d X$ other than the $(2d-m)$th component. Then 
$D^d X$ is a $\Z$-graded superalgebra concentrated in degrees $0,1,\dots, 2d$. In fact, this grading is a refinement of the superstructure on $D^d X$, in the sense that $(D^d X)_\0$ is the sum of even graded components of $D^d X$ and $(D^d X)_\1$ is the sum of odd graded components. From now on, we forget the superstructure on $D^d X$ and view 
\[
 D_Q (n,d):= D^d X
\]
simply as a graded $\Z$-algebra. 

As before, let $H_{\al} (q)$ be a RoCK block, with 
$\La_0-\al\in \O_d$. 
The precise conditions that $\al$ must satisfy in order for this to be the case are stated in~\S\ref{SSRoCK}.
Let $R^{\La_0}_{\al}$ be the corresponding
cyclotomic KLR algebra, which has a natural grading, see~\S\ref{SSKLR}. 

\begin{MainTheorem}\label{TA}
If $n\ge d$, then
the $\Z$-algebras $R^{\La_0}_{\al}$ and 
$D_Q (n,d)$ are graded Morita equivalent. 
\end{MainTheorem}

For any graded $\Z$-algebra $A$, define $A_{\F}:=A\otimes_{\Z} \F$.
The $\F$-algebra 
$R^{\La_0}_{\al,\F}$
is isomorphic to the RoCK block $H_{\al} (q)$ of a Hecke algebra,~see \cite{BK,R}. Applying this result and the aforementioned theorem of Chuang-Rouquier, one deduces the following from Theorem~\ref{TA}: 

\begin{MainCorollary}\label{CB}
If $n\ge d$, then:
\begin{enumerate}
\item[{\rm (i)}] 
The RoCK block $H_{\al}(q)$ is Morita equivalent to $D_Q(n,d)_{\F}$.
\item[{\rm (ii)}]
For every $\be$ with $\La_0-\be\in \O_d$, the algebra $H_{\be} (q)$ is derived equivalent to $D_Q(n,d)_\F$. 
\end{enumerate}
\end{MainCorollary}

Alperin's Weight Conjecture~\cite{Alperin} predicts an equality between the number of simple modules of an arbitrary block of a finite group and the number of `weights' defined in terms of normalisers of local $p$-subgroups. In the case of blocks with abelian defect group, the 
conjecture of Brou{\'e}~\cite{Broue} lifts Alperin's Weight Conjecture to the categorical level, but no such categorical conjecture is known for  blocks of arbitrary finite groups with non-abelian defect groups.

An important step in the proof of Brou{\'e}'s conjecture for symmetric groups is the theorem~\cite{CK} asserting that, if $q=1$ and $d<\charact \F$, then there is a Morita equivalence between 
the RoCK block $H_{\al} (1)$ and the wreath product algebra 
$H_{\de} (1)^{\otimes d} \rtimes \F\Si_d$. Corollary~\ref{CB} shows that, for an arbitrary block  of a symmetric group, the corresponding double $D_Q(n,d)_{\F}$ is a `local' object that can replace $H_{\de} (1)^{\otimes d} \rtimes \F\Si_d$ in the context of Brou{\'e}'s conjecture. 

In fact, the wreath product $H_{\de} (q)^{\otimes d} \rtimes \F\Si_d$ has a $\Z$-form $(R^{\La_0}_{\delta})^{\otimes d}\rtimes \Z\Si_d$ that is closely related to $D_Q (n,d)$. 
Denote by $\Zig$ the zigzag algebra of type $A_{e-1}$  over $\Z$,
and consider the wreath product 
$W_d:= \Zig^{\otimes d}\rtimes \Z\Si_d$, see~\S\ref{SSQ}. 
Then $\Zig$ is Morita equivalent to $R^{\La_0}_{\delta}$, and more generally $W_d$ is (graded) Morita equivalent to 
$(R^{\La_0}_{\delta})^{\otimes d} \rtimes \F\Si_d$, see the proof of Proposition~\ref{PFourAlgebras}.
On the other hand, the double $D_Q(n,d)$ can be identified with a subalgebra of a
{\em generalized Schur algebra} $S^{\Zig} (n,d)$, and there is a Schur-Weyl duality between $S^{\Zig} (n,d)$ and $W_d$, see~\S\ref{SSTDSA}. In particular, for a certain explicit idempotent $\xi_\om \in D_Q (n,d)$, we have 
\[
\xi_\om D_Q (n,d) \xi_\om = \xi_\om S^{\Zig} (n,d) \xi_\om \cong W_d.
\]
Thus, the idempotent truncation $\xi_{\om} D_Q (n,d) \xi_\om$ is Morita equivalent to 
$R^{\La_0}_{\delta}\rtimes \F\Si_d$.

If $d<\charact \F$ or $\charact \F=0$, the double $D_Q(n,d)_{\F}$ is Morita equivalent to $\xi_{\om} D_Q (n,d)_{\F} \xi_\om\cong  
(R^{\La_0}_{\delta,\F})^{\otimes d}\rtimes \F\Si_d\cong 
H_{\de} (q)^{\otimes d} \rtimes \F \Si_d$, see Proposition~\ref{PFourAlgebras}.
When $d\ge \charact \F>0$, the algebra $D_Q(n,d)_{\F}$ has more isomorphism classes of simple modules than $\xi_{\om} D_Q (n,d)_{\F}\xi_{\om}$. 
As was predicted in~\cite[Conjecture 82]{Turner} and proved in~\cite{Evseev}, a certain explicit idempotent truncation of the RoCK block $H_{\al} (q)$ is Morita equivalent to 
$H_{\de} (q)^{\otimes d} \rtimes \F\Si_d$ in all cases.
Corollary~\ref{CB}(i) strengthens this result, replacing the idempotent truncation by the {\em whole} RoCK block 
$H_\al(q)$.

We now outline the proof of Theorem~\ref{TA} and the contents of the paper. 
Section~\ref{SPrelim} contains some general definitions and notation.
In Section~\ref{SZZ}, we review necessary definitions and results from~\cite{EK}. In particular, we introduce the zigzag algebra $\Zig$ and the wreath product $W_d$. An important role  
is played by the (right) {\em colored permutation 
$W_d$-modules} $M_{\la,\bc}$, which are parametrized by {\em colored compositions} $(\la,\bc)$. Here, $\la$ is a composition of $d$ and $\bc$ is a tuple consisting of elements 
of $\{1,\dots,e-1\}$ that assigns 
a `color' to each part of $\la$. 
In fact, the proof of Theorem~\ref{TA} uses only colored compositions with $n(e-1)$ parts of the form $(\la,\bc^0)$, where $\bc^0$ is given by~\eqref{EC0}, but it is more natural to work with more general colored compositions. 
We define the generalized Schur algebra $S^{\Zig} (n,d)$ as the endomorphism algebra of the direct sum of the appropriate $W_d$-modules $M_{\la,\bc^0}$ and review results identifying the Turner double $D^{\Zig} (n,d)$ as an explicit $\Z$-subalgebra of $S^{\Zig} (n,d)$. 

Section~\ref{SKLR} begins with the definition and standard properties of the KLR algebras $R_{\theta}$ and their cyclotomic quotients $R^{\La_0}_{\theta}$.
In~\S\ref{SSSC} we define the {\em semicuspidal algebra} 
$\hat C_{d\de}$ as an explicit quotient of $R_{d\delta}$.
In~\ref{SSParabolicSC}, we
 consider parabolic subalgebras of $\hat C_{d\delta}$.

In~\S\ref{SSRoCK}, we recall the definition of a RoCK block $R_{\al}^{\La_0}$ and construct a natural homomorphism $\Om\colon \hat C_{d\de} \to R_{\al}^{\La_0}$. 
The quotient $C_{\rho,d}:=\hat C_{d\de}/\ker \Omega$ is isomorphic to an idempotent truncation of $R_{\al}^{\La_0}$, which is later  shown to be Morita equivalent to 
$R_{\al}^{\La_0}$. We note that $C_{\rho,d}$ is finitely generated as a $\Z$-module, but $\hat C_{d\de}$ is not. 
The arguments of~\S\ref{SSRoCK} rely on 
results connecting cyclotomic KLR algebras with the 
combinatorics of standard tableaux and abaci, which are reviewed and developed in~\S\S\ref{SSAb}--\ref{SSDimAndCore}.

In~\S\ref{SSGG} we define the {\em Gelfand-Graev idempotent} 
$\ga^{\la,\bc}\in R_{d\de}$ for every colored composition $(\la,\bc)$ of $d$ and an `uncolored' idempotent $\ga^\om \in R_{d\de}$. The latter may be viewed as a KLR counterpart of 
$\xi_{\om} \in S^{\Zig} (n,d)$. The following two results are key to the proof of Theorem~\ref{TA}:
\begin{enumerate}
\item[{\rm (i)}] There is an explicit algebra isomorphism $W_d \iso \ga^\om C_{\rho,d} \ga^{\om}$ (see Theorem~\ref{TXiIso}).
\item[{\rm (ii)}] If $\ga^\om C_{\rho,d} \ga^{\om}$ is identified with $W_d$  via the isomorphism in (i), then there is an explicit isomorphism 
$M_{\la,\bc} \iso \ga^{\la,\bc} C_{\rho,d} \ga^\om  $ of right 
$W_d$-modules (see Theorem~\ref{TModIso}). 
\end{enumerate}

The isomorphism (i) is a slight generalization of the main result of~\cite{Evseev} and is constructed in \S\S\ref{SSTrunc},\ref{SSTruncIso} using a homomorphism~\cite{KM2} from $W_d$ to 
$\ga^\om \hat C_{d\de} \ga^\om$.
In order to prove (ii), we first show 
that $\ga^{\la,\bc} C_{\rho,d} \ga^\om$ and $M_{\la,\bc}$ have the same rank as free $\Z$-modules, see Corollary~\ref{CDimEqual}. 
This relies on combinatorial results about RoCK blocks proved in~\S\S\ref{SSSST}--\ref{SSCount}. Secondly, in~\S\S\ref{SSIm},\ref{SSLaOm},
we prove several results on the structure of 
$\ga^{\la,\bc} \hat C_{d\de} \ga^\om$. In particular, we find an explicit element that generates
$\ga^{\la,\bc} \hat C_{d\de} \ga^\om$ as a right $\ga^\om \hat C_{d\de} \ga^\om$-module, see Corollary~\ref{CCyc}. We use this element to construct a homomorphism from $M_{\la,\bc}$ to $\ga^{\la,\bc} \hat C_{d\de} \ga^\om$ and ultimately to prove (ii). 

In~\S\ref{SSEnd}, we define the algebra $E(n,d)$ as the endomorphism algebra of the direct sum of (graded shifts of) certain projective left $C_{\rho,d}$-modules $C_{\rho,d}\ga^{\la,\bc}$. 
Using the right $W_d$-modules $\ga^{\la,\bc} C_{\rho,d} \ga^\om$ and the isomorphism (ii), we construct a natural homomorphism 
$\Phi\colon E(n,d) \to S^{\Zig} (n,d)$. Finally, using 
the identification of the Turner double $D_Q (n,d)$ as a subalgebra of 
$S^{\Zig} (n,d)$ stated in Section~\ref{SZZ} as well as  
results about the semicuspidal algebra proved 
in Section~\ref{SSemiCusp}, we show that $\Phi$ is injective with image exactly $D_Q (n,d)$, so that $E(n,d) \cong D_Q (n,d)$ (see Theorem~\ref{TMainIso}). 

A priori, it follows from our constructions that $E(n,d)$ is Morita equivalent to an idempotent truncation of the RoCK block 
$R_{\al}^{\La_0}$. 
In~\S\ref{SSMorita}, we prove that $E(n,d)\cong D_Q(n,d)$ is (graded) Morita equivalent to $R_{\al}^{\La_0}$ by showing that the scalar extensions of $D_Q(n,d)$ and $R_{\al}^{\La_0}$ to any algebraically closed field have the same number of simple modules.

\section{Preliminaries}\label{SPrelim}

For any $m,n\in\Z$, we define the (possibly empty) {\em segment} 
\[
[m,n]:=\{l\in\Z\mid m\leq l\leq n\}.
\]

Let $l,m,n\in \Z_{\ge 0}$ and $I$ be a set. 
For any $i\in I$ and tuples $\bi=(i_1,\dots, i_l)\in I^l$,  
$\bj=(j_1,\dots, j_m)\in I^m$, we set
\begin{equation*}
i^n:= \underbrace{(i,\dots, i)}_n\in I^{n}, \quad
\bi \bj:=(i_1, \dots, i_l, j_1, \dots, j_m) \in I^{l+m}, \quad
\bi^n:=\underbrace{\bi \dots \bi}_n\in I^{ln}.
\end{equation*}
We write $i_1\dots i_l$ instead of $(i_1,\dots,i_l)$ when there is no possibility of confusion. 

\subsection{Partitions and compositions}\label{SSPar}
Fix $n\in\Z_{>0}$ and $d\in \Z_{\ge 0}$. We denote by $\La(n)$ the set of compositions $\la=(\la_1,\dots,\la_n)$ with $\la_1,\dots,\la_n\in\Z_{\geq 0}$. For $\la\in\La(n)$ we write $|\la|:=\la_1+\dots+\la_n$, and set $\La(n,d):=\{\la\in\La(n)\mid |\la|=d\}$. 
If $m\in \Z_{\ge 0}$, we define $m\la := (m\la_1,\dots,m\la_n)\in \La(n)$.

Let $S$ be an arbitrary finite set. 
We define $\La^S(n,d)$ 
 to be the set of tuples $\ula=(\la^{(i)})_{i\in S}$ of compositions in $\La(n)$ such that $\sum_{i\in S}|\la^{(i)}|=d$. 


We denote by $\Par$ the set of all partitions. For $\la=(\la_1,\dots,\la_m)\in\Par$ we write $\ell(\la):=\max\{k\mid \la_k>0\}$ and
$|\la|:=\la_1+\dots+\la_m$. We set $\Par(d):=\{\la\in\Par\mid |\la|=d\}$. We do not assume that the parts $\la_k$ of the partition $\la$ are positive, and we identify a partition $(\la_1,\dots,\la_m)$ with any partition $(\la_1,\dots,\la_m,0,\dots,0)$. 

We define the set of {\em $S$-multipartitions} 
$\Par^S$ as the set of tuples $\ula=(\la^{(i)})_{i\in S}$ of partitions. For $\ula\in \Par^S$, we set $|\ula|:=\sum_{i\in S}|\la^{(i)}|$ and $\Par^S(d):=\{\ula\in\Par^S \mid |\ula|=d\}$. The only multipartition in $\Par^S(0)$ is denoted by $\varnothing$.

We set
$\Nodes^S:=\Z_{>0}\times \Z_{>0}\times S$ 
and refer to the elements of $\Nodes^S$
as {\em nodes}. 
When $S$ has one element, we identify $\Nodes^S$ with ${\mathsf N}:=\Z_{>0}\times \Z_{>0}$. 
If $\ula=(\la^{(i)})_{i\in S}\in\Par^S$ is an $S$-multipartition, its 
{\em Young diagram}, which we often identify with $\ula$, is  
\[
\Y\ula:=\{ (r,s,i)\in\Nodes^S \mid  s\le \la^{(i)}_r  \}. 
\]

If $(r,s,i)\in \Nodes^S$, we say that $(r,s+1,i)$ is {\em the right neighbor} of $(r,s,i)$ and $(r+1,s,i)$ is the {\em bottom neighbor} of $(r,s,i)$.
Define a partial order $<$ on $\Nodes^S$ as follows: $(r,s,i)\leq(r',s',i')$ if and only if 
$i=i'$, $r\le r'$ and $s\le s'$.
Given a multipartition $\bla\in \Par^S$, a function $\T\colon \Y\bla\to \Z_{>0}$ is said to be {\em weakly increasing} if whenever $u\le v$ are in $\Y\ula$ we have $\T(u)\le \T(v)$.  
If $u,v\in \Nodes^S$ and neither $u\le v$ nor $v\le u$, then we say that $u$ and $v$ are {\em independent}. 
Two subsets $U,V \subseteq \Nodes^S$ are said to be {\em independent} if every element of $U$ is independent from every element of $V$. 
We say that a subset $U\subseteq \Nodes^S$ is {\em convex} if whenever $u\le v\le w$ are in $\Nodes^S$ and $u,w\in U$, we have $v\in U$.

A \emph{skew partition} is a pair $(\la, \mu)$ of  partitions such that $\Y\mu\subseteq \Y\la$. We denote it by $\la \sm \mu$   and set $|\la \sm \mu|:= |\la|-|\mu|$. We identify $\la\sm \varnothing$ with $\la$. 

\subsection{Symmetric groups and parabolic subgroups}
Let $d\in \Z_{\ge 0}$. 
We denote by $\Si_d$  the  symmetric group on $\{1,\dots,d\}$ 
and set
$s_r:=(r,r+1) \in \Si_d$ for $r=1,\dots,d-1$ to be the elementary transpositions.
For every $n\in \Z_{>0}$ and 
$\la=(\la_1,\dots,\la_n)\in\La(n,d)$, we have the standard parabolic subgroup
\[
\Si_\la\cong \Si_{\la_1}\times\dots\times \Si_{\la_n}\leq \Si_d.
\]
Moreover, for an ordered set $S=\{1,\dots,l\}$ and
 $\ula=(\la^{(1)},\dots,\la^{(l)})\in\La^S(n,d)$, 
 we define the parabolic subgroup 
$$
\Si_\ula\cong \Si_{\la^{(1)}}\times\dots\times \Si_{\la^{(l)}}\leq \Si_d.
$$

If $g\in \Si_d$ and $g=s_{r_1} \dots s_{r_l}$ is a {\em reduced decomposition} of $g$, i.e.~a decomposition as a product of elementary transpositions with $l$ smallest possible, then we define $\ell(g):=l$ and refer to $l$ as the {\em length} of $g$.
For any $\la,\mu\in \La(n,d)$, 
we denote by $\D^\la$ the set of the minimal length coset representatives for $\Si_d/\Si_\la$,
by ${}^\mu \D$ the set of the minimal length coset representatives for 
$\Si_\mu \backslash \Si_d$ and by ${}^\mu \D^{\la}$ the set
of the minimal length coset representatives for 
$\Si_\mu \backslash \Si_d/\Si_\la$.

\subsection{Algebras and modules}\label{SSAlg}
In this paper we mostly work over the ground ring $\Z$. Occasionally, we use the prime fields $\F_p$ and their algebraic closures $\bar\F_p$.

All gradings in this paper are $\Z$-gradings. 
Let $q$ be an indeterminate. 
Given a graded free $\Z$-module $V\cong \bigoplus_{n=1}^{k}\Z  v_k$ with homogeneous generators $v_k$, we write $\DIM V$ for the graded rank of $V$, i.e.\ $\DIM V:=\sum_{n=1}^k q^{\deg(v_n)}\in\Z[q,q^{-1}]$ and $\dim V:=k$. 
Throughout, $V^n$ denotes the $n$th graded component of $V$ for any $n\in \Z$. 
Given $m\in \Z$, let $q^m V$ denote the module obtained by 
shifting the grading on $V$ up by $m$, i.e. 
$
(q^m V)^n:=V^{n-m}.$ 
We use the notation $V^{>m}:=\bigoplus_{n> m} V^n$.
For any $m\in\Z$, we set $[m]:=(q^{m}-q^{-m})/(q-q^{-1})\in\Z[q,q^{-1}]$. If $m\in\Z_{\geq 0}$, we define $[m]^!:=\prod_{k=1}^m[k]$. 

Let $A$ be a ($\Z$-)graded algebra. 
All $A$-modules are assumed to be graded.
Let
$\mod{A}$ denote the category of all
{\em finitely generated}\, (graded) $A$-modules, with morphisms being {\em degree-preserving} module homomorphisms. 
Given $A$-modules $V$ and $W$, we denote by $\hom_A(V,W)$ the space of morphisms in $\mod{A}$. 
For any $m \in \Z$,
define
$
\Hom_A(V, W)^m := \hom_H(q^m V , W).
$
This is the space of homomorphisms
that are homogeneous of degree $m$.
Set
$$
\Hom_A(V,W) := \bigoplus_{m \in \Z} \Hom_A(V,W)^m. 
$$
In particular, $\End_A(V):=\Hom_A(V,V)$ is a graded algebra. 
All homomorphisms between graded algebras are assumed to be degree-preserving. We have the grading shift functor $q \colon \mod{A} \to \mod{A}$, $V\mapsto q V$. 

Given an $A$-module $V$ and a commutative ring $\k$, we denote by $A_\k:=A\otimes_\Z\k$ the (graded) 
algebra obtained by scalar extension, and by $V_\k:=V\otimes_\Z\k$ the corresponding $A_\k$-module.  
If $B=A/K$ is the quotient of $A$ by an ideal $B$ and $x\in A$, we denote an element $x+ K$ of $B$ simply by $x$ when there is no possibility of confusion.

If $\k$ is a field and $A$ is a finite-dimensional 
graded $\k$-algebra, we denote by $\ell(A)$ the number of irreducible graded $A$-modules up to isomorphism and degree shift.

\section{Zigzag algebras, wreath products and Turner doubles}\label{SZZ}

Throughout the paper, we fix $e\in\Z_{\geq 2}$. 

\subsection{Zigzag algebras and wreath products}\label{SSQ}

Let $Q$ be a type $A_{e-1}$ quiver with vertex set 
\begin{equation}\label{EJ}
J:=\{1,\dots,e-1\}.
\end{equation}
We will use the {\em  zigzag algebra $\Zig$ of type $A_{e-1}$},  defined in \cite{HK} as follows. 
First assume that $e>2$. Let $\hat Q$ be the 
quiver with vertex set $J$ and an arrow
 $\za^{k,j}$ from $j$ to $k$ for all ordered pairs $(k,j)\in J^2$ such that $|k-j|=1$:
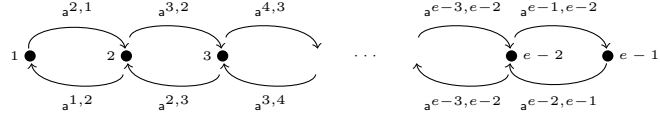
\begin{align*}
\begin{braid}\tikzset{baseline=3mm}
\coordinate (1) at (0,0);
\coordinate (2) at (4,0);
\coordinate (3) at (8,0);
\coordinate (4) at (12,0);
\coordinate (6) at (16,0);
\coordinate (L1) at (20,0);
\coordinate (L) at (24,0);
\draw [thin, black,->,shorten <= 0.1cm, shorten >= 0.1cm]   (1) to[distance=1.5cm,out=100, in=100] (2);
\draw [thin,black,->,shorten <= 0.25cm, shorten >= 0.1cm]   (2) to[distance=1.5cm,out=-100, in=-80] (1);
\draw [thin,black,->,shorten <= 0.25cm, shorten >= 0.1cm]   (2) to[distance=1.5cm,out=80, in=100] (3);
\draw [thin,black,->,shorten <= 0.25cm, shorten >= 0.1cm]   (3) to[distance=1.5cm,out=-100, in=-80] (2);
\draw [thin,black,->,shorten <= 0.25cm, shorten >= 0.1cm]   (3) to[distance=1.5cm,out=80, in=100] (4);
\draw [thin,black,->,shorten <= 0.25cm, shorten >= 0.1cm]   (4) to[distance=1.5cm,out=-100, in=-80] (3);
\draw [thin,black,->,shorten <= 0.25cm, shorten >= 0.1cm]   (6) to[distance=1.5cm,out=80, in=100] (L1);
\draw [thin,black,->,shorten <= 0.25cm, shorten >= 0.1cm]   (L1) to[distance=1.5cm,out=-100, in=-80] (6);
\draw [thin,black,->,shorten <= 0.25cm, shorten >= 0.1cm]   (L1) to[distance=1.5cm,out=80, in=100] (L);
\draw [thin,black,->,shorten <= 0.1cm, shorten >= 0.1cm]   (L) to[distance=1.5cm,out=-100, in=-100] (L1);
\blackdot(0,0);
\blackdot(4,0);
\blackdot(8,0);
\blackdot(20,0);
\blackdot(24,0);
\draw(0,0) node[left]{$1$};
\draw(4,0) node[left]{$2$};
\draw(8,0) node[left]{$3$};
\draw(14,0) node {$\cdots$};
\draw(20,0) node[right]{$e-2$};
\draw(24,0) node[right]{$e-1$};
\draw(2,1.2) node[above]{$\za^{2,1}$};
\draw(6,1.2) node[above]{$\za^{3,2}$};
\draw(10,1.2) node[above]{$\za^{4,3}$};
\draw(18,1.2) node[above]{$\za^{e-3,e-2}$};
\draw(22,1.2) node[above]{$\za^{e-1,e-2}$};
\draw(2,-1.2) node[below]{$\za^{1,2}$};
\draw(6,-1.2) node[below]{$\za^{2,3}$};
\draw(10,-1.2) node[below]{$\za^{3,4}$};
\draw(18,-1.2) node[below]{$\za^{e-3,e-2}$};
\draw(22,-1.2) node[below]{$\za^{e-2,e-1}$};
\end{braid}
\end{align*}
Then $\Zig$ is the path algebra $\Z \hat Q$, generated by length $0$ paths $\ze_j$ for $j\in J$ and length $1$ paths $\za^{k,j}$, subject to the following relations:
\begin{enumerate}
\item All paths of length three or greater are zero.
\item All paths of length two that are not cycles are zero.
\item All cycles of length $2$ based at the same vertex are equal.
\end{enumerate}
The algebra $\Zig$ inherits the path length grading from $\Z \hat Q$. 
If $e =2$, we define $\Zig:= \Z[\zc]/(\zc^2)$, where $\zc$ is an indeterminate in degree 2. 
 
If $k,j\in J$, we say that $k$ and $j$ are {\em neighbors} if $|k-j|=1$. 
If $e>2$, for every vertex $j\in J$ pick its neighbor $k$ and denote $\zc^{(j)}:=\za^{j,k}\za^{k,j}$.  The relations in $\Zig$ imply that $\zc^{(j)}$ is independent of choice of $k$. Define $\zc:= \sum_{j \in J} \zc^{(j)}$. 
 Then in all cases $\Zig$ has a basis
\begin{equation}\label{EBZig}
B_{\Zig}:=
\{\za^{k,j} \mid k \in J,\ j \textup{ is a neighbor of }k\} \cup \{\zc^m\ze_j \mid j \in J,\ m \in \{0,1\}\}, 
\end{equation}
and 
\begin{equation}\label{EDimZZ}
\DIM \Zig=(e-1)(1+q^2)+2 (e-2)q.
\end{equation}
Moreover, using~\eqref{EBZig}, we see that for any $j\in J$
\begin{equation}\label{EkB}
\dim \ze_j \Zig = 
\begin{cases}
4 & \text{if } 1<j<e-1, \\
3& \text{if } j\in\{ 1,e-1\} \text{ and } e>2, \\
2 & \text{if } j=1 \text{ and } e=2.
\end{cases}
\end{equation}

We will also consider the graded {\em wreath products} 
\begin{equation}\label{EWreath}
W_d:= \Zig^{\otimes d}\rtimes \Z\Si_d,
\end{equation}
with $\Z\Si_d$ concentrated in degree $0$. (Note that, unlike \cite{EK}, we do not consider any superstructures here.)
As usual, we identify $\Zig^{\otimes d}$ and $\Z\Si_d$ with the subalgebras $\Zig^{\otimes d}\otimes 1_{\Si_d}$ and $1_{\Zig}^{\otimes d}\otimes \Z\Si_d$ of $W_d$, respectively. The multiplication in $W_d$ is then uniquely determined by the additional requirement that 
\begin{equation}\label{EWreathProductDetermined}
g^{-1}(x_1\otimes\dots\otimes x_d)g=x_{g1}\otimes\dots\otimes x_{gd}
\end{equation}
for $g\in\Si_d$ and $x_1,\dots,x_d\in \Zig$. 
Given $x\in \Zig$ and $1\leq a\leq d$, we denote 
$$
x[a]:=1\otimes \dots\otimes1\otimes x\otimes 1\otimes\dots\otimes  1\in \Zig^{\otimes d},
$$
with $x$ in the $a$th position. We have the idempotents  
$$
\ze_\bj:=\ze_{j_1}\otimes\dots\otimes \ze_{j_d}\in \Zig^{\otimes d} \subseteq W_d\qquad(\bj\in J^d).
$$

Fix $n\in\Z_{>0}$. We define the set of {\em colored compositions}
\begin{equation}\label{ELaCol}
\La^\col(n,d):=\La(n,d)\times J^n.
\end{equation}
Let $(\la,\bc)\in \La^\col(n,d)$ with $\la=(\la_1,\dots,\la_n)$ and $\bc=(c_1,\dots,c_n)$. We define the idempotent
\begin{equation}\label{EFancyE}
\ze_{\la,\bc}:=\ze_{c_1}^{\otimes \la_1}\otimes\dots\otimes \ze_{c_n}^{\otimes \la_n} \in \Zig^{\otimes d}
\end{equation}
and the 
{\em parabolic subalgebra} 
$$
W_{\la,\bc}=\ze_{\la,\bc} \otimes \Z\Si_{\la}\subseteq W_d.
$$
Note that $\ze_{\la,\bc}$ is the identity element of $W_{\la,\bc}$, so $W_{\la,\bc}$ is a (usually {\em non-unital}) subalgebra of $W_d$, isomorphic to the group algebra $\Z \Si_\la$. 

We assign signs $\zeta_j$ to the elements $j\in J$ according to the following rule: 
\begin{equation}\label{EZeta}
\zeta_j=
\left\{
\begin{array}{ll}
+1 &\hbox{if $j$ is odd,}\\
-1 &\hbox{if $j$ is even.}
\end{array}
\right.
\end{equation}

Consider the function 
$\eps_{\la,\bc} \colon \Si_{\la} \to \{ \pm 1\} \subseteq \Z$ defined by 
\begin{equation}\label{Eeps}
\eps_{\la,\bc} (g_1,\dots,g_{n}):=\zeta_{c_1}^{\ell(g_1)} \cdots \zeta_{c_n}^{\ell(g_n)}
\end{equation}
for all $(g_1,\dots,g_n)\in \Si_{\la_1}\times\dots\times \Si_{\la_n}=\Si_\la$.
We define the {\em $\bc$-alternating}\, right module 
$\alt_{\la,\bc}=\Z  \cdot 1_{\la,\bc}$ over $W_{\la,\bc}$ with the action on the basis element $1_{\la,\bc}$ given by 
$$1_{\la,\bc}\cdot(\ze_{\la,\bc}\otimes g)=
\eps_{\la,\bc} (g)
1_{\la,\bc}  \qquad (g\in \Si_{\la}).$$
 We have identified $\Zig^{\otimes d}$ and $\Z\Si_d$ as subalgebras of $W_d$, so we can also view  $\ze_{\la,\bc}$ as an element of $W_d$. Then $W_{\la,\bc}=\ze_{\la,\bc}(\Z\Si_\la)\ze_{\la,\bc}$ and  $\ze_{\la,\bc}W_d$ is naturally a left $W_{\la,\bc}$-module.  We now define the {\em colored permutation module}
\begin{equation}\label{EMLaC}
M_{\la,\bc}:=\alt_{\la,\bc}\otimes_{W_{\la,\bc}}\ze_{\la,\bc}W_d.
\end{equation}
This is a right $W_d$-module 
with generator $m_{\la,\bc}:=1_{\la,\bc}\otimes \ze_{\la,\bc}$.

\begin{Lemma}\label{LBasisMLaC}
For each $j\in J$, set $d_j:=\sum_{1\le r\le n, \, c_r=j} \la_r$. Then
the module $M_{\la,\bc}$ is $\Z$-free, with 
\[
\dim M_{\la,\bc} = 
\begin{cases}
|\Si_d : \Si_\la| 3^{d_1+d_{e-1}} \, 4^{\sum_{j=2}^{e-2} d_j} & \text{if } e>2, \\
|\Si_d: \Si_\la| 2^{d_1} & \text{if } e=2.
\end{cases}
\]
\end{Lemma}

\begin{proof}
This follows from~\eqref{EkB} and \cite[Lemma 5.21]{EK}.
\end{proof}

\subsection{Turner doubles and generalized Schur algebras}\label{SSTDSA}
Let $n\in\Z_{>0}$ and $d\in \Z_{\geq 0}$. Set
\begin{equation}\label{EC0}
\bc^0:=(1,\dots, e-1)^n=(1,\dots, e-1,1,\dots, e-1,\dots, 1,\dots, e-1)\in J^{n(e-1)}.
\end{equation}
We have a bijection 
\begin{align*}
\La^J (n,d) &\iso \La(n(e-1),d), \\
\ula=(\la^{(1)},\dots,\la^{(e-1)}) &\mapsto (\la^{(1)}_1, \dots, \la^{(e-1)}_1, \dots, \la^{(1)}_n, \dots, \la^{(e-1)}_n). 
\end{align*}
In this subsection, we use this bijection to translate 
the results of~\cite[\S 7.2]{EK} into the present notation. 

For any $\la\in \La(n(e-1),d)$, we define 
$$
M^\la:=M_{\la,\bc^0}.
$$
Let 
\begin{equation}\label{EMnd}
M(n,d):=\bigoplus_{\la\in\La(n(e-1),d)}M^\la.
\end{equation}
Following \cite{EK}, we consider the {\em generalized Schur algebra}
$$
S^\Zig(n,d):=\End_{W_d}\left(M(n,d)\right).
$$
Since the algebra $W_d$ is non-negatively graded, so are the modules $M^\la$. Since $M^\la$ has the degree zero generator 
$$m^\la:=m_{\la,\bc^0}$$ 
as a $W_d$-module, it follows that the algebra $S^\Zig(n,d)$ is non-negatively graded. 

For $\la \in\La(n(e-1),d)$, let $\xi_\la\in S^\Zig(n,d)$ be the projection onto the direct summand $M^\la$ of $M(n,d)$ along the decomposition ~\eqref{EMnd}. We always identify $\xi_\mu S^\Zig(n,d) \xi_\la$ with $\Hom_{W_d}(M^\la,M^\mu)$ in the obvious way. 

Let $\la\in \La((n-1)(e-1),d-1)$. For $j\in J$, we define 
$$
\hat\la^j:=(\underbrace{0,\dots,0,1,0,\dots,0}_{e-1 \text{ entries}},\la_1,\dots,\la_{(n-1)(e-1)})\in\La(n(e-1),d),
$$
where $1$ is in the $j$th position. 
Let 
 $z\in \ze_j \Zig \ze_k$ for some $j,k\in J$. By~\cite[Lemma 7.5]{EK}, there exists a unique endomorphism $\itt^\la(z)\in S^{\Zig}(n,d)$ with 
 \[
 \itt^\la (z)\colon
 m^\mu \mapsto 
 \begin{cases}
  m^{\hat\la^j} z[1] & \text{if } \mu=\hat\la^k, \\
 0 & \text{otherwise.}
 \end{cases}
 \]
Moreover, by~\cite[Lemma 7.6]{EK}, we have a (non-unital) injective algebra homomorphism 
\begin{equation}\label{EILa}
\itt^\la\colon \Zig \to S^{\Zig} (n,d), z \mapsto \sum_{j,k\in J} \itt^\la (\ze_j z \ze_k).
\end{equation}

Define $T^\Zig(n,d)$ to be the subalgebra of $S^{\Zig} (n,d)$ generated by the set
\[
S^\Zig (n,d)^0 \cup \bigcup_{\la \in \La((n-1)(e-1),d-1)} \itt^\la (\Zig). 
\]

\begin{Theorem}\label{TGeneration} {\rm \cite[Theorem 7.7]{EK}}
Suppose that $n\ge d$. 
There is a graded algebra isomorphism 
$
D_Q(n,d)\iso T^\Zig(n,d).
$
\end{Theorem}

\begin{Theorem}\label{TSymmetricity}
\cite[Theorem 6.6]{EK}, 
Suppose that $n\ge d$. If $A$ is a subalgebra of $S^\Zig (n,d)$ such that
$T^\Zig(n,d) \subseteq A \subseteq S^\Zig (n,d)$ and 
$A_{\F_p}$ is a symmetric $\F_p$-algebra for every prime $p$, then $A=T^\Zig(n,d)$. 
\end{Theorem}

\section{KLR algebras}\label{SKLR}
\subsection{Lie-theoretic notation}\label{SSLie}

Let 
\[
I:=\Z/e\Z=\{0,\dots,e-1\}.
\]
 We consider the quiver of type $A_{e-1}^{(1)}$ with vertex set  $I$ 
and a directed edge $i \rightarrow j$ whenever $j  = i+1$. 
The corresponding {\em Cartan matrix} 
$({\mathtt c}_{ij})_{i, j \in I}$ is defined by 
\begin{equation*}\label{ECM}
{\mathtt c}_{ij} := \left\{
\begin{array}{rl}
2&\text{if $i=j$},\\
0&\text{if $j \neq i, i \pm 1$},\\
-1&\text{if $i \rightarrow j$ or $i \leftarrow j$},\\
-2&\text{if $i \rightleftarrows j$}.
\end{array}\right.
\end{equation*}

Following \cite{Kac}, we fix a realization of the Cartan matrix $({\mathtt c}_{ ij})_{i,j\in I}$ with the {\em simple roots} 
$\{\al_i\mid i\in I\},$ 
the {\em fundamental dominant weights}  
$\{\La_i\mid i\in I\},$  the {\em normalized invariant form} $(\cdot,\cdot)$ such that
$$
(\al_i,\al_j)={\mathtt c}_{ij}, \quad (\La_i,\al_j)=\de_{ij}\qquad(i,j\in I),
$$
the {\em root system} $\Phi$, the set of {\em positive roots} $\Phi_+$,  
 and the 
{\em null-root} 
\begin{equation}\label{EDelta}
\de:=\al_0+\al_1+\dots+\al_{e-1}\in\Phi_+.
\end{equation}

Let 
$
Q_+ := \bigoplus_{i \in I} \Z_{\geq 0} \alpha_i.
$ 
For $\th \in Q_+$ let $\height(\th)$ be 
the {\em height of~$\th$}, i.e.~$\height(\th)$ is the sum of the 
coefficients when $\th$ is expanded in terms of the simple roots $\alpha_i$. 
For any $m\in\Z_{\ge 0}$, 
the symmetric group $\Si_m$ 
acts from the left on the set $ I^m$ by place permutations. If $\bi=i_1\dots i_m \in I^m$
then its {\em weight} is $|\bi|:=\alpha_{i_1}+\cdots+\alpha_{i_m}\in Q_+$.  Then the
$\Si_m$-orbits on $I^m$ are the sets 
$
I^\th := \{\bi \in I^m\mid |\bi| = \th \} 
$
parametrized by all $\th \in Q_+$ of height $m$. 

We always identify $J=\{1,\dots,e-1\}$ with the subset $I\sm \{0\}$ of 
$I$, cf.~\eqref{EJ}.
Let $\Car'$ be the type $A_{e-1}$ Cartan matrix corresponding to $J$, and let $\Phi'_+\subset \Phi_+$ be the corresponding positive part of the finite root system. 
We define
\begin{align*}
 \Phi_+^{\prec \de}:= \{-\be+n\de\mid \be\in  \Phi'_+,\ n\in\Z_{> 0}\}
\ \text{and}\
\Phi_+^{\succ \de}:=\{\be+n\de\mid \be\in  \Phi'_+,\ n\in\Z_{\geq 0}\}.
\end{align*}
Set $\Phi_+^{\preceq \de}:=\Phi_+^{\prec \de}\sqcup\{\de\}$
and
$
\Phi_+^{\succeq \de}:=\Phi_+^{\succ \de}\sqcup\{\de\}.
$
Note that
$\Phi_+=\Phi_+^\im\sqcup \Phi_+^\re
$, where
$\Phi_+^\im=\{n\de\mid n\in\Z_{>0}\}$
and $\Phi_+^\re= \Phi_+^{\prec \de} \sqcup \Phi_+^{\succ \de}$.

\subsection{Basics on KLR algebras}\label{SSKLR}
Let $\th\in Q_+$ be of height $m$. 
Following~\cite{KL1,R}, the {\em KLR algebra} (of type $A_{e-1}^{(1)}$) is the unital $\Z$-algebra $R_\th$ generated by the elements 
$
\{1_\bi\:|\: \bi\in I^\th\}\cup\{y_1,\dots,y_{m}\}\cup\{\psi_1, \dots,\psi_{m-1}\}
$, 
subject only to the following relations:
\begin{align}
1_\bi 1_\bj &= \de_{\bi,\bj} 1_\bi;
\hspace{11.3mm}{\textstyle\sum_{\bi \in I^\th}} 1_\bi = 1;\label{R1}
\\
y_r 1_\bi &= 1_\bi y_r;
\hspace{20mm}\psi_r 1_\bi = 1_{s_r{ }\bi} \psi_r;\label{R2PsiE}\\
\label{R3Y}
y_r y_s &= y_s y_r;\\
\label{R3YPsi}
\psi_r y_s  &= y_s \psi_r\hspace{42.4mm}\text{if $s \neq r,r+1$};\\
\psi_r \psi_s &= \psi_s \psi_r\hspace{41.8mm}\text{if $|r-s|>1$};\label{R3Psi}\\
\psi_r y_{r+1} 1_\bi &=  (y_r\psi_r+\delta_{i_r,i_{r+1}})1_\bi;
\label{ypsi}\\
y_{r+1} \psi_r1_\bi &= (\psi_ry_r+\delta_{i_r,i_{r+1}})1_\bi;
\label{psiy}\\
\psi_r^21_\bi &= 
\left\{
\begin{array}{ll}
0&\text{if $i_r = i_{r+1}$},\\
1_\bi&\text{if $i_{r+1} \neq i_r, i_r \pm 1$},
\\
(y_{r+1}-y_r)1_\bi&\text{if $i_r \rightarrow i_{r+1}$},\\
(y_r - y_{r+1})1_\bi&\text{if $i_r \leftarrow i_{r+1}$},\\
(y_{r+1} - y_{r})(y_{r}-y_{r+1}) 1_\bi \!\!\!&\text{if $i_r \rightleftarrows i_{r+1}$};
\end{array}
\right.
 \label{quad}
 \end{align}
\begin{align}
\psi_{r}\psi_{r+1} \psi_{r} 1_\bi
&=
\left\{\begin{array}{ll}
(\psi_{r+1} \psi_{r} \psi_{r+1} +1)1_\bi&\text{if $i_{r+2}=i_r \rightarrow i_{r+1}$},\\
(\psi_{r+1} \psi_{r} \psi_{r+1} -1)1_\bi &\text{if $i_{r+2}=i_r \leftarrow i_{r+1}$},\\
\big(\psi_{r+1} \psi_{r} \psi_{r+1} -2y_{r+1}
\\\qquad\:\quad +y_r+y_{r+2}\big)1_\bi
\hspace{2.4mm}&\text{if $i_{r+2}=i_r \rightleftarrows i_{r+1}$},\\
\psi_{r+1} \psi_{r} \psi_{r+1} 1_\bi&\text{otherwise}.
\end{array}\right.
\label{braid}
\end{align}
The {\em cyclotomic KLR algebra} $R_\th^{\La_0}$ is the quotient of  $R_\th$ by the two-sided ideal $I_\th^{\La_0}$ generated by the elements 
$y_1^{\de_{i_1,0}}1_\bi$ 
for all $\bi=(i_1,\dots,i_d)\in I^\th$. We have the natural projection map
\begin{equation}\label{EPi}
\pi_\th \colon R_\th\onto R_\th^{\La_0}=R_\th/I_\th^{\La_0}.
\end{equation} 
The algebras $R_\th$ and $R_\th^{\La_0}$ have $\Z$-gradings determined by setting 
$1_\bi$ to be of degree 0,
$y_r$ of degree $2$, and
$\psi_r 1_\bi$ of degree $-{\mathtt c}_{ i_r,i_{r+1}}$
for all admissible $r$ and $\bi$.

For $\kappa\in I=\Z/e\Z$ and $\bi=(i_1,\dots,i_n)\in I^n$, we set $\bi^{+\kappa}:=(i_1+\kappa,\dots,i_n+\kappa)\in I^n$. 
Then for any $d\in\Z_{>0}$, there is an automorphism
\begin{equation}\label{ERot}
\rot_\kappa\colon  R_{d\de}\to R_{d\de}, 1_\bi\mapsto 1_{\bi^{+\kappa}},\ y_r\mapsto y_r,\ \psi_s\mapsto\psi_s
\end{equation}
for all admissible $\bi,r,s$.

Fixing a {\em preferred reduced decomposition} $w=s_{r_1}\dots s_{r_l}$ for each element $w\in \Si_m$, we define the elements 
$
\psi_w:=\psi_{r_1}\dots\psi_{r_l}\in R_\th.
$
In general, $\psi_w$ depends on the choice of a preferred reduced decomposition of~$w$.

\begin{Theorem}\label{TBasis}{\cite[Theorem 2.5]{KL1}}, \cite[Theorem 3.7]{R} 
Let $\th\in Q_+$ and $m=\height(\th)$. Then  
$$ \{\psi_w y_1^{k_1}\dots y_m^{k_m}1_\bi\mid w\in \Si_m,\ k_1,\dots,k_m\in\Z_{\geq 0}, \ \bi\in I^\th\}
$$ 
is a $\Z$-basis of  $R_\th$. 
\end{Theorem}

As a special case of \cite[Remark 4.20]{KK}, we have 

\begin{Theorem}\label{TBasisCyc} 
Let $\th\in Q_+$. Then  the $\Z$-module 
$R_{\th}^{\La_0}$ is free of finite rank. 
\end{Theorem}

By \cite[Proposition 3.10]{SVV} (see also \cite[Remark 3.19]{Webster}), we have

\begin{Theorem} \label{TSVV} 
Let $\th\in Q_+$. Then for any field\, $\k$, the algebra  
$R_{\th,\k}^{\La_0}$ is symmetric. More precisely, $R_{\th,\k}^{\La_0}$ admits a symmetrizing form of degree $(\La_0-\th,\La_0-\th)$. 
\end{Theorem}


\subsection{Parabolic subalgebras}

For $\theta_1,\dots,\theta_r\in Q_+$ and 
$\theta=\theta_1+\dots+\theta_r$, we have a non-unital embedding 
\begin{equation}\label{EIota}
\iota_{\theta_1,\dots,\theta_r}\colon R_{\theta_1}\otimes\dots\otimes R_{\theta_r}\to R_{\theta}
\end{equation}
whose image is the {\em parabolic subalgebra}  $R_{\theta_1,\dots,\theta_r}\subseteq R_\theta$. Denoting by $1_\theta$ the identity element in $R_\theta$, we set
\begin{equation}\label{E1Par}
1_{\theta_1,\dots,\theta_r}:=\iota_{\theta_1,\dots,\theta_r}(1_{\theta_1}\otimes\dots\otimes 1_{\theta_r}).
\end{equation}
We have
$$
1_{\theta_1,\dots,\theta_r}=\sum_{\bi^{(1)}\in I^{\theta_1},\dots, \bi^{(r)}\in I^{\theta_r}}1_{\bi^{(1)}\dots\bi^{(r)}}.
$$
Note that we always identify $R_{\theta_1}\otimes\dots\otimes R_{\theta_r}$ with 
$
R_{\theta_1,\dots,\theta_r}
$
via $\iota_{\th_1,\dots,\th_r}$. 

We have the corresponding 
induction and restriction functors 
\begin{align*}\Ind_{\theta_1,\dots,\theta_r}&\colon
\mod{R_{\theta_1,\dots,\theta_r}}\to\mod{R_{\theta}},\ W\mapsto  R_\theta 1_{\theta_1,\dots,\theta_r}\otimes_{R_{\theta_1,\dots,\theta_r}} W,
\\
\Res_{\theta_1,\dots,\theta_r}&\colon
\mod{R_{\theta}}\to \mod{R_{\theta_1,\dots,\theta_r}},\ U\mapsto   1_{\theta_1,\dots,\theta_r}U.
\end{align*}
Let $W_1\in\mod{R_{\theta_1}}, \dots, W_r\in\mod{R_{\theta_r}}$. We define
$$W_1\circ\dots\circ W_r:=\Ind_{\theta_1,\dots,\theta_r} W_1\boxtimes \dots\boxtimes W_r.$$ 

We refer to the elements of $I^{\th}$ as {\em words}.
Given $W\in\mod{R_\th}$ and $\bi\in I^\th$, we say that 
$\bi$ is a {\em word of $W$} if $1_\bi W\neq 0$. 
If every $W_\bi$ is free of finite rank as a $\Z$-module, we define 
the {\em formal character}  of $W$  as 
$ \CH W=
\sum_{\bi\in I^\th}(\DIM W_\bi)\bi\in \Z[q,q^{-1}]\cdot I^\th.$ 

Given a composition $\la\in \La(r,m)$ and words $\bi^{(1)}\in I^{\la_1},\dots, \bi^{(r)}\in I^{\la_r}$, a word 
$\bi\in I^m$ is called a {\em shuffle} of $\bi^{(1)},\dots,\bi^{(r)}$ if 
$\bi=g\cdot(\bi^{(1)}\dots \bi^{(r)})$ for some $g\in\D^\la$. By \cite[Lemma 2.20]{KL1}, an element $\bi\in I^m$ is a word of 
$W_1\circ\dots\circ W_r$ if and only if $\bi$ is a shuffle of words $\bi^{(1)},\dots,\bi^{(r)}$ where $\bi^{(s)}$ is a word of $W_s$ for $s=1,\dots,r$. 

We will need the following weak version of the Mackey Theorem for KLR algebras, see \cite[Proposition 3.7]{Evseev} or the proof of \cite[Proposition 2.18]{KL1}:

\begin{Lemma} \label{LMackey} 
Let $\theta_1,\dots,\theta_r,\theta_1',\dots,\theta_t'\in Q_+$ satisfy $\theta_1+\dots+\theta_r=\theta_1'+\dots+\theta_t'=:\theta$. Define $m:=\height(\theta)$, 
$\la:=(\height(\theta_1),\dots,\height(\theta_r))\in\La(r,m)$, and  $\la':=(\height(\theta_1'),\dots,\height(\theta_t'))\in\La(t,m)$. Then
$$1_{\theta_1,\dots,\theta_r}R_\theta1_{\theta_1',\dots,\theta_t'}=\sum_{w\in{}^\la\D^{\la'}}R_{\theta_1,\dots,\theta_r}\psi_w R_{\theta_1',\dots,\theta_t'}.
$$ 
\end{Lemma}

With the notation as in the beginning of the subsection, we have the {\em parabolic subalgebra}
$$
R^{\La_0}_{\theta_1,\dots,\theta_r}:=\pi_\theta(R_{\theta_1,\dots,\theta_r})\subseteq R^{\La_0}_\theta.
$$
Let $\th,\eta\in Q_+$. We have a natural embedding 
$\zeta_{\th,\eta}\colon R_\th\to R_{\th,\eta}, \ x\mapsto 
\iota_{\th,\eta}(x\otimes 1_\eta)$. The map $\pi_{\th+\eta}\circ \zeta_{\th,\eta}$ factors through the quotient $R_\th^{\La_0}$ to give the natural {\em unital} algebra homomorphism   
\begin{equation}\label{EZetaHom}
\zeta_{\th,\eta}\colon R_\th^{\La_0}\to R_{\th,\eta}^{\La_0}.
\end{equation}

\subsection{Divided power idempotents}\label{SSDiv}
Fix $i\in I$. Let $m\in \Z_{\ge 0}$ and denote by $w_0$ the longest element of $\Si_m$. 
The algebra $R_{m\al_i}$ is known to be the nil-Hecke algebra 
and has an idempotent 
$1_{i^{(m)}}:=\psi_{w_0}\prod_{s=1}^m y_s^{s-1}$, cf. \cite{KL1}. The fact that  $1_{i^{(m)}}$ is an idempotent follows immediately from the equality 
\begin{equation}\label{EKLId}
1_{i^{(m)}}\psi_{w_0}=\psi_{w_0}
\end{equation}
noted in \cite[\S2.2]{KL1}.

\begin{Lemma} \label{LNHecke} 
For any $x\in R_{m\al_i}$ there exists $y\in \Z[y_1,\dots,y_m]$ such that $1_{i^{(m)}}x=\psi_{w_0}y$. 
\end{Lemma}
\begin{proof}
By Theorem~\ref{TBasis}, we can write $(\prod_{s=1}^m y_s^{s-1})x=\sum_{w\in\Si_m}\psi_w y(w)$ for some $y(w)\in\Z[y_1,\dots,y_m]$. So $1_{i^{(m)}}x=\psi_{w_0}(\prod_{s=1}^m y_s^{s-1})x=\sum_{w\in\Si_m}\psi_{w_0}\psi_w y(w)=\psi_{w_0}y(1)$. 
\end{proof}

Let $\theta\in Q_+$. We define $I^{\theta}_{\di}$ to be the set of all expressions of the form
$(i_1^{(m_1)}, \ldots, i_r^{(m_r)})$ with
$m_1,\dots,m_r\in \Z_{\ge 0}$, $i_1,\dots,i_r\in I$
 and $m_1 \al_{i_1} + \cdots + m_r \al_{i_r} = \theta$. We refer to such expressions as {\em divided power words}. Analogously to the words, for $\kappa\in I=\Z/e\Z$ and a divided power word $\bi=(i_1^{(m_1)}, \ldots, i_r^{(m_r)})$, we define the divided power word $\bi^{+\kappa}:=((i_1+\kappa)^{(m_1)}, \ldots, (i_r+\kappa)^{(m_r)})$. 
 We identify $I^\theta$ with the subset of $I^\theta_\di$ which consists of all expressions as above with all $m_k=1$. 
 We use the same notation for concatenation of divided power words as for concatenation of words.

Fix $\bi = (i^{(m_1)},\ldots, i^{(m_r)}) \in I^{\theta}_{\di}$. We have the {\em divided power idempotent} 
\[
1_{\bi}:=\iota_{m_1\al_{i_1},\dots,m_r\al_{i_r}}(1_{i_1^{(m_1)}}\otimes\dots\otimes 1_{i_r^{(m_r)}}) \in R_{\theta}.
\]
Define $\bi! := [m_1]^! \cdots [m_r]^!$ and
\begin{equation}\label{EAngles}
\langle \bi \rangle := \sum_{k=1}^r m_k (m_k-1)/2.
\end{equation} 
Set 
\begin{equation}\label{EHatBi}
\hat\bi := (i_1,\ldots,i_1,\ldots,i_r,\ldots,i_r)\in I^{\theta},
\end{equation}
with $i_k$ repeated $m_k$ times. 
Note that $1_{\vphantom{\hat bi}\bi} 1_{\hat\bi} = 1_{\hat\bi} 
1_{\vphantom{\hat bi}\bi}
=1_{\vphantom{\hat bi}\bi}$.

\begin{Lemma}\label{LFact} {\rm \cite[\S2.5]{KL1}} 
 Let $U$ (resp.~$W$) be a left (resp.~right)
 $R_{\theta}$-module, free of finite rank as a $\Z$-module. 
 For  $\bi\in I^{\theta}_{\di}$, we have 
 \[
  \DIM (1_{\hat\bi} U) = \bi! q^{\langle \bi\rangle} \DIM(1_{\bi} U)
  \quad \text{and} \quad 
  \DIM (W 1_{\hat\bi})= \bi! q^{-\langle \bi\rangle} \DIM(W 1_{\bi}). 
 \]
\end{Lemma}

\subsection{Semicuspidal  modules}
\label{SSSC}

Let $d,f\in \Z_{\ge 0}$. 
A word $\bi\in I^{d\de}$ is called {\em separated}\, if whenever $\bi=\bj\bk$ for $\bj\in I^\theta$ and $\bk\in I^\eta$, it follows that $\theta$ is a sum of positive roots in $\Phi_+^{\preceq\de}$ and $\eta$ is a sum of positive roots in $\Phi_+^{\succeq\de}$. 
We denote by $I^{d\de}_{\sep}$ the set of all  separated words in $I^{d\de}$. An $R_{d\de}$-module is {\em (imaginary) semicuspidal} if all of  its words are separated. 
Note that a shuffle of separated words is separated, so:

\begin{Lemma} \label{LIndCusp} 
If $U\in\mod{R_{d\de}}$ and $W\in\mod{R_{f\de}}$ are semicuspidal modules, then $U\circ W\in\mod{R_{(d+f)\de}}$ is semicuspidal. 
\end{Lemma}

Set
$
1_{\nsep}
:=\sum_{\bi\in I^{d\de}\setminus I^{d\de}_{\sep}} 1_{\bi}. 
$  
The {\em (imaginary) semicuspidal algebra} is defined as 
\begin{equation}\label{ESCA}
\hat C_{d\de}:=R_{d\de}/R_{d\de} 1_{\nsep}R_{d\de}.
\end{equation}
The category of finitely generated semicuspidal $R_{d\de}$-modules is equivalent to the category $\mod{\hat C_{d\de}}$. 
A word $\bi\in I^{d\de}$ is called {\em semicuspidal} if the idempotent $1_\bi$ is non-zero in $\hat C_{d\de}$. Denote by $I^{d\de}_{\scusp}$ the set of all semicuspidal words. Then, setting $
1_{\noncusp}
:=\sum_{\bi\in I^{d\de}\setminus I^{d\de}_{\scusp}} 1_{\bi},
$ 
we have $\hat C_{d\de}\cong R_{d\de}/R_{d\de} 1_{\noncusp}R_{d\de}$. By definition, we always have  $I^{d\de}_{\scusp}\subseteq I^{d\de}_{\sep}$, but this containment may be strict, see Example~\ref{ExSep} below.

Everything in this subsection so far makes sense over any ground ring. In particular the notion of a semicuspidal module over $R_{d\de,\F}$ is defined for any field $\F$. We now explain the classification of the semicuspidal irreducible $R_{d\de,\F}$-modules for an arbitrary field $\F$.

We begin with the case $d=1$, in which case the semicuspidal irreducible $R_{d\de,\F}$-modules are parametrized by the elements of 
$
J=\{1,\dots,e-1\}=I\setminus\{0\}.
$
More precisely, let $j\in J$. We denote by $I^{\de,j}$ the set of all words in $I^\de$ of the form $0\bk j$ where $\bk$ is an arbitrary shuffle of the words $(1,2,\dots, j-1)$ and  $(e-1,e-2,\dots,j+1)$. 
Let $L_{\de,j}$ be the graded $\Z$-module with basis $\{v_\bi\mid\bi\in I^{\de,j}\}$ where all basis elements have degree $0$. 
By \cite[Theorem 3.4]{KRhomog}, there is a unique structure of a graded $R_\de$-module on $L_{\de,j}$ such that 
\begin{equation}\label{EMinuscule}
1_\bj v_\bi=\de_{\bi,\bj}v_\bi,\ y_r v_{\bi}=0,\ \psi_r v_{\bi}=
\left\{
\begin{array}{ll}
v_{s_r \bi} &\hbox{if $s_r\bi\in I^{\de,j}$,}\\
0 &\hbox{if $s_r\bi\not\in I^{\de,j}$}
\end{array}
\right.
\end{equation}
for all admissible $\bi,\bj,r$. 
All the words in $I^{\de,j}$ are separated, so the module $L_{\de,j}$ is semicuspidal, which implies that $I^{\de,j}\subseteq I^\de_{\scusp}$. 

For example, 
$$
I^{\de,1}=\{(0,e-1,e-2,\dots,1)\}\quad\text{and}\quad 
I^{\de,e-1}=\{(0,1,\dots, e-1)\},
$$
so $L_{\de,1}$ and $L_{\de,e-1}$ have $\Z$-rank $1$. On the other hand,  for $e\geq 3$, the module $L_{\de,e-2}$ has $\Z$-rank $e-2$, since
$$
I^{\de,e-2}=\{(0,1,\dots,r, e-1,r+1,r+2,\dots,e-2)\mid 0\leq r< e-2\}. 
$$

For a composition $\bd=(d_1,\dots,d_{e-1})\in\La(e-1,d)$, consider the semicuspidal $R_{d\de}$-module 
$
V^\bd:=L_{\de,1}^{\circ d_1}\circ\dots\circ L_{\de,e-1}^{\circ d_{e-1}}.
$

\begin{Theorem} \label{TVAl} 
Let $\F$ be an arbitrary field. 
There is an assignment $\ula\mapsto L(\ula)$ which maps every element $\ula\in\Par^J(d)$ to a semicuspidal irreducible $R_{d\de,\F}$-module $L(\ula)$ such that
\begin{enumerate}
\item[{\rm (i)}] $\{L(\ula)\mid \ula \in \Par^J(d)\}$ is a complete and irredundant set of irreducible semicuspidal $R_{d\de,\F}$-modules;
\item[{\rm (ii)}] Let $\bd=(d_1,\dots,d_{e-1})\in\La(e-1,d)$ and $$\Par^J(\bd)=\{\ula=(\la^{(1)},\dots,\la^{(e-1)})\in\Par^J(d)\mid |\la^{(j)}|=d_j \ \text{for all $j\in J$}\}.$$ 
Then $\{L(\ula)\mid \ula\in \Par^J(\bd)\}$ is the set of composition factors of $V^{\bd}_\F$. 
\end{enumerate}
\end{Theorem}
\begin{proof}
This is essentially contained in \cite{Kcusp} and \cite{KM}, but we  provide some details for the reader's convenience. 
In this proof, we drop the subscript $\F$ from our notation. Fix $n\in\Z_{\geq d}$. Let $j\in J$, $m\in\Z_{\geq 0}$, and $\nu\in\La(n,m)$. In \cite[\S1.4]{KM}, certain submodules $Z_j^\nu\subseteq L_{\de,j}^{\circ m}$ are constructed. Let $Z_j:=\bigoplus_{\nu\in\La(n,m)}Z_j^\nu$ and ${\mathscr S}_{m,j}:=R_{d\de}/\Ann_{R_{d\de}}(Z_j)$. In \cite[Theorems 4 and 6]{KM} a complete and irredundant family $\{L_j(\la)\mid \la\in\Par(m)\}$ of irreducible ${\mathscr S}_{m,j}$-modules is constructed and it is proved that $Z_j$ is a projective generator for ${\mathscr S}_{m,j}$, hence every $L_j(\la)$ appears as a composition factor of $Z_j$. But $Z_j$ is a direct sum of submodules of $L_{\de,j}^{\circ m}$ and one of the summands is $L_{\de,j}^{\circ m}$ itself. So every $L_j(\la)$ appears as a composition factor of $L_{\de,j}^{\circ m}$. 

Now, for $\ula\in\Par^J(\bd)$, we consider the $R_{d\de}$-module  $L(\ula):=L_1(\la^{(1)})\circ\dots\circ L_{e-1}(\la^{(e-1)})$. By 
\cite[Theorem 5.10]{Kcusp}, this module is semicuspidal and irreducible, and $\{L(\ula)\mid \ula \in \Par^J(d)\}$ is a complete and irredundant set of irreducible semicuspidal $R_{d\de,\F}$-modules, proving (i). Now (ii) follows from the description of the composition factors of each $L_{\de,j}^{\circ d_j}$ in the previous paragraph. 
\end{proof}

\begin{Corollary} \label{CSepSC} 
The set $I^{d\de}_{\scusp}$ is exactly the set of all shuffles of  words $\bi^{(1)},\dots,\bi^{(d)}$ such that each 
$\bi^{(a)}\in\bigsqcup_{j\in J}I^{\de,j}$. 
\end{Corollary}
\begin{proof}
If $\bi$ is a shuffle of words $\bi^{(1)},\dots,\bi^{(d)}$ such that $\bi^{(a)}\in I^{\de,j_a}$ for $a=1,\dots,d$, then $\bi$ is a word of the semicuspidal module $L_{\de,j_1}\circ\dots\circ L_{\de,j_d}$, so $\bi\in I^{d\de}_{\scusp}$. Conversely, let $\bi\in I^{d\de}_{\scusp}$. By definition, $1_\bi$ is non-zero in $\hat C_{d\de}$. Since $1_\bi$ is an idempotent, it follows that $1_{\bi,\F}:=1_\bi\otimes 1_{\F}$ is non-zero in $\hat C_{d\de,\F}$ for some field $\F$. Hence there is an irreducible semicuspidal $R_{d\de,\F}$-module $L$ such that $1_{\bi,\F} L\neq 0$. By Theorem~\ref{TVAl}, the word $\bi$ is a shuffle of words $\bi^{(1)},\dots,\bi^{(d)}$ such that each $\bi^{(a)}\in\bigsqcup_{j\in J}I^{\de,j}$.
\end{proof}

\begin{Example} \label{ExSep} 
{\rm 
Let $e=5$ and $d=2$. Then the word $0012342341$ is in $I^{d\de}_\sep$,  but is not in  $I^{d\de}_\scusp$ by Corollary~\ref{CSepSC}.
}
\end{Example}

\subsection{Induction and restriction of semicuspidal modules}\label{SSParabolicSC}
Throughout the subsection we fix $d\in \Z_{\ge 0}$, $n\in\Z_{>0}$ and $\la=(\la_1,\dots,\la_n)\in\La(n,d)$.
Denote
$$R_{\la\de}:=R_{\la_1\de,\dots,\la_n\de}\subseteq R_{d\de}.$$
Let $1_{\la\de}$ denote the identity element of $R_{\la\de}$.
Define the semicuspidal parabolic subalgebra
$$\hat C_{\la\de}\subseteq \hat C_{d\de}$$
to be the image of $R_{\la\de}$ under the natural projection $R_{d\de}\onto \hat C_{d\de}$.
Whereas the parabolic subalgebra $R_{\la\de}$ has been identified with $R_{\la_1\de}\otimes\dots\otimes R_{\la_n\de}$ via the embedding $\iota_{\la_1\de,\dots,\la_n\de}$, it is not clear a priori that $\hat C_{\la\de}\cong \hat C_{\la_1\de}\otimes\dots\otimes \hat C_{\la_n\de}$. This will be proved in Lemma~\ref{L030216}.

We call an $R_{\la_1\de}\otimes\dots\otimes R_{\la_n\de}$-module $W$ semicuspidal if $(1_{\bi^{(1)}}\otimes\dots\otimes 1_{\bi^{(n)}})W=0$ whenever $\bi^{(1)},\dots,\bi^{(n)}$ are not all separated. This is equivalent to the property that $W$ factors through the natural quotient $\hat C_{\la_1\de}\otimes\dots\otimes \hat C_{\la_n\de}$ of $R_{\la_1\de}\otimes\dots\otimes R_{\la_n\de}$.

\begin{Lemma} \label{LResCusp}
We have:
\begin{enumerate}
\item[{\rm (i)}] If $W$ is a semicuspidal $R_{d\de}$-module, then $\Res_{\la_1\de,\dots,\la_n\de}W$ is a semicuspidal $R_{\la_1\de}\otimes\dots\otimes R_{\la_n\de}$-module.
\item[{\rm (ii)}] If $\bi^{(1)}\in I^{\la_1\de},\dots,\bi^{(n)}\in I^{\la_n\de}$ and $\bi^{(1)}\dots\bi^{(n)}\in I^{d\de}_\scusp$, then we have that $\bi^{(1)}\in I^{\la_1\de}_\scusp,\dots,\bi^{(n)}\in I^{\la_n\de}_\scusp$.
\end{enumerate}
\end{Lemma}
\begin{proof}
This is known and can be proved combinatorially using Corollary~\ref{CSepSC}. We sketch a representation-theoretic proof for the reader's convenience. Clearly (i) and (ii) are equivalent, and hence it suffices to prove (i) with scalars extended to $\C$ in the case where  $W$ is irreducible. This follows for example from \cite[Theorem 14.6]{McNAff}.
\end{proof}

\begin{Lemma} \label{LIndCuspProj} 
If $\bi^{(1)}\in I^{\la_1\de}_\scusp,\dots,\bi^{(n)}\in I^{\la_n\de}_\scusp$, then there is an isomorphism of $R_{d\de}$-modules
\begin{align*}
\hat C_{d\de}1_{\bi^{(1)}\dots\bi^{(n)}}&\iso \hat C_{\la_1\de}1_{\bi^{(1)}} \circ\dots\circ \hat C_{\la_n\de}1_{\bi^{(n)}},
\\
1_{\bi^{(1)}\dots\bi^{(n)}}
&\mapsto
1_{\la_1\de,\dots,\la_n\de}\otimes 1_{\bi^{(1)}}\otimes \dots\otimes 1_{\bi^{(n)}}.\end{align*}
\end{Lemma}
\begin{proof}
Since $\hat C_{\la_1\de}1_{\bi^{(1)}} \circ\dots\circ \hat C_{\la_n\de}1_{\bi^{(n)}}$ is semicuspidal, we can consider it as a $\hat C_{d\de}$-module. So there exists a homomorphism as in the lemma. To construct the inverse homomorphism, use adjointness of induction and restriction together with Lemma~\ref{LResCusp}(i).
\end{proof}

\begin{Lemma}\label{L030216}
The natural map
$
R_{\la_1\de}\otimes\dots\otimes R_{\la_n\de}\into R_{d\de}\onto \hat C_{d\de}
$
factors through $\hat C_{\la_1\de}\otimes\dots\otimes \hat C_{\la_n\de}$ and induces an isomorphism
$\hat C_{\la_1\de}\otimes\dots\otimes \hat C_{\la_n\de}\iso \hat C_{\la\de}$. Moreover, $\hat C_{d\de} 1_{\la\de}$ is a free right $\hat C_{\la\de}$-module with basis
$\{ \psi_w \mid w\in \D^{e\la}\}$.
\end{Lemma}
\begin{proof}
That the map factors through $\hat C_{\la_1\de}\otimes\dots\otimes \hat C_{\la_n\de}$ follows from Lemma~\ref{LResCusp}.
For the remaining claims, let us consider the $R_{d\de}$-module
$W:=\hat C_{\la_1\de}\circ\dots\circ \hat C_{\la_n\de}$. By Lemma~\ref{LIndCusp}, the module $W$ factors through $\hat C_{d\de}$. On the other hand by the Basis Theorem~\ref{TBasis}
 for $R_{d\de}$, we can decompose $W=\bigoplus_{w\in \D^{e\la}} \psi_w 1_{\la\de}\otimes \hat C_{\la_1\de}\otimes\dots\otimes \hat C_{\la_n\de}$ as a $\Z$-module, with each summand being naturally isomorphic to $\hat C_{\la_1\de}\otimes\dots\otimes \hat C_{\la_n\de}$ as a $\Z$-module. The lemma follows.
\end{proof}

In view of the lemma we identify $\hat C_{\la_1\de}\otimes\dots\otimes \hat C_{\la_n\de}$ with $\hat C_{\la\de}$. Then:

\begin{Corollary} 
Suppose that for each $r=1,\dots,n$ we have a $\hat C_{\la_r\de}$-module $W_r$. Then there is a natural isomorphism of semicuspidal $R_{d\de}$-modules
\begin{align*}
W_1\circ\dots\circ W_n&\iso \hat C_{d\de}1_{\la\de}\otimes_{\hat C_{\la\de}} (W_1\boxtimes\dots\boxtimes W_n), \\
u1_{\la\de}\otimes w_1\otimes\dots\otimes w_n&\mapsto
\bar u1_{\la\de}\otimes w_1\otimes\dots\otimes w_n,
\end{align*}
where $\bar u\in\hat C_{d\de}$ is the image of $u\in R_{d\de}$ under the natural projection $R_{d\de}\onto \hat C_{d\de}$.
\end{Corollary}

From now on we identify the induced modules as in the corollary.

\section{Abaci, tableaux and RoCK blocks} 
\subsection{Abaci}\label{SSAb}

We will use the abacus notation for partitions, see~\cite[Section 2.7]{JK}. Recall that we have fixed a number $e\in\Z_{\geq 2}$ and $I=\Z/e\Z$. When convenient we identify $I$ with the subset 
$\{0,1,\dots,e-1\}\subset \Z$. We define the {\em abacus} $\Ab^e:=\N \times I$. For $i\in I$, the subset ${\mathsf R}_i:=\N \times \{ i\}\subset \Ab^e$ is referred to as the ($i$th) {\em runner} of $\Ab^e$.

Let $\la$ be a partition, and fix an integer $N\ge \ell(\la)$, so that we can write $\la=(\la_1,\dots,\la_N)$. 
Let 
\begin{equation}\label{E260216}
\Ab_N(\la):=\{\la_k+N-k\mid k=1,\dots,N\}\subset \Z_{\geq 0}.
\end{equation}
The {\em abacus display of $\la$} is
$$
\Ab_N^{e} (\la):=\{(t,i)\in \Ab^e\mid e t+i \in \Ab_N(\la)\}. 
$$ 
The elements of $\Ab_N^{e} (\la)$ are called the \emph{beads} of $\Ab_N^{e} (\la)$, and the elements of $\Ab^e \sm \Ab_N^e (\la)$ are called the {\em non-beads} of $\Ab_N^e (\la)$. 

We have the total order $<$ on $\Ab^e$ defined by the condition that 
$(t,i)<(q,j)$ if and only if $et+i<eq+j$. If $(t,i)<(q,j)$, we say that $(t,i)$ 
{\em precedes} $(q,j)$ and $(q,j)$ {\em succeeds} $(t,i)$. 
For any $r\in \Z_{>0}$, we say that a bead $(t,i)$ of $\Ab_N^e (\la)$ is
the {\em bead with number $r$} in $\Ab_N^e (\la)$ if exactly $r-1$ beads of $\Ab_N^e(\la)$ succeed $(t,i)$, and we say that a non-bead $(t,i)$
of $\Ab_N^e (\la)$ is the {\em non-bead with number $r$} in 
$\Ab_N^e(\la)$ if exactly $r-1$ non-beads of $\Ab_N^e (\la)$ precede 
$(t,i)$. 

It is easy to see that the bead $(t,i)$ with number $r$ of $\Ab_N^e(\la)$ satisfies $et+i=N+\la_r-r$. Moreover, if $(\la'_1,\la'_2,\ldots)$ is the conjugate partition to $\la$, then the non-bead $(t,i)$ with number $s$ of $\Ab_N^e(\la)$ satisfies 
$et+i = N-\la'_s +s-1$. 
Using these observations, it is easy to prove the following well-known fact:

\begin{Lemma} \label{LPrecSucc}
 Let $\la\in\Par$ and $(r,s)\in \Nodes$. Then $(r,s)\in \Y\la$ if and only if the bead with number $r$ succeeds the non-bead with number $s$ in $\Ab^e_N (\la)$. 
\end{Lemma}

For $\la\in\Par$, we write 
$b_i (\la): = |\Ab_N^e (\la) \cap {\mathsf R}_i|$ for $i\in I$. The {\em $e$-core} of $\la$ is the partition $\operatorname{core}(\la)$ defined by 
$$
\Ab^e_N(\operatorname{core}(\la))=\{(t,i)\in \Ab^e\mid i\in I,\ 0\leq t<b_i(\la)\}. 
$$
Recall the notation (\ref{E260216}). The \emph{$e$-quotient} 
of $\la$ is defined as the multipartition $\operatorname{quot}_N(\la)=(\la^{(i)})_{i\in I}\in\Par^I$ such that for every $i\in I$, the partition $\la^{(i)}$ is determined from  
$\Ab_{b_i (\la)} (\la^{(i)})=\Ab_N^{e} (\la) \cap {\mathsf R}_i$, where we have identified ${\mathsf R}_i$ with $\Z_{\geq 0}$. 
The $e$-quotient of $\la$ depends on the residue of $N$ modulo $e$
and changes by a `cyclic permutation' of the components $\la^{(i)}$ when this residue changes. So the {\em $e$-weight} of $\la$, defined as 
$
\wt(\la):=|\operatorname{quot}_N(\la)|,
$
does not depend on $N$. 

Note that $\la=\operatorname{core}(\la)$ if and only $\operatorname{quot}_N(\la)=\varnothing$, in which case $\la$ is said to be an {\em $e$-core}. 
For any $e$-core $\rho$ and $d\in \Z_{\ge 0}$, we set
\[
\Par_\rho:= \{ \la\in \Par \mid \core (\la) = \rho\}, \quad 
\Par_{\rho,d}:= \{ \la \in \Par_\rho \mid \wt (\la) = d\}.  
\]
The following is easy to check and well known: 

\begin{Lemma} \label{LQuotUn} 
The map $\la\mapsto \quot(\la)$ is a bijection from $\Par_{\rho,d}$ to $\Par^I(d)$. 
\end{Lemma}

The {\em ($e$-)residue} of a node $(r,s)\in \Nodes$ is   
$\res(r,s) := s-r +e \Z\in I=\Z/e\Z.$ 
For $i\in I$, we say that $(r,s)$ is an {\em $i$-node} if its residue is $i$. 
For $\la\in\Par$, we define
$$
\cont(\la):= \sum_{u\in \Y\la} \al_{\res(u)} \in Q_+.
$$

\begin{Lemma} {\rm \cite[Theorem 2.7.41]{JK}} \label{LCoreCont} 
Let $\rho$ be an $e$-core, $d\in\Z_{\geq 0}$, and $\la\in\Par$. Then  
$\cont(\la)=\cont(\rho)+d\de$ if and only if $\la\in \Par_{\rho,d}$. 
\end{Lemma}

\subsection{Tableaux}
Let $\nu$ be a partition. 
A node $u\in {\mathsf N}$ is called an \emph{addable} node for  $\nu$ if 
$u\notin \Y\nu$ and $\Y\nu\cup \{u\}$ is the Young diagram of a partition, 
and $u$ is called a \emph{removable} 
node of $\nu$ if $u\in \Y\nu$ and $\Y\nu \sm \{u\}$ 
is the Young diagram of a partition. For $i\in I$, we denote by $\Add(\nu,i)$ (resp.\ $\Rem(\nu,i)$) the set of all addable (resp.\ removable) $i$-nodes for $\nu$. 
We say that a node $(r,s)$ 
is \emph{above} a node $(r',s')$ if $r<r'$. Given a node $v\in\Nodes$ and a finite subset $U\subset \Nodes$, denote by $a(v,U)$ the number of elements of $U$ which are above $v$.  

Let $i\in I$ and $U$ be a set of removable $i$-nodes of $\nu$.  Define 
\[ d_U (\nu) = \sum_{v\in\Add(\nu,i)}a(v,U)-\sum_{v\in\Rem(\nu,i)\setminus U}a(v,U).
 \] 

Let $\la \sm \mu$ be a skew partition, 
and $\theta=\cont(\la\sm \mu) := \sum_{u\in \Y\la \sm \Y\mu} \al_{\res(u)} \in Q_+$. Fix $\bi=(i_1^{(m_1)},\dots,i_r^{(m_r)})\in I^\theta_\di$. An \emph{$\bi$-standard tableau} of shape $\la\sm\mu$ is a map 
$\t\colon \Y\la \sm \Y\mu \to \{1,\ldots,r\}$ such that 
\begin{enumerate}
\item[{\rm (i)}] $\t (u) < \t(u')$ whenever $u,u'\in \Y\la\sm\Y\mu$ 
and $u<u'$;
\item[{\rm (ii)}] for all $k=1,\dots,r$ and $u\in \t^{-1} (k)$, we have $\res u=i_k$;
\item[{\rm (iii)}] for all $k=1,\dots,r$, we have $|\t^{-1} (k) | = m_k$. 
\end{enumerate}
We denote the set of all $\bi$-standard tableaux of shape $\la\sm\mu$ by $\Std(\la\sm\mu, \bi)$. 
If $\t \in \Std(\la\sm \mu,\bi)$, we define 
\[
 \deg(\t) = \sum_{k=1}^r d_{\t^{-1}(k)} (\t^{-1} ([1,k]) \cup \Y\mu ).
\]
Note that $\deg(\t)$ depends on $\la$ and $\mu$, not just on the set $\Y\la\sm\Y\mu$. If  $\bi\in I^\theta$ and $\mu=\varnothing$, then the notion of an $\bi$-standard tableau is the same as the usual notion of  a standard tableau with residue sequence $\bi$ as in \cite[\S3.2]{BKW}, and the notion of the degree agrees with the one from \cite[\S3.5]{BKW}. 
If $\bi \in I^{\eta}_{\di}$ for some $\eta \ne \theta$, then we set 
$\Std(\la\sm \mu, \bi):=\varnothing$.
We denote $\Std(\la\sm\mu):=\bigsqcup_{\bi\in I^{\cont(\la\sm\mu)}}\Std(\la\sm\mu, \bi)$. 

Let $\bi=(i_1^{(m_1)},\dots,i_r^{(m_r)})\in I^\theta_\di$ and $\hat\bi\in I^\theta$ be as in (\ref{EHatBi}). Given $\t\in\Std(\la\sm\mu, \bi)$, a  tableau $\s\in \Std(\la\sm\mu, \hat\bi)$ is called a refinement of $\t$ if 
$$
\t^{-1}(k)=\s^{-1}\big([m_1+\dots+m_{k-1}+1, m_1+\dots+m_k]\big)
$$
for all $k=1,\dots,r$. Let  $\hat\t\subseteq \Std(\la\sm\mu, \hat\bi)$ denote the set of all refinements of $\t$.

\begin{Lemma}\label{LTabFact} 
For any $\t\in\Std(\la\sm\mu, \bi)$, we have $\sum_{\s\in\hat\t}q^{\deg(\s)}=\bi!q^{\deg(\t)}$. 
\end{Lemma}
\begin{proof}
The lemma is easily reduced to the case $r=1$. In that case, let $\s\in\hat t$ be the tableau such that for $u,v\in\Y\la\sm\Y\mu$ 
the node $u$ is above $v$ if and only if $\s(u)<\s(v)$, in other words we assign the numbers $1,\dots,m:=m_1$ to the nodes of $\la\sm\mu$ from top to bottom. Then $\deg(\s)=\deg(\t)+m(m-1)/2$. We have 
$\hat t=\{w\s\mid w\in\Si_m\}$, where $w\s$ is the tableaux defined by $(w\s)(u)=w(\s(u))$. In view of \cite[Proposition 3.13]{BKW}, we have $\deg(w\s)=\deg(\s)-2\ell(w)$, where $\ell(w)$ is the length of $w\in\Si_d$. So
$$
\sum_{\s\in\hat\t}q^{\deg(\s)}=q^{\deg(\t)+m(m-1)/2}\sum_{w\in \Si_m}q^{-2\ell(w)}=[m]^!q^{\deg(\t)},
$$
where the last equality comes from the well-known formula for the Poincar\'e polynomial of the symmetric group \cite[\S3.15]{Hu}. 
\end{proof}

\subsection{Dimensions and core algebras}\label{SSDimAndCore}
Recall the notation (\ref{EAngles}). The following is a variation of a known result:

\begin{Theorem}\label{TDim}  
 For any $\theta\in Q_+$ and $\bi,\bj\in I^\theta_{\di}$, the $\Z$-module $1_{\bi} R^{\La_0}_{\theta} 1_\bj$ is free of graded rank
\[
 \qdim \left( 1_{\bi} R^{\La_0}_{\theta} 1_\bj \right) = 
 q^{\langle \bj\rangle - \langle \bi \rangle}
\sum_{\substack{\mu\in \Par \\ 
\s \in \Std(\mu,\bi) \\ \t\in \Std(\mu,\bj) 
}}
q^{\deg(\s)+\deg(\t)}.
\]
In particular, 
the idempotent 
$1_{\bi}$ is non-zero in $R^{\La_0}_{\theta}$ 
if and only if $\Std(\mu,\bi)\ne \vn$ 
for some $\mu\in\Par$.
\end{Theorem}
\begin{proof}
The freeness statement follows from Theorem~\ref{TBasisCyc}. 
Extending scalars to $\C$ and using \cite[Theorem 4.20]{BKllt} yields the graded rank formula in the case when $\bi,\bj \in I^{\theta}$, and the general case then follows from Lemmas~\ref{LFact} and~\ref{LTabFact}. 
\end{proof}

Recall the notation $\ell(A)$ for an algebra $A$ from~\S\ref{SSAlg}.

\begin{Theorem}\label{TNSimples}
Let $\k$ be a field, $\rho$ be an $e$-core and $d\in \Z_{\ge 0}$. Then
\[
\ell (R^{\La_0}_{\cont(\rho)+d\de,\k}) = |\Par^J (d)|.
\]
\end{Theorem}

\begin{proof}
By~\cite[Theorem 6.2]{KK} or~\cite[Theorem 7.5]{LV}, 
the number 
$\ell(R^{\La_0}_{\cont(\rho)+d\de,\k})$ 
is equal
to the dimension of the weight space 
$V(\La_0)_{\La_0-\cont(\rho)-d\de}$ for the integrable highest weight module $V(\La_0)$ over the Kac-Moody algebra $\g$ of type $A_{e-1}^{(1)}$. It is well known that this dimension is equal to $|\Par^J (d)|$, see e.g. \cite[(13.11.5)]{Kac} or \cite[Sections 4,5]{LLT}. 
\end{proof}

Let $\rho$ be an $e$-core. 
We pick an extremal word $(i_1^{a_1},\dots,i_r^{a_r})\in I^{\cont(\rho)}$ for the left regular module $R_{\cont(\rho)}^{\La_0}$, see \cite[\S2.8]{Kcusp}. In particular, $i_k\neq i_{k+1}$ for $1\leq k<r$. Let $\bi=(i_1^{(a_1)},\dots,i_r^{(a_r)})\in I^{\cont(\rho)}_\di$. 

\begin{Lemma} \label{LMatrix} 
Let $\rho$ be an $e$-core and $\bi\in  I^{\cont(\rho)}_\di$ be chosen as above. Then there is an isomorphism of graded $\Z$-algebras $R_{\cont(\rho)}^{\La_0}\iso \End_\Z(R_{\cont(\rho)}^{\La_0}1_\bi)$, where $x\in R_{\cont(\rho)}^{\La_0}$ gets mapped to the left multiplication by $x$. 
\end{Lemma}
\begin{proof}
We clearly have a homomorphism $\phi\colon R_{\cont(\rho)}^{\La_0}\to \End_\Z(R_{\cont(\rho)}^{\La_0}1_\bi)$ as in the statement. In view of Theorem~\ref{TBasisCyc}, to prove that $\phi$ is an isomorphism, it suffices to prove its scalar extension $\phi_\k$ is an isomorphism for any algebraically closed field $\k$. By Theorem~\ref{TNSimples}, the algebra $R^{\La_0}_{\cont(\rho),\k}$ has only one irreducible module $L$ up to isomorphism and degree shift. Considering the composition series of the left regular module over $R^{\La_0}_{\cont(\rho),\k}$, we see that $\bi$ is an extremal weight for $L$, hence by \cite[Lemma 2.8]{Kcusp}, the space $1_\bi L$ is $1$-dimensional. It follows that $\Hom_{R_{\cont(\rho),\k}^{\La_0}}(R_{\cont(\rho),\k}^{\La_0}1_\bi,L)\cong 1_\bi L$ is $1$-dimensional, so $R_{\cont(\rho),\k}^{\La_0}1_\bi$ is the projective cover of $L$. We claim that in fact $R_{\cont(\rho),\k}^{\La_0}1_\bi\cong L$. This is known for $\k=\C$ since $R_{\cont(\rho),\C}^{\La_0}$ is a simple algebra: indeed, by \cite{BK} it is a defect zero block of an Iwahori-Hecke algebra at an $e$th root of unity. Hence $\Hom_{R_{\cont(\rho),\C}^{\La_0}}(R_{\cont(\rho),\C}^{\La_0}1_\bi,R_{\cont(\rho),\C}^{\La_0}1_\bi)\cong 1_\bi R_{\cont(\rho),\C}^{\La_0} 1_\bi$ is $1$-dimensional. This proves that $1_\bi R_{\cont(\rho)}^{\La_0} 1_\bi$ has rank $1$ as a $\Z$-module, whence $1_\bi R_{\cont(\rho),\k}^{\La_0} 1_\bi\cong \Hom_{R_{\cont(\rho),\k}^{\La_0}}(R_{\cont(\rho),\k}^{\La_0} 1_\bi,R_{\cont(\rho),\k}^{\La_0} 1_\bi)$ has dimension $1$. Hence, $R_{\cont(\rho),\k}^{\La_0}1_\bi\cong L$. We deduce that $R_{\cont(\rho),\k}^{\La_0}$ is a simple algebra and $\phi_\k$ is an isomorphism. 
\end{proof}

Recall the map $\zeta_{\theta,\eta}$ from (\ref{EZetaHom}). 

\begin{Lemma} \label{LZeta}
If $\rho$ is an $e$-core and $d\in\Z_{\geq 0}$, then the map 
$$\zeta_{\cont(\rho),d\de}\colon R_{\cont(\rho)}^{\La_0}\to  1_{\cont(\rho),d\de}R_{\cont(\rho)+d\de}^{\La_0}1_{\cont(\rho),d\de}$$ is injective. 
\end{Lemma}
\begin{proof}
By Theorem~\ref{TBasisCyc}, it suffices to prove that the scalar extension of the map to $\C$ is injective. By Lemma~\ref{LMatrix}, $R_{\cont(\rho),\C}^{\La_0}$ is a simple algebra, so it is enough to show that $1_{\cont(\rho),d\de}R_{\cont(\rho)+d\de}^{\La_0}1_{\cont(\rho),d\de}\neq 0$. The last fact follows easily from Theorem~\ref{TDim}.
\end{proof}

\subsection{RoCK blocks} \label{SSRoCK}
 Let $\rho$ be an $e$-core and $d\in\Z_{\geq 1}$.  
Following \cite[Definition 52]{Turner}, we say that $\rho$ is a \emph{$d$-Rouquier core} if there exists an integer $N\ge \ell (\rho)$ such that for all 
$i=0,\ldots,e-2$, the abacus display 
$\Ab_N^e (\rho)$ has at least $d-1$ more beads on runner $i+1$ than on runner $i$. In this case, 
 $$\kappa:=-N+e\Z \in \Z/e\Z$$ is well-defined and is called the  \emph{residue} of $\rho$. 
 
If $\rho$ is a $d$-Rouquier core, we refer to the cyclotomic KLR algebra $R^{\La_0}_{\cont(\rho)+d\de}$ as a {\em RoCK block}.

\begin{Remark} 
{\rm 
The term {\em RoCK} comes from the names of Rouquier~\cite{RoTh},  Chuang and Kessar \cite{CK}. 
We refer to the algebra $R^{\La_0}_{\cont(\rho)+d\de}$ as a {\em block} since, with notation as in Section~\ref{SIntro}, 
the block $H_{\cont(\rho)+d\de} (q)$ of an Iwahori--Hecke algebra
is isomorphic to the 
$\F$-algebra $R^{\La_0}_{\cont(\rho)+d\de,\F}$, 
see~\cite{BK,R}. Note however that the analogous isomorphism in general does not make sense over $\Z$. Moreover, if $\k$ is a field such that $e=m \cha k $ for some $m\in\Z_{>1}$, the  algebra $R_{\theta,\k}^{\La_0}$ is not in general isomorphic to a block of a Hecke algebra. 
}
\end{Remark} 

We now review and develop some results from \cite[Section 4]{Evseev}. 
Throughout the subsection, we fix $d\in\Z_{>0}$ and a $d$-Rouquier core $\rho$ of residue $\kappa$. We then set 
$$\al := \cont(\rho)+d\de\in Q_+.$$ 
Let 
$$
\Omega\colon R_{d\de}\to R_{\cont(\rho),d\de}^{\La_0},\ x\mapsto \pi_\al\big(\iota_{\cont(\rho),d\de}(1_{\cont(\rho)}\otimes \rot_\kappa(x))\big),
$$
cf.~\eqref{EPi}, \eqref{ERot} and~\eqref{EIota}. 
Note that $\Omega$ is in general a non-unital algebra homomorphism with $\Omega(1)=1_{\cont(\rho),d\de}$.


\begin{Lemma}\label{LOmegaFact} 
Let $\bi\in I^{d\de}$, and $\bj\in I^{\rho}$ be such that  $\Std(\rho,\bj)\neq \vn$. If $1_{\bj(\bi^{+\kappa})}$ is non-zero in $R^{\La_0}_{\cont(\rho)+d\de}$, then $\bi\in I^{d\de}_{\scusp}$. In particular, $\Omega$ factors through $\hat C_{d\de}$. 
\end{Lemma}
\begin{proof}
This follows from \cite[Lemma 4.6]{Evseev} thanks to Theorem~\ref{TDim} and Corollary~\ref{CSepSC}. 
\end{proof}

In view of the lemma, from now on, we will consider $\Om$ as a homomorphism 
\begin{equation}\label{EOm}
\Om: \hat C_{d\de}\to R^{\La_0}_{\cont(\rho),d\de}.
\end{equation}

\begin{Lemma}\label{Lcoreab}
 Let $\si$ be a partition such that $\Y\si\subsetneq \Y\rho$.
 Then the number of nodes of residue $\ka$ 
 in $\Y\rho\sm \Y\si$ is less than $(|\rho|-|\si|)/e$. 
\end{Lemma}

\begin{proof}
In this proof we use abacus displays with $N$ beads, where  $N$ is greater than the number of parts in all the partitions involved and 
 $N+e\Z = -\ka$. Recall from~\S\ref{SSAb} that for $\tau\in\Par$, we denote 
 $b_i(\tau):=|\Ab_N^e (\tau) \cap {\mathsf R}_i|$. For $0\leq l<e$, we denote
 $b_{\geq l}(\tau):=\sum_{i=l}^{e-1}b_i(\tau)$.  
 Recall also the fundamental dominant weights $\La_i$ from \S\ref{SSLie}. Let $0\leq m<e$ be the integer such that $m+e\Z=-\kappa$. 

For any  $\tau\in\Par$, we claim that 
\begin{equation}\label{contab}
e(\La_\ka, \cont(\tau)) - |\tau| 
= \frac{(e-1)N-(e-m)m}{2} -\sum_{l=1}^{e-1} b_{\geq l} (\tau).
\end{equation}
Indeed, it is straightforward to check 
that both sides are $0$ when $\tau=\varnothing$, since 
 $b_0(\varnothing)=\cdots = b_{m-1}(\varnothing) = b_m(\varnothing)+1=\cdots=b_{e-1}(\varnothing)+1$. 
Furthermore, adding a box of residue $i\in I$ to $\tau$ changes both sides by $e-1$ if $i=\ka$ and by $-1$ if $i\ne \ka$ (for the right-hand side, consult~\cite[Lemma 4.2]{Evseev}). The claim is proved. 

Let $l\in \{0,\ldots,e-1\}$ and $b= b_l(\rho)$. 
Suppose for a contradiction that $b_{\geq l} (\si) >  b_{\geq l} (\rho)$. As $\rho $ is a Rouquier core, $\Ab_N^e (\rho)$ contains the rectangle $[0,b-1]\times [ l,e-1]$, whence  
\[
|\Ab^e_N (\si)\cap (\Z_{\ge b} \times [ l,e-1])|
> |\Ab^e_N (\rho) \cap (\Z_{\ge b} \times [ l,e-1])|,
\]
and it follows that 
$
|\Ab_N (\si) \cap \Z_{\ge be}| > |\Ab_N (\rho)\cap \Z_{\ge be}|.
$
This is a contradiction to the hypothesis $\Y\si \subseteq \Y\rho$.
Hence, $b_{\geq l} (\si) \le  b_{\geq l}  (\rho)$ for all $l\in \{0,\ldots,e-1\}$.
Moreover, the inequality must be strict for at least one $l\in \{1,\ldots,e-1\}$, for otherwise we have $b_i (\si)=b_i (\rho)$ for all $i\in I$, and so $\rho=\core(\si)$, contradicting the hypothesis $\Y\si\subsetneq \Y\rho$. 
Hence, using~\eqref{contab}, we deduce that 
$e(\La_\ka, \cont(\si)) - |\si|>e(\La_\ka, \cont(\rho)) - |\rho|$, which implies the lemma. 
\end{proof}

Recall that throughout the subsection $\al=\rho+d\de$ is a RoCK block. 

\begin{Lemma}\label{Ptens1}
We have $1_{\cont(\rho),d\de}R^{\La_0}_{\al}1_{\cont(\rho),d\de}=R^{\La_0}_{\cont(\rho),d\de}$.
\end{Lemma}

\begin{proof} 
By Lemma~\ref{LMackey}, $1_{\cont(\rho),d\de}R_{\al}^{\La_0}1_{\cont(\rho),d\de}$ is generated by $R_{\cont(\rho),d\de}^{\La_0}$ together with the elements $\psi_w$ for $w\in {}^{(|\rho|,de)}\D^{(|\rho|,de|)} \sm \{1\}$. Thus, it will suffice to show that $1_{\cont(\rho),d\de}\psi_w1_{\cont(\rho),d\de}=0$ in $R^{\La_0}_{\al}$ for each such $w$. 
 If not, then $1_{\bj' ((\bi')^{+\ka})} \psi_w 1_{\bj (\bi^{+\ka})}\ne 0$
 for some $\bj,\bj'\in I^{\cont(\rho)}$ such that $\Std(\rho,\bj), \Std(\rho,\bj')$ are non-empty, and some  
 $\bi,\bi'\in I^{d\de}_{\scusp}$, see Theorem~\ref{TDim} and Lemma~\ref{LOmegaFact}. In this case 
 $w(\bj (\bi^{+\ka}))= \bj' ((\bi')^{+\ka})$. 
Moreover, 
 $w=\prod_{t=1}^m (|\rho|-m+t, |\rho|+t)$ for some $m>0$, and therefore
 the last $m$ entries of $\bj'$ are 
 $i_1+\kappa,\ldots,i_m+\kappa$. Since 
 $\bi$ is semicuspidal, the number of entries $\ka$ in the tuple
 $(i_1+\kappa,\ldots,i_m+\kappa)$ is at least $m/e$. But by Lemma~\ref{Lcoreab}, this means that $\Std(\rho,\bj')=\varnothing$, a contradiction. 
\end{proof}

By Lemmas~\ref{LZeta} and \ref{Ptens1}, there is a natural unital algebra embedding 
$$\zeta_{\cont(\rho),d\de}\colon R_{\cont(\rho)}^{\La_0}\to  R_{\cont(\rho),d\de}^{\La_0}=1_{\cont(\rho),d\de}R^{\La_0}_{\al}1_{\cont(\rho),d\de}.$$ 
We always identify $R_{\cont(\rho)}^{\La_0}$ with a subalgebra of $R_{\cont(\rho),d\de}^{\La_0}$ via this embedding. 
We consider the centralizer  of $R^{\La_0}_{\cont(\rho)}$ in $R_{\cont(\rho),d\de}^{\La_0}$:
$$\Cent_{\rho,d} := \Cent_{R_{\cont(\rho),d\de}^{\La_0}} (R^{\La_0}_{\cont(\rho)}).$$

\begin{Lemma} \label{L270116} 
We have an algebra isomorphism 
$R^{\La_0}_{\cont(\rho)} \otimes \Cent_{\rho,d} \iso R_{\cont(\rho),d\de}^{\La_0}$
given by $a\otimes b \mapsto ab$. 
\end{Lemma}
\begin{proof}
This follows from Lemma~\ref{LMatrix} using \cite[Proposition 4.10]{Evseev} (whose proof goes through over $\Z$). 
\end{proof}

Recalling (\ref{EOm}), we denote 
\begin{equation}\label{ECRhoD}
C_{\rho,d}:=\hat C_{d\de}/\ker\Om.
\end{equation}
We have the induced embedding $\bar\Om\colon C_{\rho,d}\to R_{\cont(\rho),d\de}^{\La_0}$. 
By Theorem~\ref{TBasisCyc}, $R^{\La_0}_{\cont(\rho)+d\de}$ is $\Z$-free, so 

\begin{Lemma}\label{LCFree}
The $\Z$-module $C_{\rho,d}$ is free of finite rank. 
\end{Lemma}

\begin{Lemma}\label{Ltens2}
We have $\Cent_{\rho,d}=\bar\Omega(C_{\rho,d})$. 
\end{Lemma}

\begin{proof}
It is clear from the definitions that $\bar\Omega(C_{\rho,d})=\Omega(\hat C_{d\de}) \subseteq \Cent_{\rho,d}$. 
Conversely, let $x\in \Cent_{\rho,d}$. We can write 
$x= \sum_{i=1}^m a_i b_i$ for some $a_1,\ldots,a_m \in R^{\La_0}_{\cont(\rho)}$
and $b_1,\ldots,b_m \in \Omega(\hat C_{d\de})=\bar\Omega(C_{\rho,d})$, and we may assume that $a_1,\ldots,a_m$ are linearly independent with  $a_1=1$. 
By Lemma~\ref{L270116},  
$x=b_1$, so $x\in \bar\Omega(C_{\rho,d})$. 
\end{proof}

In view of Lemma~\ref{L270116}, we deduce:

\begin{Corollary}\label{Ctens} We have:
\begin{enumerate}
\item[{\rm (i)}] The map $\bar\Om\colon C_{\rho,d} \to \Cent_{\rho,d}$ is an algebra isomorphism. 
\item[{\rm (ii)}]
There is an algebra isomorphism $R^{\La_0}_{\cont(\rho)}\otimes C_{\rho,d}\iso R_{\cont(\rho),d\de}^{\La_0}$ given by 
$a\otimes b \mapsto a \bar\Omega(b)$. 
\end{enumerate}
\end{Corollary}

\begin{Remark} \label{RCrazy} 
{\rm 
By Lemma~\ref{LMatrix}, the algebra $R_{\cont(\rho)}^{\La_0}$ is isomorphic to a graded matrix algebra. 
Consider the homogeneous matrix unit $E_{1,1}$ in 
$R_{\cont(\rho)}^{\La_0}\subseteq R^{\La_0}_{\cont(\rho),d\de}$.
By Corollary~\ref{Ctens}, we have 
$C_{\rho,d}\cong E_{1,1} R_{\cont(\rho),d\de}^{\La_0} E_{1,1}$.
 So by Lemma~\ref{Ptens1}, we have 
 $C_{\rho,d}\cong E_{1,1} 1_{\cont(\rho),d\de}R^{\La_0}_{\al}E_{1,1}1_{\cont(\rho),d\de}$. 
 Note that $\mathbf{e}:=E_{1,1} 1_{\cont(\rho),d\de}$ is an idempotent in $R^{\La_0}_{\al}$, so $C_{\rho,d}\cong \mathbf{e}R^{\La_0}_{\al}\mathbf{e}$ is an idempotent truncation of $R^{\La_0}_{\al}$. 

The definition of $C_{\rho,d}$, Lemma~\ref{LMatrix} and Corollary~\ref{Ctens} make sense and can be proved over an arbitrary unital commutative ring $\k$, so the algebra $C_{\rho,d}$ defined over $\k$ is isomorphic to the idempotent truncation $$(\mathbf{e}\otimes 1)R^{\La_0}_{\al,\k}(\mathbf{e}\otimes 1)
\cong (\mathbf{e}R^{\La_0}_{\al}\mathbf{e})\otimes \k\cong C_{\rho,d,\k}.$$ 
}
\end{Remark}

\begin{Corollary} \label{CCSymmetric} 
For any field $\k$, the algebra $C_{\rho,d,\k}$ is symmetric. More precisely, it admits a symmetrizing form of degree $-2d$. 
\end{Corollary}
\begin{proof}
By Remark~\ref{RCrazy}, $C_{\rho,d}$ is an idempotent truncation of $R^{\La_0}_\al$. 
By\cite[Theorem IV.4.1]{SY}, an idempotent truncation of a symmetric algebra is symmetric, with a symmetrizing form obtained by restriction. So it suffices to prove that $R^{\La_0}_{\al,\k}$ is symmetric with a symmetrizing form of degree $-2d$. But this follows from Theorem~\ref{TSVV} and an easy Lie-theoretic computation, see \cite[Lemma~11.1.4]{Kbook}.
\end{proof}

\section{Dimensions}
Throughout the section we fix $d\in\Z_{>0}$, a $d$-Rouquier core $\rho$ of residue $\kappa$, and $n\in\Z_{>0}$. We also fix an integer $N\ge |\rho|+de$ such that $N+e\Z = -\ka$ and assume in this section that {\em all abaci have $N$ beads}, cf.~\S\ref{SSAb}.

The main goal of this section is to compute dimensions of certain idempotent truncations of the algebras $C_{\rho,d}$. The idempotents we use here are the so-called Gelfand-Graev idempotents first considered in \cite{KM}.

\subsection{Gelfand-Graev idempotents}
\label{SSGG}
Recall from \S\ref{SSSC} that for all $j\in J$, we have defined special $R_\de$-modules $L_{\de,j}$ with $\CH L_{\de,j}=\sum_{\bi\in I^{\de,j}}\bi$. From now on, for every $j\in J$, we fix an arbitrary word  
\begin{equation}\label{Eli}
\bl^j=(l_{j,1},\dots,l_{j,e})\in I^{\de,j}.
\end{equation}
Consider the divided power words
\begin{equation}\label{EL(d)}
\bl^j(d):=(l_{j,1}^{(d)},\dots,l_{j,e}^{(d)})\in I^{d\de}_\di\qquad(j\in J).
\end{equation} 
Recall the notation~\eqref{ELaCol} and 
let $(\la,\bc)\in\La^\col(n,d)$. 
We set
\begin{align*}
\bl(\la,\bc)&:=\bl^{c_1}(\la_1)\dots \bl^{c_n}(\la_{n})\in I^{d\de}_\di.
\end{align*}
Now, we define the {\em Gelfand-Graev idempotent} $\ga^{\la,\bc}$ and the integer $a_\la$ as follows:
\begin{align}\label{EGG}
\ga^{\la,\bc}&:=1_{\bl(\la,\bc)} \in R_{d\de}, \\
\label{Eala}
a_{\la}&:= -\langle \bl(\la,\bc) \rangle 
= - e \sum_{t=1}^n \la_t (\la_t-1)/2, 
\end{align}
cf.~\S\ref{SSDiv}.
In the special case $n=1$, $\la=(d)$, $\bc=(j)$, we also use the notation
\begin{equation}\label{ESpecialCaseGa}
\ga^{d,j}:=1_{\bl^j(d)}.
\end{equation}
We set
\begin{align}
\om&:=(1,\dots,1) \in \La(d,d), \label{ELittleOm} \\
\ga^\om&:=\sum_{\bb\in J^d}\ga^{\om,\bb}\in R_{d\de}.
\label{EGaOm}
\end{align}

\begin{Lemma}\label{Ldim1}
For any $(\la,\bc),(\la',\bc')\in\La^\col(n,d)$, we have
\begin{equation}\label{Ldim1_0}
\qdim(\ga^{\la,\bc} C_{\rho,d} \ga^{\la',\bc'}) = 
q^{a_{\la}-a_{\la'}} \hspace{-5mm}
\sum_{\substack{\mu\in \Par_{\rho,d} \\ 
\t \in \Std(\mu\sm\rho,\, \bl(\la,\bc)^{+\ka}) \\
\t'\in \Std(\mu\sm\rho,\, \bl(\la',\bc')^{+\ka})}}
q^{\deg(\t) + \deg(\t')}. 
\end{equation}
\end{Lemma}

\begin{proof}
It follows from Lemma~\ref{LCoreCont}, Theorem~\ref{TDim} and Corollary~\ref{Ctens} 
that 
\begin{align} \notag
 \qdim(R^{\La_0}_{\cont(\rho)})
 \qdim(\ga^{\la,\bc} C_{\rho,d} \ga^{\la',\bc'}) =
 \\
&\hspace{-2cm} = 
q^{a_{\la}-a_{\la'}} \hspace{-2mm}
\sum_{\substack{\mu\in \Par_{\rho,d},\, \bj,\bj'\in I^{\cont(\rho)}
\\ \t \in \Std(\mu,\, \bj (\bl(\la,\bc)^{+\ka})) \\
\t' \in \Std(\mu,\, \bj (\bl(\la',\bc')^{+\ka})) }}
q^{\deg(\t)+\deg(\t')}.
\label{Ldim1_1}
\end{align}
For each $\mu\in \Par_{\rho,d}$ and $\bj\in I^{\cont(\rho)}$, in view of Lemma~\ref{LCoreCont}, 
we have a bijection 
$$\Std(\mu, \bj(\bl(\la,\bc)^{+\ka}))\iso \Std(\rho,\bj)\times \Std(\mu\sm\rho, \bl(\la,\bc)^{+\ka}),\ \t \mapsto (\t_0,\t_1)$$ 
where $\t_0=\t|_{\Y\rho}$ and
$\t_1 (u) = \t(u)-|\rho|$ for all $u\in \Y\mu\sm \Y\rho$.
Moreover, by definition, $\deg(\t) = \deg(\t_0)+ \deg(\t_1)$. Hence, the right-hand side of~\eqref{Ldim1_1} is equal to the right-hand side of~\eqref{Ldim1_0} multiplied by 
\[
\sum_{\t_0,\t'_0\in \Std(\rho)} q^{\deg(\t_0)+\deg(\t'_0)} = 
\qdim(R^{\La_0}_{\cont(\rho)}),
\]
and the result follows after dividing both sides of~\eqref{Ldim1_1}
by $\qdim(R^{\La_0}_{\cont(\rho)})$. 
\end{proof}

The main aim of the rest of this section is to determine the rank of the free $\Z$-module $\ga^{\la,\bc} C_{\rho,d} \ga^\om$ for any $(\la,\bc)\in \La^{\col} (n,d)$, see Corollaries~\ref{CDimLaOm} and~\ref{CDimEqual}.

\subsection{Colored tableaux}\label{SSSST}
 A {\em horizontal strip} is
a convex subset $U$ of $\Nodes$ 
such that whenever $(r,s)\ne (k,l)$ are in $U$ we have $s\ne l$. A {\em vertical strip} is a convex subset $U$ of $\Nodes$ such that whenever 
$(r,s)\ne (k,l)$ are in $U$ we have $r\ne k$.

Recalling the notation of \S\ref{SSPar}, for any $i\in I$, we set 
$\Nodes^{I,i}=\Z_{>0} \times \Z_{>0} \times \{i\}\subset \Nodes^I$. 
Identifying $\Nodes^{I,i}$ with $\Nodes$, we have a notion of what it means for a subset of $\Nodes^{I,i}$ to be 
a {\em horizontal} or {\em vertical} strip. 
Given $j\in J$, we say that a subset $U$ of $\Nodes^I$ is a 
{\em $j$-bend} if the following conditions are satisfied:
\begin{enumerate}
\item $U\subset \Nodes^{I,j-1} \cup \Nodes^{I,j}$;
\item $U\cap \Nodes^{I,j-1}$ is a horizontal strip in $\Nodes^{I,j-1}$, and
$U\cap \Nodes^{I,j}$ is a vertical strip in $\Nodes^{I,j}$. 
\end{enumerate}

Now let $\bmu\in\Par^I (d)$. 
Given $(\la,\bc)\in \La^{\col}(n,d)$, 
we denote by $\CT(\bmu;\la,\bc)$ 
the set of all weakly increasing maps 
$\T \colon \Y\bmu \to \{1,\ldots,n\}$ such that for all $r=1,\ldots,n$ the set $\T^{-1} (r)$ is a $c_r$-bend  and 
$|\T^{-1} (r)|=\la_r$. We refer to the elements of $\CT(\bmu;\la,\bc)$ as the {\em colored tableaux of shape $\umu$ and type $(\la,\bc)$}. 

Colored tableaux will play the role of a combinatorial intermediary connecting the standard tableaux appearing in Lemma~\ref{Ldim1} and the explicit expression for $\dim \ga^{\la,\bc} C_{\rho,d} \ga^\om$ appearing on the right hand side of the formula in Corollary~\ref{CDimLaOm}.

\subsection{Counting standard tableaux in terms of colored tableaux}
\label{SSquot}

Given $0\leq i<e$ and $u\in \Z\times \Z$, we call the image of 
$\Y{(i+1,1^{e-i-1})}$ under the translation of $\Z\times \Z$ 
sending $(1,1)$ to $u$ the {\em $e$-hook with vertex $u$ and arm length $i$}, or simply an {\em $e$-hook}. 

Recall the abacus notation from \S\ref{SSAb}. 
For any $i\in I$, let $b_i=b_i (\rho)$, $b_{>i} = \sum_{j=i+1}^{e-1} b_j$ and 
$b_{<i} = \sum_{j=0}^{i-1} b_j$. Since $\rho$ is $d$-Rouquier, we have $b_{i+1}\ge b_i+d-1$ for $i=0,\dots,e-2$, and hence, for all $i\le j$ in $I$, 
\begin{align}
b_{>i} - b_{>j} & \ge (b_i+d-1) (j-i), \label{Eij1} \\
b_{<j} - b_{<i} & \le (b_j-d+1)(j-i). \label{Eij2}
\end{align}
Given $(r,s,i)\in \Nodes^I$, define the integers
\begin{align*}
x(r,s,i)&:=r- (e-i-1)(b_i-r+s)+b_{>i}, \\
y(r,s,i)&:= s+i(b_i-r+s)-b_{<i}.
\end{align*}
Define 
$$\Hook(r,s,i) \subset \Z\times \Z$$ 
be the $e$-hook with arm length $i$ and vertex $(x(r,s,i),y(r,s,i))$. 
The following lemma is a refinement of~\cite[Lemma 4]{CK} and \cite[Lemma 4.3]{Evseev}.

\begin{Lemma}\label{LUnion}
Let $\mu\in \Par_{\rho,d}$ and $\bmu=\quot(\mu)$. Then
\[
\Y\mu = \Y\rho \sqcup \bigsqcup_{u\in \Y\bmu} \Hook(u).
\]
Moreover, every $\Hook(u)$ with $u\in \Y\bmu$ has vertex of residue $\kappa$. 
\end{Lemma}

\begin{proof}
It is easy to check that $y(r,s,i)-x(r,s,i)\equiv -N\pmod{e}$\!, so the second statement holds. 

For the first statement, there is nothing to prove when $|\bmu|=0$, so we assume that $|\bmu|\ge 1$ and choose $(r,s,i)\in \Y\bmu$ such that 
$\Y\bmu \sm \{(r,s,i)\}=\Y\bnu$ for some 
$\bnu \in \Par^I (d-1)$. Arguing by induction on $d$, 
we may assume that 
the lemma holds for the partition $\nu\in \Par_{\rho,d-1}$ determined
from $\quot (\nu) = \bnu$, so it is enough to show that 
$\Y\mu\sm \Y\nu=\Hook(r,s,i)$.

Let $\bmu=(\mu^{(0)},\dots,\mu^{(e-1)})$ and 
$\bnu=(\nu^{(0)},\dots,\nu^{(e-1)})$. Then 
$\Y{ \mu^{(i)} } \sm \Y { \nu^{(i)} } = \{ (r,s) \}$ and 
$\Y{ \mu^{(j)} }= \Y{ \nu^{(j)} }$ for all $j\in I\sm\{ i\}$.
We have 
\begin{equation}\label{Emunu}
\Ab^e_N (\mu) = 
\big(\Ab^e_N (\nu) \sm \{ (a-1,i)\} \big) \cup \{ (a,i)\}
\end{equation}
for some $a\in\Z_{>0}$. 
In view of Lemma~\ref{LPrecSucc}, $\Ab^e_N (\mu)$ has $b_i-r$ beads and $s$ non-beads belonging to the runner $\Ru_i$ and preceding $(a,i)$, so
$a=b_i-r+s$. 
By~\cite[Lemma 4(1)]{CK}, we have 
\begin{align}
 \Ab_N^e (\mu)&\supseteq [0,a-1]\times [i+1,e-1],
 \label{ERect1} \\
\Ab_N^e (\mu)&\cap (\Z_{\ge a}\times [0, i-1]) = \vn. 
 \label{ERect2}
\end{align}
In particular, each of $(a-1,i+1),\ldots, (a-1,e-1)$ is a bead of 
$\Ab_N^e (\mu)$, and each of $(a,0), \ldots,(a,i-1)$ is a non-bead of $\Ab_N^e(\mu)$. 
By~\eqref{Emunu} and Lemma~\ref{LPrecSucc}, it follows that 
$\Y\mu\sm \Y\nu$ is an $e$-hook with arm length $i$ and vertex $(x,y)$ where $x$ is the number of the bead $(a,i)$ and $y$ is the number of the non-bead $(a-1,i)$ of $\Ab^e_N (\mu)$, cf.~the proof of~\cite[Lemma 4(2)]{CK}.
Using~\eqref{ERect1},~\eqref{ERect2} and the fact that there are $r-1$ beads of $\Ab_N^e (\mu)$ on $\Ru_i$ succeeding $(a,i)$, we obtain 
 $x=r+b_{>i} - a(e-i-1)=x(r,s,i)$. 
 Similarly, $y=s+ia-b_{<i} = y(r,s,i)$. 
\end{proof}

\begin{Corollary} \label{CDonkey} 
Let $0\leq f\leq d$, and $\mu\in \Par_{\rho,d}$, $\nu\in \Par_{\rho,f}$ be partitions with the $e$-quotients  $\bmu$, $\bnu$ respectively. Then $\Y\nu\subseteq \Y\mu$ if and only if $\Y{\bnu}\subseteq  \Y{\bmu}$. 
\end{Corollary}
\begin{proof}
The if-part follows from Lemma~\ref{LUnion}. For the only-if-part, we apply induction on $d-f$, the case $d=f$ being obvious. Let $d-f>0$. 
If $\Y{\bnu}\not\subseteq  \Y{\bmu}$, then there is a node $(r,s,i)\in \Y{\bmu}\sm \Y{\bnu}$ such that $\Y{\bnu}\cup\{(r,s,i)\}=\Y{\quot(\la)}$ for some $\la \in \Par_{\rho,f+1}$. Then $\Y\la=\Y\nu\sqcup \Hook(r,s,i)\subseteq \Y\mu$ by Lemma~\ref{LUnion}. By induction, $\Y{\quot(\la)}\subseteq\Y{\umu}$, which is a contradiction. 
\end{proof}

\begin{Lemma}\label{LHookTab}
For any $j\in J$, the set of standard $\bl^j$-tableaux whose shape is a partition consists of exactly two elements, $\t$ and $\s$, where 
\begin{enumerate}
\item[(a)] $\t$ has shape $(j,1^{e-j})$, with 
$\t(e-j+1,1)=e$, and $\deg (\t)=0$. 
\item[(b)] $\s$ has shape $(j+1,1^{e-j-1})$, with $\s(1,j+1)=e$, 
 and $\deg(\s)=1$. 
\end{enumerate}
\end{Lemma}

\begin{proof}
By Lemma~\ref{LCoreCont}, the shape of any standard tableau in question must be an element of $\Par_{\vn,1}$, and the rest is easy to see. 
\end{proof}

The graded dimension of $C_{\rho,d}$ for $d=1$ can be easily computed:

\begin{Lemma} \label{LDimD=1} 
For any $k,j\in J$, we have:
$$
\DIM (1_{\bl^k}C_{\rho,1}1_{\bl^j})=
\left\{
\begin{array}{ll}
1+q^2 &\hbox{if $k=j$,}\\
q &\hbox{if $k$ and $j$ are neighbors,}
 \\
0 &\hbox{otherwise.}
\end{array}
\right.
$$
\end{Lemma}
\begin{proof}
By Lemma~\ref{Ldim1}, we have
\[
\DIM (1_{\bl^k}C_{\rho,1}1_{\bl^j})=
\sum_{\substack{\mu\in \Par_{\rho,1} 
\\ \t \in \Std(\mu\sm \rho, (\bl^k)^{+\ka})
\\ \t' \in \Std(\mu\sm \rho, (\bl^j)^{+\ka})
}}
q^{\deg(\t) + \deg(\t')}.
\]
Let $\mu \in \Par_{\rho,1}$. By Lemma~\ref{LUnion}, the set 
$\Y \mu\sm \Y \rho$ is an $e$-hook with a vertex $v$ of residue $\ka$.
Let $i$ be the arm length of this $e$-hook and $\nu=(i+1,1^{e-i-1})$. 
Denoting by $\tau$ the translation of $\Z\times \Z$ which maps $(1,1)$ to $v$, we have a bijection 
$\Std(\mu\sm \rho, (\bl^k)^{+\ka}) \iso \Std(\nu, \bl^k)$ given by $\t\mapsto \s$ where $\s(u)=\t(\tau(u))$ for all $u\in \Y\nu$ (and similarly for $\bl^j$). Moreover, we have $\deg(\s)=\deg(\t)$ by~\cite[(4.6)]{Evseev}. 
Hence, 
$$\DIM(1_{\bl^k}C_{\rho,1}1_{\bl^j})=\sum q^{\deg(\s)+\deg(\s')},
$$ 
where the sum is over all $\mu\in \Par_{\varnothing,1}$ and all pairs 
$(\s,\s') \in \Std(\mu, \bl^k) \times \Std(\mu,\bl^j)$.
The result now follows by Lemma~\ref{LHookTab}. 
\end{proof}

Let $\Hook$ be an $e$-hook with arm length $i$ and vertex 
$(x,y)\in \Z\times \Z$, and let $\Gook$ be another $e$-hook. We refer to the node 
$(x,y+i)$ as the {\em hand} of $\Hook$ and to $(x+e-i-1,y)$ as the {\em foot} of $\Hook$. We call $\Gook$ a {\em right extension} of $\Hook$ if the foot of $\Gook$ is the right neighbor of the hand of $\Hook$. We call $\Gook$ a {\em bottom extension} of $\Hook$ if the hand of $\Gook$ is the bottom neighbor of the foot of $\Hook$. 
The following is deduced from the definition of $\Hook(r,s,i)$ by an easy calculation:

\begin{Lemma} \label{LHRS} 
Let $(r,s,i)\in\Nodes^I$. 
\begin{enumerate}
\item The hook $\Hook(r,s+1,i)$ is a right extension of $\Hook(r,s,i)$. 
\item The hook $\Hook(r+1,s,i)$ is a bottom extension of $\Hook(r,s,i)$.  
\end{enumerate} 
\end{Lemma}

\begin{Lemma}\label{LHandFoot}
Let $\bmu\in \Par^I (d)$. If nodes 
$(r,s,i),(k,l,j)\in \Y\bmu$ are independent, then
$\Hook(r,s,i)$ and $\Hook(k,l,j)$ are independent. 
\end{Lemma}

\begin{proof}
First, suppose that $i\ne j$. Without loss of generality, $i<j$. 
Since $|\bmu|= d$, we have $k+s\le d$. 
Also, $b_j-b_i\ge d-1$ as $\rho$ is $d$-Rouquier. 
We have 
\begin{align*}
y(k,l,j)-y(r,s,i)&= (l-s)+i(b_j-k+l-b_i+r-s)+(j-i)(b_j-k+l) \\
&\ \ \ -(b_{<j}-b_{<i}) \\
&\ge
1-s +i(b_j - b_i+2-k-s)+(j-i)(-k+l+d-1)\\
& \ge 1+i+d-k-s\\ 
&\ge 1+i,
\end{align*}
where we have used~\eqref{Eij2} for the second step. 
Hence, the vertex of $\Hook(k,l,j)$ has a greater second coordinate than the hand of $\Hook(r,s,i)$. 
A similar calculation using~\eqref{Eij1} shows that $x(r,s,i)-x(k,l,j)\geq e-j$, hence the vertex of $\Hook(r,s,i)$ has first coordinate  greater than that of the foot of $\Hook(k,l,j)$. 
Thus, 
$\Hook(r,s,i)$ and $\Hook(k,l,j)$ are independent. 

Now let $i=j$. Without loss of generality, $k<r$ and $l>s$. We have 
\begin{align*}
y(k,l,i)-y(r,s,i)&=(l-s)+i(l-k+r-s)\ge 1+2i\ge 1+i, \\
x(r,s,i)-x(k,l,i)&=(r-k)+(e-i-1)(r-s+l-k)\ge e-i,
\end{align*}
and it follows again that $\Hook(r,s,i)$ and $\Hook(k,l,j)$ are independent.  
\end{proof}

Recall from \eqref{Eli} that for every $j\in J$, we have fixed a word 
$\bl^j=(l_{j,1},\dots,l_{j,e})\in I^{\de,j}$.  
Define the map $q\colon J\times I\to\{1,\dots,e\}$ by the condition that 
$l_{j,q(j,i)}=i$ for all $j\in J$ and $i\in I$. 
Let $\mu\in \Par_{\rho,d}$ and $0<f\le d$. Suppose that $\nu\in \Par_{\rho,d-f}$ 
is a partition with $\Y\nu\subseteq \Y\mu$. 
Note that $\cont(\mu\sm \nu) = f \de $. 
For any $j\in J$, define the function  
$$\t_{\mu\sm \nu,j}\colon \Y\mu \sm \Y\nu \to \{1,\dots,e\},\ u\mapsto q(j,\res(u)-\kappa).$$ 
For the notation $\bl^j (f)=(l_{j,1}^{(f)},\dots,l_{j,e}^{(f)})\in I^{f\de}_\di$ in the following lemma see (\ref{EL(d)}). 
 
 \begin{Lemma}\label{LBend}
Let $j\in J$ and $\mu\in \Par_{\rho,d}$ with $e$-quotient $\umu$. Let  $0<f\le d$ and $\nu\in \Par_{\rho,d-f}$ with $e$-quotient $\unu$ satisfy $\Y\nu\subseteq \Y\mu$. Then
$$
\Std(\mu\sm \nu, \bl^j (f)^{+\ka})=
\left\{
\begin{array}{ll}
\varnothing &\hbox{if $\Y{\bmu}\sm \Y{\bnu}$ is not a $j$-bend;}\\
\{\t_{\mu\sm\nu,j}\} &\hbox{if $\Y{\bmu}\sm \Y{\bnu}$ is a $j$-bend.}
\end{array}
\right.
$$
\end{Lemma}

\begin{proof} Since $l_{j,1},\dots,l_{j,e}\in I$ are all distinct, any element of $\Std(\mu\sm \nu, \bl^j (f)^{+\ka})$ must assign $q(j,i-\kappa)$ to every node of residue $i$, i.e.\ such an element must be $\t_{\mu\sm\nu,j}$. So it suffices to prove the following:

\vspace{2 mm}
\noindent
{\em Claim.} We have  $\t_{\mu\sm\nu,j}\in \Std(\mu\sm \nu, \bl^j (f)^{+\ka})$ if and only if $\Y{\bmu}\sm \Y{\bnu}$ is a $j$-bend. 

\vspace{ 2mm}
For the claim, by Lemma~\ref{LUnion} and Corollary~\ref{CDonkey}, we have $\unu\subset\umu$ and 
$\Y{ \mu}\sm\Y{\nu} = \bigsqcup_{u\in \Y\bmu\sm \Y\bnu} \Hook(u)$.
Suppose that $\t_{\mu\sm\nu,j} \in \Std(\mu\sm \nu, \bl^j (f)^{+\ka})$. Then, for every $u\in \Y\bmu \sm \Y\bnu$, the restriction $\t_{\mu\sm\nu,j}|_{\Hook(u)}$ is $(\bl^j)^{+\ka}$-standard. By Lemmas~\ref{LHookTab} and~\ref{LUnion}, we deduce that  $\Y{\bmu}\sm \Y{\bnu}\subset \Nodes^{I,j-1} \cup \Nodes^{I,j}$. 
Suppose for contradiction that 
$(\Y{\bmu} \sm \Y{\bnu}) \cap \Nodes^{I,j}$ is not a vertical strip.
Then $(r,s,j),(r,s+1,j)\in \Y\bmu\sm \Y\bnu$ for some $r,s$. Let $u$ be the hand of $\Hook(r,s,j)$ and $v$ be the foot of $\Hook(r,s+1,j)$. By Lemma~\ref{LHRS}(i), the node 
$v$ is the right neighbor of $u$. By Lemma~\ref{LHookTab}(b), we have 
$\t_{\mu\sm\nu,j}(u)=e>\t_{\mu\sm\nu,j}(v)$, which contradicts the standardness of $\t_{\mu\sm\nu,j}$. Hence, $(\Y{\bmu} \sm \Y{\bnu}) \cap \Nodes^{I,j}$ is a vertical strip. A similar argument, using Lemmas~\ref{LHRS}(ii) and~\ref{LHookTab}(a), shows that $(\Y{\bmu} \sm \Y{\bnu}) \cap \Nodes^{I,j-1}$ is a horizontal strip. 

Conversely, suppose that $\Y{\bmu}\sm \Y{\bnu}$ is a $j$-bend. By Lemmas~\ref{LHandFoot} and~\ref{LHRS},  $\Y{\mu}\sm \Y{\nu}$ is a disjoint union of independent sets of two types: (1) consecutive right extensions of hooks with arm length $j-1$; (2) consecutive bottom extensions of hooks with arm length $j$. In fact, we may assume that either $\Y{\mu}\sm \Y{\nu}$ is of type (1) or  $\Y{\mu}\sm \Y{\nu}$ is of type (2). If $\Y{\mu}\sm \Y{\nu}$ is of type (1), i.e. $\Y{\mu}\sm \Y{\nu}$ is a union $\Hook_1\sqcup\dots\sqcup\Hook_m$ of hooks with arm length $j-1$, 
then by Lemma~\ref{LHookTab}(a), $\t_{\mu\sm\nu,j}(v)=e$ for any $v$ which is a foot of $\Hook_a$ for $a=1,\dots,m$. So $\t_{\mu\sm\nu,j}$ is standard if $\t_{\mu\sm\nu,j}|_{\Hook_a}$ is standard for all $a$. Hence we may assume that $m=1$. But in this case $\t_{\mu\sm\nu,j}$ is easily seen to be standard using Lemma~\ref{LHookTab}(a) one more time. The case where $\Y{\mu}\sm \Y{\nu}$ is of type (2) is similar. 
\end{proof}

Recall the set $\La^\col(n,d)$ of colored compositions defined by~\eqref{ELaCol} and the set $\CT(\bmu;\la,\bc)$ of colored tableaux of shape $\umu$ and type $(\la,\bc)$ from \S\ref{SSSST}. 

\begin{Corollary}\label{CStTRB}
Let $\mu\in \Par_{\rho,d}$, 
$\bmu=\quot (\mu)$ and 
$(\la,\bc)\in \La^{\col} (n,d)$. 
Then 
$|\Std(\mu\sm\rho,\bl(\la,\bc)^{+\ka})|=|\CT(\bmu;\la,\bc)|$.
\end{Corollary}

\begin{proof} 
Recall that
$
\bl(\la,\bc)=\bl^{c_1}(\la_1)\dots \bl^{c_{n-1}}(\la_{n-1})\bl^{c_n}(\la_n).
$ 
We have 
$$
\bl^{c_1}(\la_1)\dots \bl^{c_{n-1}}(\la_{n-1})=\bl(\la',\bc')
$$
for $\la'=(\la_1,\dots,\la_{n-1})$ and $\bc'=(c_1,\dots,c_{n-1})$. 
Then 
\begin{align*}
|\Std(\mu\sm\rho,\bl(\la,\bc)^{+\ka})|
=\sum_{\Y\rho\subseteq \Y\nu\subseteq \Y\mu}
|\Std(\nu\sm\rho,\bl(\la',\bc')^{+\ka})||\Std(\mu\sm\nu,\bl^{c_n}(\la_n))|.
\end{align*}
If $|\Std(\nu\sm\rho,\bl(\la',\bc')^{+\ka})|\neq 0$, then $\cont(\nu)=\cont(\rho)+(d-\la_n)\de$, whence by Lemma~\ref{LCoreCont}, we have $\nu\in\Par_{\rho,d-\la_n}$. Arguing by induction on $n$, for such $\nu$ we have $|\Std(\nu\sm\rho,\bl(\la',\bc')^{+\ka})|=|\CT(\bnu;\la',\bc')|$, where $\unu=\quot(\nu)$. Moreover, by Lemma~\ref{LBend}, we have 
$$
|\Std(\mu\sm\nu,\bl^{c_n}(\la_n))|=
\left\{
\begin{array}{ll}
1 &\hbox{if $\Y\umu\sm\Y\unu$ is a $c_n$-bend,}\\
0 &\hbox{otherwise.}
\end{array}
\right.
$$
 The result follows. 
\end{proof}

\subsection{Counting colored tableaux}\label{SSCount}
In view of Lemma~\ref{Ldim1} and Corollary~\ref{CStTRB}, we can understand the dimensions of $\ga^{\la,\bc} C_{\rho,d} \ga^{\la',\bc'}$ for $(\la,\bc),(\la',\bc')\in \La^\col (n,d)$ by counting appropriate colored tableaux. The first main goal of this subsection is a formula for $|\CT(\bmu;\la,\bc)|$.

Recall that $J=\{1,\dots,e-1\} = I\sm \{0\}$. 
For $j\in J$, we define
\begin{equation*} 
\Inc(j):=\{j,j-1\}\subseteq I.
\end{equation*}

\begin{Remark} 
{\rm 
The notation $\Inc(j)$ is motivated by the following considerations. The irreducible semicuspidal $R_{\de,\F}$-modules are exactly the irreducible $R_{\de,\F}$-modules which factor through $R_{\de,\F}^{\La_0}$, see \cite[Lemma 5.1]{Kcusp}. The algebra $R_{\de,\F}^{\La_0}$ is a Brauer tree algebra of type $A_e$ with vertices $I$ and edges in natural bijection with $J$, so that $\Inc(j)$ is just the set of vertices incident to the edge $j$.  
}
\end{Remark}

Let $\Char:=\bigoplus_{t\in\Z_{\geq 0}}\Z\Irr(\Si_t)$
 be the $\Z$-module of all formal $\Z$-linear combinations of irreducible characters of $\Si_t$ for $t=0,1,\dots$. We have the inner product $\langle\cdot,\cdot\rangle$ on 
$\Char$ such that on each summand it is the standard inner product on (generalized) characters and $\Z\Irr(\Si_t)$ is orthogonal to $\Z\Irr(\Si_u)$ for $t\neq u$. Let 
$$
\Char^I:=\bigotimes_{i\in I}\Char,
$$ 
with the induced inner product. For every $\umu=(\mu^{(0)},\dots,\mu^{(e-1)})\in \Par^I$, we define
$$
\chi^\umu:=\chi^{\mu^{(0)}}\otimes\dots\otimes \chi^{\mu^{(e-1)}}\in\Char^I,
$$
where $\chi^\mu$ denotes the irreducible character of $\Si_t$ corresponding to the partition $\mu\in\Par(t)$. 

Let $S,T$ be finite sets and $m,l\in\Z_{>0}$. We denote by 
$\Mat(S,T)$ the set of all matrices $A=(a_{s,t})_{s\in S,t\in T}$ with non-negative integer entries. Given $A\in \Mat(S,T)$, we set 
\begin{align*}
\al_s(A)&:=\sum_{t\in T} a_{s,t}  \qquad (s\in S), \\
\be_t(A)&:=\sum_{s\in S} a_{s,t} \qquad (t\in T). 
\end{align*}
We write $\Mat(m,T):=\Mat([1,m], T)$, 
$\Mat(S,m):=\Mat(S,[1,m])$, etc.
Given $\mu\in \La(m)$ and $\la\in \La(l)$, we define 
\begin{align*}
{}_\mu \Mat(m,T)&:=\{ A \in \Mat(m,T) \mid \al_r(A) = \mu_r\text{ for all } r\in [1,m]\}, \\
\Mat(S,m)_{\mu} &:= \{ A \in \Mat(S,m) \mid \be_r (A) = \mu_r \text{ for all } r\in [1,m]\}, \\
{}_\la \Mat(m,l)_{\mu}&:= {}_\la \Mat(m,l) \cap \Mat(m,l)_{\mu}. 
\end{align*}

Let $(\la,\bc)\in\La^\col(n,d)$. We define 
\begin{equation}\label{ELaCMat}
{}_{(\la,\bc)}\Mat(n,I):=
\{ A = (a_{r,i})\in {}_{\la} \Mat(n,I) \mid 
a_{r,i}=0 \text{ if } i\notin\Inc(c_r) \}
\end{equation}
Let $A =(a_{r,i})\in {}_{(\la,\bc)}\Mat(n,I)$.
For each $i\in I$, define
the parabolic subgroup
$$
\Si_{A,i}:=\Si_{a_{1,i}}\times \dots\times \Si_{a_{n,i}}\leq \Si_{\be_i(A)}
$$
and the induced character 
$$
\chi^{A,i}:=\ind_{\Si_{A,i}}^{\Si_{\be_i(A)}} \big(\sgn_{a_{1,i}}^{\de_{c_1,i}}\boxtimes\dots\boxtimes \sgn_{a_{n,i}}^{\de_{c_n,i}}\big),
$$
where, for $a\in\Z_{\geq 0}$ and $j\in J$, we interpret 
$
\sgn_a^{\de_{j,i}}$
as the trivial character of $\Si_a$ when $j\neq i$ and as the sign character of $\Si_a$ when $j= i$. 
Then set 
\begin{align*}
\chi^A&:=\chi^{A,0}\otimes\dots\otimes \chi^{A,e-1}\in\Char^I,
\\
\chi^{(\la,\bc)}&:=\sum_{A\in {}_{(\la,\bc)}\Mat(n,I)}\chi^A.
\end{align*}

\begin{Lemma} \label{LChar} 
Let $\umu\in \Par^I(d)$ and 
$(\la,\bc)\in \La^{\col} (n,d)$. Then $$|\CT(\bmu;\la,\bc)|=\langle \chi^{(\la,\bc)},\chi^\umu\rangle.$$ 
\end{Lemma}
\begin{proof}
We apply induction on $n=1,2,\dots$, the induction base being immediate from the definitions. Let $n>1$. Set $\la'=(\la_1,\dots,\la_{n-1})\in\La(n-1)$ and $\bc'=(c_1,\dots,c_{n-1})$. 
We denote ${}_{((\la_n),(c_n))}\Mat(1,I)$ by ${}_{(\la_n,c_n)}\Mat(1,I)$. 
For a matrix $A'\in {}_{(\la',\bc')}\Mat(n-1,I)$ and a one-row matrix $B\in{}_{(\la_n,c_n)}\Mat(1,I)$ we denote by 
$
\left(
\begin{matrix}
 A'   \\
 B 
\end{matrix}
\right)\in {}_{(\la,\bc)}\Mat(n,I)$ 
the vertical concatenation of $A'$ and $B$. 
Then
$$
{}_{(\la',\bc')}\Mat(n-1,I)\times {}_{(\la_n,c_n)}\Mat(1,I)\to {}_{(\la,\bc)}\Mat(n,I),\ (A',B)\mapsto \left(
\begin{matrix}
 A'   \\
 B 
\end{matrix}
\right)
$$
is a bijection. Denoting the entries of $B\in {}_{(\la_n,c_n)}\Mat(1,I)$
 by $b_i$, we have by transitivity of induction 
\begin{align*}
\chi^{(\la,\bc)}=\sum_{\substack{A'\in {}_{(\la',\bc')}\Mat(n-1,I)\\ B\in {}_{(\la_n,c_n)}\Mat(1,I)}} \bigotimes_{i\in I} 
\ind_{\Si_{\be_i(A'),b_i}}^{\Si_{\be_i(A)}}\left(\chi^{A',i}\boxtimes \sgn_{b_i}^{\de_{c_n,i}}\right).
\end{align*}
The proof is concluded by the following computation:
\begin{align*}
\langle\chi^\umu,\chi^{(\la,\bc)}\rangle&=\sum_{\substack{A'\in {}_{(\la',\bc')}\Mat(n-1,I)\\ B\in {}_{(\la_n,c_n)}\Mat(1,I)}}  \prod_{i\in I} \Big\langle  \chi^{\mu^{(i)}}, \ind_{\Si_{\be_i(A'),b_i}}^{\Si_{\be_i(A)}}\left(\chi^{A',i}\boxtimes \sgn_{b_i}^{\de_{c_n,i}}\right)\Big\rangle
\\
&=\sum_{\substack{A'\in {}_{(\la',\bc')}\Mat(n-1,I)\\ B\in {}_{(\la_n,c_n)}\Mat(1,I)}}  \prod_{i\in I} \de_{\be_i(A),|\mu^{(i)}|}\langle  \res_{\Si_{\be_i(A'),b_i}}\chi^{\mu^{(i)}}, \chi^{A',i}\boxtimes \sgn_{b_i}^{\de_{c_n,i}}\rangle
\\
&=
\sum_{\substack{A'\in {}_{(\la',\bc')}\Mat(n-1,I)\\ B\in {}_{(\la_n,c_n)}\Mat(1,I)}}
\sum_{\substack{\unu\subseteq \umu\\
|\mu^{(i)}|-|\nu^{(i)}|=b_i,\ \forall i\in I\\
\umu\setminus\unu \ \text{is a $c_n$-bend}
}}
\prod_{i\in I} \langle  \chi^{\nu^{(i)}}, \chi^{A',i}\rangle
\\
&=
\sum_{\substack{B\in {}_{(\la_n,c_n)}\Mat(1,I)}}
\sum_{\substack{\unu\subseteq \umu\\
|\mu^{(i)}|-|\nu^{(i)}|=b_i,\ \forall i\in I\\
\umu\setminus\unu \ \text{is a $c_n$-bend}
}}
\langle  \chi^{\unu}, \chi^{(\la',\bc')}\rangle
\\
&=
\sum_{\substack{B\in {}_{(\la_n,c_n)}\Mat(1,I)}}
\sum_{\substack{\unu\subseteq \umu\\
|\mu^{(i)}|-|\nu^{(i)}|=b_i,\ \forall i\in I\\
\umu\setminus\unu \ \text{is a $c_n$-bend}
}}
|\CT(\bnu;\la',\bc')|
\\
&=
\sum_{\substack{\unu\subseteq \umu\\
\umu\setminus\unu \ \text{is a $c_n$-bend}
}}
|\CT(\bnu;\la',\bc')|
\\
&=|\CT(\bmu;\la,\bc)|,
\end{align*}
where the second equality holds by Frobenius reciprocity, the third equality comes from the Littlewood-Richardson rule, the fifth equality holds by the inductive assumption and the remaining equalities are clear. 
\end{proof}

 Let $\bb\in J^d$ so that $(\om,\bb)\in\La^\col(d,d)$, cf.~\eqref{ELittleOm}.
Recalling~\eqref{ELaCMat}, we set 
\begin{align*}
\Mat(I,d)_{(\om,\bb)}&:=\{B\in \Mat(I,d)\mid B^\tr\in {}_{(\om,\bb)}\Mat(d,I)\}.
\end{align*}
Define the set 
\[
{}_{(\la,\bc)} \Mat(n,I,d)_{(\om,\bb)} = \{ (A,B) \in {}_{(\la,\bc)} \Mat(n,I) \times 
\Mat(I,d)_{(\om,\bb)} \mid \be_i (A) = \al_i(B)\ \forall i\in I \}. 
\]

\begin{Lemma}\label{LScProd}
For any $(\la,\bc)\in\La^\col(n,d)$, we have
$$
\langle \chi^{(\la,\bc)},\chi^{(\om,\bb)}\rangle
=
\sum_{(A,B) \in {}_{(\la,\bc)} \Mat(n,I,d)_{(\om,\bb)}}
\prod_{i\in I} |\Si_{\be_i(A)}:\Si_{A,i}|.
$$
\end{Lemma}

\begin{proof}
Denoting by $\reg_{\Si_t}$ the regular character of $\Si_t$, we have 
\begin{align*}
\langle \chi^{(\la,\bc)},\chi^{(\om,\bb)}\rangle
&=\Big\langle \sum_{A\in {}_{(\la,\bc)}\Mat(n,I)}\chi^A, \sum_{B\in {}_{(\om,\bb)}\Mat(d,I)}\chi^B\Big\rangle 
\\
&=\sum_{\substack{A\in {}_{(\la,\bc)}\Mat(n,I)\\ B\in {}_{(\om,\bb)}\Mat(d,I)}} \prod_{i\in I}
\langle \chi^{A,i}, \reg_{\Si_{\be_i(B)}} \rangle,
\end{align*}
which implies the lemma since 
$$
\langle \chi^{A,i}, \reg_{\Si_{\be_i(B)}} \rangle=
\left\{
\begin{array}{ll}
\chi^{A,i}(1)= |\Si_{\be_i(A)}:\Si_{A,i}|&\hbox{if $\be_i(A)=\be_i(B)$,}\\
0 &\hbox{otherwise}
\end{array}
\right.
$$
for any $i\in I$.
\end{proof}  

For $(\la,\bc)\in\La^\col(n,d)$ and  $(\om,\bb)\in\La^\col(d,d)$ as above, we define the set ${}_{(\la,\bc)}\underline{\Mat}(n,d)_{(\om,\bb)}$ of tuples $(T^0,\dots,T^{e-1})$ such that 
(1) $T^i=(t^i_{r,s})\in \Mat(n,d)$ for all $i\in I$; (2) $T^0+\dots+T^{e-1}\in {}_\la\Mat(n,d)_\om$; (3) $t^i_{r,s}=0$ unless $i\in\Inc(c_r)\cap \Inc(b_s)$.

\begin{Lemma} \label{L020216} 
For any $(\la,\bc)\in\La^\col(n,d)$ and $(\om,\bb)\in\La^\col(d,d)$, we have
$$
|{}_{(\la,\bc)}\underline{\Mat}(n,d)_{(\om,\bb)}|=\sum_{(A,B) \in {}_{(\la,\bc)} \Mat(n,I,d)_{(\om,\bb)}}
\prod_{i\in I} |\Si_{\be_i(A)}:\Si_{A,i}|.
$$
\end{Lemma}
\begin{proof}
Consider the map $\theta\colon {}_{(\la,\bc)}\underline{\Mat}(n,d)_{(\om,\bb)} \to
\Mat(n,I) \times \Mat(I,d)$ defined as follows. 
Given ${\mathbf T}=(T^0,\dots,T^{e-1}) \in {}_{(\la,\bc)}\underline{\Mat}(n,d)_{(\om,\bb)}$, we set $\theta({\mathbf T}) = (A,B)$ where 
 $A=(a_{r,i})\in \Mat(n,I)$ and 
$B=(b_{i,s})\in \Mat(I,d)$ are given by $a_{r,i}:=\al_r (T^i)$ and $b_{i,s}:=\be_s(T^i)$. 
Clearly, the image of $\theta$ is contained in 
${}_{(\la,\bc)} \Mat(n,I,d)_{(\om,\bb)}$.

 Let $(A,B) \in {}_{(\la,\bc)} \Mat(n,I,d)_{(\om,\bb)}$.
 Then the preimage $\theta^{-1}(A,B)$ consists of all tuples $(T^0,\dots,T^{e-1})$ of matrices in $\Mat(n,d)$ such that $ \al_r(T^i)=a_{r,i}$ and $\be_s(T^i)=b_{i,s}$  for all $i\in I$, $r\in[1,n]$ and $s\in[1,d]$. So, denoting 
 $$
 S_i:=\{T\in \Mat(n,d)\mid \al_r(T)=a_{r,i},\ \be_s(T)=b_{i,s}\ \text{for all $r\in[1,n]$, $s\in[1,d]$}\}
 $$
 for any $i\in I$, we have $|\theta^{-1}(A,B)|=\prod_{i\in I}|S_i|$. 

To compute $|S_i|$ for a fixed $i\in I$, let $X=\{ s \in [1,d] \mid b_{i,s}=1 \}$, so that $|X|=\al_i(B) = \be_i(A)$. 
 Then the set of partitions of $X$ into a disjoint union of subsets $X_r$, $r\in [1,n]$, with $|X_r|=a_{r,i}$ for each $r$, is in bijection with the set $S_i$: a bijection is given by 
 assigning to each such partition $X=\bigsqcup_{r=1}^n X_r$ the matrix 
 $T=(t_{r,s})$ given by 
\[
t_{r,s}=
\begin{cases}
1 & \text{if } s\in X_r, \\
0 & \text{otherwise.}
\end{cases}
\] 
Therefore, $|S_i| = |\Si_{\be_i(A)}: \Si_{A,i}|$, proving the lemma. 
\end{proof}

\begin{Theorem} \label{TComb} 
For any $(\la,\bc)\in\La^\col(n,d)$ and $\bb\in J^d$, we have 
$$
\dim(\ga^{\la,\bc} C_{\rho,d} \ga^{\om,\bb}) = |{}_{(\la,\bc)}\underline{\Mat}(n,d)_{(\om,\bb)}|.
$$
\end{Theorem}
\begin{proof}
We have
\begin{align*}
\dim(\ga^{\la,\bc} C_{\rho,d} \ga^{\om,\bb}) &=\sum_{\mu\in\Par_{\rho,d}}|\Std(\mu\sm\rho, \bl(\la,\bc)^{+\ka})|\,|\Std(\mu\sm\rho, \bl(\om,\bb)^{+\ka})|
\\
&=\sum_{\umu\in\Par^I(d)}
|\CT(\bmu;\la,\bc)|\, |\CT(\bmu;\om,\bb)|
\\
&=\sum_{\umu\in\Par^I(d)}\langle \chi^{(\la,\bc)},\chi^\umu\rangle
\, \langle \chi^{(\om,\bb)},\chi^\umu\rangle
\\
&=\langle \chi^{(\la,\bc)}, \chi^{(\om,\bb)}\rangle
\\&=|{}_{(\la,\bc)}\underline{\Mat}(n,d)_{(\om,\bb)}|,
\end{align*}
where the first equality comes from Lemma~\ref{Ldim1}, the second equality uses Corollary~\ref{CStTRB} and Lemma~\ref{LQuotUn}, the third equality uses Lemma~\ref{LChar}, the fourth equality holds since the elements $\chi^\umu$ form an orthonormal basis of $\Char^I$, and the final equality comes from Lemmas~\ref{LScProd} and \ref{L020216}. 
\end{proof}

\begin{Corollary}\label{CDimLaOm} 
Let $(\la,\bc) \in\La^{\col} (n,d)$. 
For all $j\in J$, set  
$d_j=\sum_{\substack{1\le r\le n \\ c_r = j}} \la_r$.  
Then
$$
\dim(\ga^{\la,\bc} C_{\rho,d} \ga^{\om}) = 
\left\{
\begin{array}{ll}
|\Si_d:\Si_\la|\,3^{d_1+d_{e-1}}\,4^{\sum_{j=2}^{e-2}d_j} &\hbox{if $e>2$,}\\
|\Si_d:\Si_\la|\, 2^{d_1} &\hbox{if $e=2$.}
\end{array}
\right.
$$
\end{Corollary}
\begin{proof}
In this paragraph we fix $\bb\in J^d$. Let $Y_\bb$ be the set of all maps 
$\phi \colon [1,d]\to [1,n]\times I$ such that 
(1) $|\phi^{-1}(\{r\}\times I)|=\la_r$ for all $r\in [1,n]$; (2) for all $s\in [1,d]$, if $\phi(s)= (r,i)$, then $i\in\Inc(c_r)\cap\Inc(b_s)$. 
Observe that there is a bijection $f\colon {}_{(\la,\bc)}\underline{\Mat}(n,d)_{(\om,\bb)}\iso Y_\bb$ such that $f({\mathbf T})(s)$ is the unique 
$(r,i)\in [1,n]\times I$ such that $t^i_{r,s}=1$, if we write ${\mathbf T}=(T^0,\dots,T^{e-1})$ with $T^i=(t^i_{r,s})$. 

Now, let $Y:=\{(\phi,\bb)\mid \bb\in J^d,\ \phi\in Y_\bb\}$. By Theorem~\ref{TComb} and (\ref{EGaOm}), we have $\dim(\ga^{\la,\bc} C_{\rho,d} \ga^{\om}) = |Y|$. Let $W$ be the set of all set partitions $[1,d]=\bigsqcup_{r\in [1,n]}\Om_{r}$ such that $|\Om_{r}|=\la_r$ for all $r$. Note that $|W|=|\Si_d:\Si_\la|$.

We define the map $\xi\colon Y\to W$ by setting $\xi(\phi,\bb)$ to be the partition $[1,d]= \bigsqcup_{r\in [1,n]}\phi^{-1}(\{r\}\times I)$. To complete the proof, we fix a set partition 
$\Om\colon [1,d]=\bigsqcup_{r\in [1,n]}\Om_{r}$ in $W$ and compute 
$|\xi^{-1}(\Om)|$. 
Given $j\in J$, set $\Inc^2(j):=\{(i,l)\in I\times J\mid i\in\Inc(j)\cap\Inc(l)\}$. Note that 
$$
\Inc^2(j)=
\left\{
\begin{array}{ll}
\{(j,j),(j,j+1),(j-1,j),(j-1,j-1)\} &\hbox{if $1<j<e-1$,}\\
\{(1,1),(1,2),(0,1)\} &\hbox{if $j=1$ and $e>2$,}
\\
\{(e-1,e-1),(e-2,e-1),(e-2,e-2)\} &\hbox{if $j=e-1$ and $e>2$,}
\\
\{(0,1),(1,1)\} &\hbox{if $j=1$ and $e=2$.}
\end{array}
\right.
$$

Note that $\xi^{-1}(\Om)$ consists of all pairs $(\phi,\bb)$ where 
$\phi\colon [1,d]\to [1,n]\times I$ and $\bb\in J^d$ are such that for any $r\in J\times[1,n]$ and any $s\in\Om_{r}$ we have 
$\phi(s)=(r,i)$ with $(i,b_s)\in\Inc^2(c_r)$. So 
\begin{align*}
|\xi^{-1}(\Om)|&=\prod_{r\in [1,n]}\prod_{s\in\Om_{r}}|\Inc^2(c_r)|
=\prod_{j\in J}|\Inc^2(j)|^{d_j}
\\
&=\left\{
\begin{array}{ll}
3^{d_1+d_{e-1}}\,4^{\sum_{j=2}^{e-2}d_j} &\hbox{if $e>2$,}\\
2^{d_1} &\hbox{if $e=2$,}
\end{array}
\right.
\end{align*}
and the corollary follows. 
\end{proof}

Recall the algebra $W_d$ and the right $W_d$-modules $M_{\la,\bc}$ defined in~\S\ref{SSQ}. 
Combining Lemma~\ref{LBasisMLaC} and Corollary~\ref{CDimLaOm}, we obtain:

\begin{Corollary}\label{CDimEqual}
For all $(\la,\bc)\in \La^\col (n,d)$, we have 
$\dim (\ga^{\la,\bc} C_{\rho,d} \ga^\om) = \dim M_{\la,\bc}$. 
\end{Corollary}

\begin{Corollary} \label{CDimOmOm}
We have 
$
 \dim (\ga^\om C_{\rho,d} \ga^\om) = d! (4e-6)^d = \dim W_d
$.
\end{Corollary}

\begin{proof}
This can be derived from the algebra isomorphism in 
\cite[Theorem 3.4]{Evseev}. We give a more direct proof for the reader's convenience.
By~\eqref{EDimZZ}, we have $\dim \Zig = 4e-6$, and the second equality follows. For any $\bc \in J^d$ and $j\in J$, set 
$d_j (\bc) := |\{ r\in [1,d] \mid c_r = j\} |$.
For $e>2$, we compute:
\begin{align*}
\dim(\ga^\om C_{\rho,d} \ga^\om) &= \sum_{\bc \in J^d} \dim(\ga^{\om,\bc} C_{\rho,d} \ga^\om) \\
&= d! \sum_{\bc \in J^d} 3^{d_1(\bc)+d_{e-1} (\bc)} 4^{d_2(\bc)+\dots+d_{e-2}(\bc)}\\
&= d! (3+ 4(e-3) + 3)^d = d! (4e-6)^d,
\end{align*}
where the second equality is due to Corollary~\ref{CDimLaOm}.
For $e=2$, the same computation yields $\dim(\ga^\om C_{\rho,d} \ga^\om) = d! 2^d = d! (4e-6)^d$. 
\end{proof}

\section{The semicuspidal algebra}\label{SSemiCusp}
As usual, $d\in\Z_{>0}$ is fixed. Recall the semicuspidal algebra $\hat C_{d\de}$ from \S\ref{SSSC}.
In this section we prove some results on the structure of $\hat C_{d\de}$. These results are used in Section~\ref{SRoCKSchur} to study the quotient $C_{\rho,d}$ of $\hat C_{d\de}$ in the context of a RoCK block, cf.~\eqref{ECRhoD}.

\subsection{Preliminary results on the semicuspidal algebra}\label{SSTrunc}
We have the parabolic subalgebra 
\[
\hat C_{\om\de}\cong \hat C_\de\otimes\dots\otimes \hat C_\de\subseteq \hat C_{d\de}
\]
with the identity element $1_{\om\de}$, cf.~Lemma~\ref{L030216}. 

\begin{Lemma} \label{LPos} 
We have:
\begin{enumerate}
\item[{\rm (i)}] The algebra $1_{\om\de}\hat C_{d\de}1_{\om\de}$ is non-negatively graded.
\item[{\rm (ii)}] $1_{\om\de}  \hat C_{d\de}^{>0}1_{\om\de}=
1_{\om\de} \hat C_{d\de} \hat C_{\om\de}^{>0}=1_{\om\de} \hat C_{d\de}1_{\om\de}\hat C_{\om\de}^{>0}$.
\end{enumerate}
\end{Lemma}
\begin{proof}
(i) follows from  \cite[Lemma 6.9(iii)]{Evseev}. 
The second equality in (ii) is obvious and the first one follows from \cite[Lemma 6.9(i)(ii)]{Evseev}. 
\end{proof}

For $\bj=(j_1,\dots,j_d)\in J^d$, we define 
\begin{equation}\label{EEBk}
e_\bj:=1_{\bl^{j_1}\dots\bl^{j_d}}\in R_{d\de}.
\end{equation}
In particular, for $j\in J$, we interpret $e_j$ as $1_{\bl^j}$.  
In fact, the idempotent $e_\bj$ is also known as $\ga^{\om,\bj}$,
cf.~\eqref{ELittleOm},~\eqref{EGaOm}. So we have $\ga^\om=\sum_{\bj\in J^d}e_\bj \in R_{d\de}.
$

Following \cite{KM2}, we consider the $R_{\de}$-modules $\De_{\de,j}:=\hat C_\de e_{j}$ for every $j\in J$. Note that $R_\de$ and hence $\hat C_\de$ is non-negatively graded. Recalling the modules $L_{\de,j}$ with basis  $\{v_\bi\mid \bi\in I^{\de,j}\}$ from \S\ref{SSSC}, the following is immediate from \cite[Proposition 4.13]{KM2}:

\begin{Lemma} \label{L030216_2} 
Let $j\in J$. Then $\De_{\de,j}$ is non-negatively graded and there is an isomorphism of $R_{\de}$-modules  
$$\De_{\de,j}/\De_{\de,j}^{>0}\iso L_{\de,j}, \; e_j+\De_{\de,j}^{>0}\mapsto v_{\bl^j}.$$ 
\end{Lemma}

\begin{Lemma}\label{LDeIso} 
For any $\bj\in J^d$, we have an isomorphism of $R_{d\de}$-modules
$$
\hat C_{d\de} e_{\bj}\iso \De_{\de,j_1}\circ\dots\circ \De_{\de,j_d}, \; e_\bj\mapsto 1_{\om\de}\otimes e_{j_1}\otimes\dots\otimes e_{j_d}.
$$
\end{Lemma}
\begin{proof}
This follows from Lemma~\ref{LIndCuspProj}. 
\end{proof}

Let $\bj\in J^d$. We consider the following submodule of the $\hat C_{d\de}$-module $\hat C_{d\de} e_{\bj}$:
\begin{equation}\label{ENk}
N_\bj:=\hat C_{d\de} (\hat C_{\om\de}^{>0}) e_{\bj}.
\end{equation}

\begin{Lemma}\label{L030216_3}
For any $\bj\in J^d$, we have an isomorphism of $R_{d\de}$-modules
\begin{align*}
\hat C_{d\de} e_{\bj}/N_\bj\iso
L_{\de,j_1} \circ \cdots \circ L_{\de,j_d}, \ 
e_{\bj} +N_\bj \mapsto 
1_{\om\de} \otimes v_{\bl^{j_1}} \otimes \cdots \otimes v_{\bl^{j_d}}.
\end{align*}
\end{Lemma}
\begin{proof}
By Lemmas~\ref{LIndCusp}, \ref{L030216_2} and \ref{LDeIso}, there is a surjective $R_{d\de}$-module homomorphism as in the statement  of the lemma. That the homomorphism is injective follows from Lemmas~\ref{L030216} and \ref{L030216_2} again.
\end{proof}

\begin{Lemma}\label{LParTr}
If $\bj \in J^d$, then $e_{\bj} \hat C^0_{\om\de} e_{\bj}=\Z  e_{\bj}$.
\end{Lemma}

\begin{proof}
Clearly, it suffices to prove the lemma in the case $d=1$. 
For any word in $\bi\in I^{\de}$, the entries $i_1,\dots,i_e$ are distinct. Hence, by Theorem~\ref{TBasis}, we have 
$e_{\bj} R_{\de} e_{\bj}= \Z[y_1,\dots,y_e] e_{\bj}$, and the lemma follows. 
\end{proof}

\subsection{Some explicit elements of $\ga^\om  \hat C_{d\de}\ga^\om$}\label{SSExplicit}
Let $\De_\de:=\bigoplus_{j\in J}\De_{\de,j}$. 
In view of Lemma~\ref{LDeIso}, we have an isomorphism
\begin{equation}\label{EDeDeIso}
 \Delta_\de^{\circ d} \cong \hat C_{d\de} \ga^\om
\end{equation}
of left $\hat C_{d\de}$-modules. More precisely, we
can explicitly identity 
$\De_\de^{\circ d}$ with $\hat C_{d\de}\ga^\om$ so that the element $1_{\om\de}\otimes e_{j_1}\otimes\dots\otimes e_{j_d}$ of the natural direct summand $\De_{\de,j_1}\circ\dots\circ \De_{\de,j_d}$ of $\De_\de^{\circ d}$ corresponds to $e_\bj=e_\bj\ga^\om\in \hat C_{d\de}\ga^\om$ for all $\bj=(j_1,\dots,j_d)\in J^d$. 
So $\ga^\om  \hat C_{d\de}\ga^\om$ is naturally identified with 
$\End_{\hat C_{d\de}}(\De_\de^{\circ d})^\op$. The algebra $\End_{\hat C_{d\de}}(\De_\de^{\circ d})$ is 
described in \cite{KM2} as an {\em affine zigzag algebra} of rank $d$, 
so we can reinterpret this as a description of $\ga^\om \hat C_{d\de}\ga^\om$. 
We now define some explicit elements of $\ga^\om \hat C_{d\de}\ga^\om$ which correspond 
(up to an antiautomorphism and signs) to the elements of $\End_{\hat C_{d\de}}(\De_\de^{\circ d})$ with the same names   introduced in \cite[\S5.1]{KM2}.

For neighbors $k,j\in J$, we define $w_{k,j}\in\Si_e$ to be the unique permutation such that $w_{k,j}\bl^j=\bl^k$. Set $a^{k,j}:=\psi_{w_{k,j}}e_j\in \hat C_\de$. 
Further, define 
\begin{equation}\label{EEJ}
e_J:=\sum_{j\in J} e_j\in \hat C_\de, 
\end{equation}
and 
set $c:=(y_1-y_e) e_J \in \hat C_\de$, 
 $z:=y_1 e_J\in \hat C_\de$. Recall that, in view of Lemma~\ref{L030216}, we have identified the parabolic subalgebra $\hat C_{\om\de}\subseteq \hat C_{d\de}$ with $\hat C_\de\otimes\dots\otimes \hat C_\de$. For $t=1,\dots,d$ and $x\in\hat C_\de$, we define 
$$
x_t:=e_J^{\otimes t-1}\otimes x\otimes e_J^{\otimes d-t}\in\hat C_{\om\de}\subseteq \hat C_{d\de}.
$$
In particular we have the elements  $c_t,z_t, a^{j,k}_t \in  \ga^\om \hat C_{\om\de}\ga^\om$.

Recall the algebra $W_d$ and the signs $\zeta_1,\dots,\zeta_{e-1}$ defined by~\eqref{EWreath} and~\eqref{EZeta}.
Let $1\leq t<d$.
Consider the product of transpositions
\begin{equation}\label{EBl}
w_t:=\prod_{a=(t-1)e+1}^{te}(a,a+e)\in\Si_{de},
\end{equation}
and let $w:=w_1 \in \Si_{2e}$.
We set 
\begin{equation}\label{Erhat}
\hat r_t:=\sum_{j,k\in J}e_J^{\otimes t-1}\otimes (-\psi_{w}-\de_{k,j}\zeta_k)(e_k\otimes e_j)\otimes  e_J^{\otimes d-t-1}\in \ga^\om \hat C_{d\de}\ga^\om.
\end{equation}
Note that the sign here is opposite to the one in \cite{KM2}, which is technically more convenient for us, but does not affect the result below.

\begin{Theorem}\label{THomZZ} 
We have:
\begin{enumerate}
\item
There is an injective algebra homomorphism 
$\Theta\colon W_d\to \ga^{\om} \hat C_{d\de} \ga^{\om}$ with 
\[ 
\ze_\bj \mapsto e_{\bj}, \quad s_u \mapsto \hat r_u, \quad 
\zc^{(j)}[t] \mapsto \pm (ce_j)_{t}, \quad \za^{k,j}[t] 
\mapsto \pm a^{k,j}_t
\]
for all $\bj\in J^d$, $1\leq u<d$, $1\leq t\leq d$, and all admissible $k,j\in J$, where the signs depend on $k$ and $j$.
\item
For each $a\in \{0,1\}$, the map $\Theta$ restricts to a $\Z$-module isomorphism of graded components $W_d^a \iso \ga^\om \hat C_{d\de}^a\ga^\om$.
\item 
The algebra $\ga^\om \hat C_{d\de} \ga^{\om}$ is generated by 
$\Theta(W_d)$ together with $y_1 \ga^\om$. 
\end{enumerate}
\end{Theorem} 
\begin{proof}
Part (i) follows from  \cite[Theorem 5.9]{KM2} together with the fact that $a^{j,k}a^{k,j}=\pm ce_j$ for all neighbors $k,j\in J$, as observed in the proof of \cite[Theorem 4.17]{KM2}. Parts (ii) and (iii) follow from \cite[Theorem 5.9]{KM2} and the easy facts that the affine zigzag algebra is non-negatively graded and is generated by the finite zigzag algebra isomorphic to $W_d$ and a homogeneous element $\zz_1$ of degree $2$, see \cite[Section 3]{KM2}.
\end{proof}

Considering scalar extensions to a field $\k$, we also have the following result. Here and in the sequel, we denote
$\ga^{\om}:=\ga^\om \otimes 1 \in \hat C_{d\de,\k}$, 

\begin{Lemma}\label{LProjGen}
Let $\k$ be a field with $\charact \k =0$ or $\charact \k>d$.
The left $\hat C_{d\de,\k}$-module $\hat C_{d\de,\k} \ga^\om$ is a projective generator for the algebra $\hat C_{d\de,\k}$. 
\end{Lemma}

\begin{proof}
By~\cite[Lemma 6.22]{KM1}, the $\hat C_{d\de,\k}$-module 
$\Delta_\de^{\circ d}\otimes_{\Z} \k \cong (\Delta_{\de} \otimes_{\Z}\k)^{\circ d}$ is a projective generator.
By~\eqref{EDeDeIso}, we have 
$\Delta_{\de}^{\circ d} \otimes_{\Z} \k\cong \hat C_{d\de,\k} \ga^\om$, and the lemma follows. 
\end{proof}

\subsection{Imaginary tensor spaces}\label{SSIm}
Let $j\in J$. Following \cite{KM}, we refer to $T_{d,j}:=L_{\de,j}^{\circ d}$ as the {\em imaginary tensor space of color $j$}. 
In \cite[(4.2.9)]{KM}, an action of the symmetric group $\Si_d$ on $T_{d,j}$ with $R_{d\de}$-endomorphisms is defined as follows: 
$$
(1_{\om\de}\otimes v_{\bl^j}^{\otimes d})\cdot s_t=
(\zeta_j\psi_{w_t}+1_{d\de})1_{\om\de}\otimes v_{\bl^j}^{\otimes d} \qquad(1\leq t<d).
$$ 
Comparing with (\ref{Erhat}), we see that 
\begin{equation}\label{E060216}
(1_{\om\de}\otimes v_{\bl^j}^{\otimes d})\cdot s_t=
-\zeta_j \hat r_t \otimes v_{\bl^j}^{\otimes d}.
\end{equation}

As in \cite[\S5.2]{KM}, we define
$$
Z_{d,j}:=\{v\in T_{d,j}\mid v \cdot g =(-1)^{\ell(g)}v\ \text{for all $g\in\Si_d$}\}.
$$
Recall the Gelfand-Graev idempotent $\ga^{d,j}$ from (\ref{ESpecialCaseGa}).

\begin{Lemma} \label{L6.4.1} {\rm \cite[Lemma 6.4.1(ii)]{KM}} 
We have $Z_{d,j}=R_{d\de}\ga^{d,j}T_{d,j}$.
\end{Lemma}

More generally, fix $(\la,\bc)\in\La^\col(n,d)$ for some $n\in\Z_{>0}$. Define the semicuspidal $R_{d\de}$-module
$$
T_{\la,\bc}:=T_{\la_1,c_1}\circ\dots\circ T_{\la_n,c_n}.
$$
By the $n=1$ case considered above, we have the right action of $\Si_{\la_1}\times\dots\times \Si_{\la_n}=\Si_\la$ on $T_{\la_1,c_1}\boxtimes\dots\boxtimes T_{\la_n,c_n}$ with $R_{\la\de}$-endomorphisms. By functoriality of induction, this induces a right action of $\Si_\la$ on $T_{\la,\bc}$ with $R_{d\de}$-endomorphisms. 
Define
\[
Z_{\la,\bc}:= \{ v\in T_{\la,\bc} \mid v \cdot g = (-1)^{\ell(g)}v \text{ for all } g\in \Si_\la\}. 
\]
Recall the idempotent $\ga^{\la,\bc}$ from~\eqref{EGG}, and note that
$\ga^{\la,\bc} = \ga^{\la,\bc} 1_{\la\de}$. 

\begin{Lemma}\label{L6.4.1Mult}
We have $Z_{\la,\bc}= R_{d\de} (\ga^{\la,\bc} \otimes (T_{\la_1,c_1} \boxtimes \dots \boxtimes T_{\la_n,c_n}))$. 
\end{Lemma}

\begin{proof}
By Lemma~\ref{L6.4.1}, we have 
\begin{align*}
\{ v\in T_{\la_1,c_1} \boxtimes \cdots \boxtimes  T_{\la_n,c_n} \mid v \cdot g = (-1)^{\ell(g)} \text{ for all } g\in \Si_{\la}\} \hspace{2cm} &
\\ 
& \hspace{-7cm} =
R_{\la_1,\de}\ga^{\la_1,c_1} T_{\la_1,c_1} \boxtimes \cdots \boxtimes R_{\la_n\de}\ga^{\la_n,c_n} T_{\la_n,c_n}. 
\end{align*}
Moreover, for each $w\in {\D^{e\la}}$, we have an isomorphism of $\Z$-modules 
\[
T_{\la_1,c_1} \boxtimes \cdots \boxtimes  T_{\la_n,c_n} 
\iso \psi_w \otimes T_{\la_1,c_1} \boxtimes \cdots \boxtimes  T_{\la_n,c_n}, \; v \mapsto \psi_w \otimes v,
\]
which is equivariant with respect to the right action of $\Si_{\la}$. 
Therefore,
\begin{align*}
Z_{\la,\bc} &= \sum_{w\in \D^{e\la}} \psi_w \otimes R_{\la_1,\de}\ga^{\la_1,c_1} T_{\la_1,c_1} \boxtimes \cdots \boxtimes R_{\la_n\de}\ga^{\la_n,c_n} T_{\la_n,c_n} \\
&=
R_{d\de} \ga^{\la,\bc} \otimes (T_{\la_1,c_1} \boxtimes \dots \boxtimes T_{\la_n,c_n}),
\end{align*}
as required.
\end{proof}

Define the idempotent 
\begin{equation}\label{EElac}
e_{\la,\bc}:= e_{c_1^{\la_1}\ldots c_{\vphantom{1}n}^{\la_{\vphantom{1} n} }} \in \hat C_{d\de}
\end{equation}
and the $\hat C_{d\de}$-module
\begin{equation}\label{EMlac}
\hat T_{\la,\bc}:=\hat C_{d\de}e_{\la,\bc}.
\end{equation}
Recalling the notation~\eqref{ENk}, define the left 
$\hat C_{d\de}$-module
\begin{equation}\label{ENLaC}
 N_{\la,\bc} := N_{c_1^{\la_1}\ldots c_{\vphantom{1}n}^{\la_{\vphantom{1} n} }}
  = \hat C_{d\de} \hat C_{\om\de}^{>0} e_{\la,\bc}\subseteq \hat T_{\la,\bc}.
\end{equation}
By Lemma~\ref{L030216_3}, we have an isomorphism of left $R_{d\de}$-modules
\begin{equation}\label{EHatMM}
 T_{\la,\bc} \iso \hat T_{\la,\bc}/N_{\la,\bc}, \; 
1_{\om\de} \otimes 
v_{\bl^{c_1}}^{\otimes \la_1} \otimes \cdots \otimes v_{\bl^{c_n}}^{\otimes \la_n} \mapsto e_{\la,\bc} + N_{\la,\bc}. 
\end{equation}

Let $\Theta\colon W_d \to \ga^\om \hat C_{d\de} \ga^{\om}$ be the algebra  homomorphism of Theorem~\ref{THomZZ}. 
Recalling the element $\ze_{\la,\bc} \in W_d$ defined by~(\ref{EFancyE}), 
note that by Theorem~\ref{THomZZ}(i) we have 
\begin{equation}\label{EThLaC}
\Theta(\ze_{\la,\bc})=e_{\la,\bc}.
\end{equation}

Recall the function $\eps_{\la,\bc}$ from~\eqref{Eeps}. 
 Define the left $\hat C_{d\de} $-submodule
\begin{equation}\label{EZt}
 \Zt_{\la,\bc}:= 
 \{ v\in \hat T_{\la,\bc} \mid v \Theta(g) -
 \eps_{\la,\bc} (g)v \in N_{\la,\bc} \text{ for all } g\in \Si_{\la} \} \subseteq \hat T_{\la,\bc}. 
\end{equation}

\begin{Lemma}\label{LgTh}
For every $g\in \Si_{\la}$, we have $e_{\la,\bc} \Theta(g) = \Theta(g)e_{\la,\bc}$.
\end{Lemma}

\begin{proof}
Since we have 
 $\ze_{\la,\bc} g = g \ze_{\la,\bc}$ in $W_d$,  the lemma follows from~\eqref{EThLaC}.
\end{proof}

\begin{Lemma}\label{LZt}
We have 
$\Zt_{\la,\bc} = 
\hat C_{d\de} \ga^{\la,\bc} \hat C_{\la\de}e_{\la,\bc}+ 
N_{\la,\bc}$. 
\end{Lemma}

\begin{proof}
Throughout the proof, we identify $T_{\la,\bc}$ with $\hat T_{\la,\bc}/N_{\la,\bc}$ via the isomorphism (\ref{EHatMM}), so 
\[
1_{\la\de} \otimes T_{\la_1,c_1} \boxtimes \cdots \boxtimes T_{\la_n,c_n} = (\hat C_{\la\de}e_{\la,\bc}+N_{\la,\bc})/N_{\la,\bc}
\]
and we have a right action of $\Si_\la$ on $\hat T_{\la,\bc}/N_{\la,\bc}$.
The space $Z_{\la,\bc}$ of signed invariants under this action
 becomes a $\hat C_{d\de}$-submodule of 
$\hat T_{\la,\bc}/N_{\la,\bc}$, and by 
Lemma~\ref{L6.4.1Mult}, we have 
$Z_{\la,\bc}=(\hat C_{d\de} \ga^{\la,\bc} \hat C_{\la\de}e_{\la,\bc}+N_{\la,\bc})/N_{\la,\bc}$. 

Let $1\leq t<d$ satisfy $s_t\in\Si_{\la}$, and moreover,
let $q\in [1,n]$ be defined by the condition that $s_t$ lies in the $\Si_{\la_q}$-component of $\Si_\la$. 
By~\eqref{E060216}, we have 
\[
(e_{\la,\bc} + N_{\la,\bc})\cdot s_t = - \zeta_{c_q} \hat r_t e_{\la,\bc}+N_{\la,\bc}.
\]
Let $v=ve_{\la,\bc}\in \hat T_{\la,\bc}$. Then 
\begin{align*}
(v + N_{\la,\bc})\cdot s_t &= - \zeta_{c_q} v \hat r_t e_{\la,\bc}+N_{\la,\bc} = 
 - \zeta_{c_q} v \Theta(s_t) e_{\la,\bc} + N_{\la,\bc} \\
 & = -\zeta_{c_q} v e_{\la,\bc} \Theta(s_t) + N_{\la,\bc}
 = -\zeta_{c_q} v \Theta(s_t) + N_{\la,\bc},
\end{align*}
using Lemma~\ref{LgTh} for the third equality. 
So for any $g\in \Si_{\la}$, we have 
\[
(-1)^{\ell(g)} (v+N_{\la,\bc})\cdot g = 
\eps_{\la,\bc} (g)v \Theta(g) + N_{\la,\bc}.
\]
In particular, $N_{\la,\bc} \Theta(g) \subseteq N_{\la,\bc}$ for all 
$g\in \Si_{\la}$. It follows that $\Zt_{\la,\bc}$ is the preimage of 
$Z_{\la,\bc}$ under the canonical projection 
$\hat T_{\la,\bc} \twoheadrightarrow \hat T_{\la,\bc}/N_{\la,\bc}$. 
So $\Zt_{\la,\bc} = \hat C_{d\de} \ga^{\la,\bc} \hat C_{\la\de}e_{\la,\bc}+N_{\la,\bc}$ by the first paragraph of the proof.
\end{proof}

\subsection{The structure of $\ga^{\la,\bc} \hat C_{d\de} \ga^{\om}$} \label{SSLaOm}
In view of (\ref{EHatBi}) to $\bi\in I^\theta_\di$ we associate $\hat\bi\in I^\theta$. Throughout this subsection we drop the hats and usually write  $\bi$ for $\hat \bi$. For example, $\widehat{\bl^j(d)}$ is written simply as $\bl^j(d)$.

Let $h_d\in \Si_{ed}$ be defined by $h_d((t-1)e+q)=(q-1)d+t$ for all $t=1,\dots,d$ and $q=1,\dots,e$. In other words, $h_d$ is the shortest element of $\Si_{de}$ such that  $h_d ((\bl^j)^d) =  \bl^j(d)$ for all $j\in J$.  Let $w_{0,d}\in \Si_{ed}$ be the longest element of 
$\Si_{(d^e)}$, i.e.~$w_{0,d} ((q-1)d+t)=(q-1)d+d+1-t$ for all 
$q=1,\dots,e$ and $t=1,\dots,d$. Let $j\in J$ and note that $e_{j^d}=1_{(\bl^j)^d}$.  
We set 
\[
u_{d,j}:= \psi_{w_{0,d}}\psi_{h_d} e_{j^d} \in \hat C_{d\de}.
\]
More generally, fix $n\in\Z_{>0}$ and $(\la,\bc)\in\La^\col(n,d)$.
Recalling~\eqref{EElac}, we define
\begin{align*}
h_{\la}&:=(h_{\la_1},\dots,h_{\la_n})\in\Si_{e\la_1}\times\dots\times \Si_{e\la_n}=\Si_{e\la}\leq \Si_{ed},
\\
w_{0,\la}&:=(w_{0,\la_1},\dots,w_{0,\la_n})\in\Si_{e\la_1}\times\dots\times \Si_{e\la_n}=\Si_{e\la}\leq \Si_{ed},
\\
u_{\la,\bc}&:=\psi_{w_{0,\la}}\psi_{h_\la} e_{\la,\bc}
\\
&=u_{\la_1,c_1}\otimes\dots\otimes u_{\la_n,c_n}\in
\hat C_{\la_1\de}\otimes \dots\otimes \hat C_{\la_n\de}= \hat C_{\la\de}\subseteq \hat C_{d\de},
\end{align*}
where we have used the identification from Lemma~\ref{L030216}. 

\begin{Example}
If $e=3$, then $J=\{1,2\}$ and the only choice of the words~\eqref{Eli} is 
$\bl^1 = 021\in I^\de$ and $\bl^2 = 012\in I^\de$. In this case, if 
$d=4$, $n=5$, $\la=(3,0,1,0,0)\in \La(5,4)$ and $\bc = (2,1,1,1,2)\in J^5$, then in terms of Khovanov-Lauda diagrams~\cite[\S 2.1]{KL1} we have 
\[
u_{\la,\bc}=
\begin{array}{c}
\begin{braid}
\newcommand*{\xx}{1.2}
 \draw (0*\xx,0)node[below]{$0$}--(0*\xx,5);
 \draw (1*\xx,0)node[below]{$1$}--(3*\xx,5);
 \draw(2*\xx,0)node[below]{$2$}--(6*\xx,5);
 \draw(3*\xx,0)node[below]{$0$}--(1*\xx,5);
 \draw(4*\xx,0)node[below]{$1$}--(4*\xx,5);
 \draw(5*\xx,0)node[below]{$2$}--(7*\xx,5);
 \draw(6*\xx,0)node[below]{$0$}--(3*\xx,2.5)--(2*\xx,5);
 \draw(7*\xx,0)node[below]{$1$}--(5*\xx,5);
 \draw(8*\xx,0)node[below]{$2$}--(8*\xx,5);
 \draw(9*\xx,0)node[below]{$0$}--(9*\xx,9)node[above]{$0$};
 \draw(10*\xx,0)node[below]{$2$}--(10*\xx,9)node[above]{$2$};
 \draw(11*\xx,0)node[below]{$1$}--(11*\xx,9)node[above]{$1$};
 \draw(0*\xx, 5)--(2*\xx,9)node[above]{$0$};
 \draw(1*\xx,5)--(1*\xx,9)node[above]{$0$};
 \draw(2*\xx,5)--(0.3*\xx,7)--(0*\xx,9)node[above]{$0$};
 \draw(3*\xx, 5)--(5*\xx,9)node[above]{$1$};
 \draw(4*\xx,5)--(4*\xx,9)node[above]{$1$};
 \draw(5*\xx,5)--(3.3*\xx,7)--(3*\xx,9)node[above]{$1$};
 \draw(6*\xx, 5)--(8*\xx,9)node[above]{$2$};
 \draw(7*\xx,5)--(7*\xx,9)node[above]{$2$};
 \draw(8*\xx,5)--(6.3*\xx,7)--(6*\xx,9)node[above]{$2$};
\end{braid}
\end{array}
\]
\end{Example}

In view of (\ref{EKLId}) we get:

\begin{Lemma} \label{LStyagIdemp}
We have $u_{\la,\bc}\in\ga^{\la,\bc} \hat C_{d\de}e_{\la,\bc}$. 
\end{Lemma}

Recall the integer $a_{\la}$ defined by~\eqref{Eala}.
 The following is easily deduced from the definitions:

\begin{Lemma}\label{LDegU}
The element $u_{\la,\bc}\in \hat C_{d\de}$ is homogeneous of degree $a_{\la}$. 
\end{Lemma}

\begin{Lemma} \label{LExcellent} 
Let $j\in J$ and $\bi=\bi^{(1)}\dots\bi^{(d)}$ for some $\bi^{(1)},\dots,\bi^{(d)}\in I^{\de}$. If $g\in{}^{(d^{e})}\D^{(e^d)}$ is such that $g\bi=  \bl^j(d)$, then $g=h_d$ and $\bi^{(1)}=\dots=\bi^{(d)}=\bl^j$. 
\end{Lemma}
\begin{proof}
If $\bi\in I^\de$, the letters of $\bi$ are distinct. The result follows from this observation together with the definition of ${}^{(d^{e})}\D^{(e^d)}$. 
\end{proof}

Given $\la\in\La(n,d)$, define the composition
$$
\la^{\{e\}}:=(\la_1^e,\dots,\la_n^e)\in\La(ne,de).
$$
Define the {\em block permutation group} $B_d\cong \Si_d$ as the subgroup of $\Si_{de}$ generated by the involutions $w_1,\dots,w_{d-1}$ defined by (\ref{EBl}).

\begin{Lemma}\label{LStyagBlPerm} 
Let $\bi^{(1)},\dots,\bi^{(d)}\in I^{\de}_\scusp$ and $\bi=\bi^{(1)}\dots\bi^{(d)}$. If $g\in{}^{\la^{\{e\}}}\D^{(e^d)}$ is such that $g\bi=  \bl(\la,\bc)$ for some $\bc\in J^n$, then $g=h_\la b$ for some $b\in B_d$ such that $\ell(g)=\ell(h_\la)+\ell(b)$. 
\end{Lemma}
\begin{proof}
We apply the induction on $n$, the case $n=1$ being Lemma~\ref{LExcellent}. Let $n>1$. By the inductive hypothesis, we may assume that $\la_n>0$. Note that $ \bl(\la,\bc)= \bl(\la',\bc') \bl^{c_n}(\la_n)$, where $\la'=(\la_1,\dots,\la_{n-1})$ and $\bc'=(c_1,\dots,c_{n-1})$. Let $ \bl^{c_n}=(l_1,\dots,l_e)$ so that $ \bl^{c_n}(\la_n)=(l_1^{\la_n},\dots,l_e^{\la_n})$. We know that $l_1=0$ and $i^{(t)}_1=0$ for $t=1,\dots,d$, see Corollary~\ref{CSepSC}. Note that the positions $(d-\la_n)e+q$ 
for $q=1,\dots,\la_n$ in $ \bl(\la,\bc)$ correspond to the first $\la_n$ positions in $ \bl^{c_n}(\la_n)$, and so they are occupied with $0$s. So there exist $1\leq a_1,\dots,a_{\la_n}\leq d$ such that $g$ sends the first position of the word $\bi^{(a_q)}$  to 
the $q$th position of $ \bl^{c_n}(\la_n)$, 
i.e.~$g((a_q-1)e+1)=(d-\la_n)e+q$ for $q=1,\dots,\la_n$. Since $g\in {}^{\la^{\{e\}}}\D$, we have $a_1<\dots<a_{\la_n}$. Since $g\in \D^{(e^d)}$, it sends the remaining positions in the words $\bi^{(a_1)},\dots,\bi^{(a_{\la_n})}$ to the remaining positions of $ \bl^{c_n}(\la_n)$, i.e.~to the last $\la_n(e-1)$ positions of $ \bl(\la,\bc)$. It follows that $\bi^{(a_1)}=\dots=\bi^{(a_{\la_n})}=\bl^{c_n}$.

Let $b'\in B_d$ be the block permutation which moves the blocks $\bi^{(a_1)},\dots,\bi^{(a_{\la_n})}$ to the end in the same order and preserves the order of the remaining blocks. Let $g'=g(b')^{-1}$. We claim that $\ell(g')=\ell(g)-\ell(b')$. To prove this, it suffices to show that $g(r)>g(s)$ for all $1\leq r<s\leq ed$ such that $b'(r)>b'(s)$, which is clear since for any such $r,s$, the element $r$ is in one of the blocks corresponding to $\bi^{(a_1)},\dots,\bi^{(a_{\la_n})}$, whereas $s$ is not. 

We have $g' \in{}^{\la^{\{e\}}}\D^{(e^d)}$. Indeed, it is obvious that $g' \in\D^{(e^d)}$, and $g'\in {}^{\la^{\{e\}}}\D$ because $g=g'b'\in {}^{\la^{\{e\}}}\D$ and $\ell(g)=\ell(g')+\ell(b')$. 
Moreover, $g'\in\Si_{(d-\la_n)e,\la_ne}$, so the result follows by the inductive assumption. 
\end{proof}

\begin{Lemma} \label{LHLa} 
For any $y\in\Z[y_1,\dots,y_{de}]$ there exists $y'\in \Z[y_1,\dots,y_{de}]$ such that $1_{\bl(\la,\bc)}y\psi_{h_\la}=1_{\bl(\la,\bc)}\psi_{h_\la}y'$.
\end{Lemma}
\begin{proof}
This follows from the observation that the Khovanov--Lauda  diagram \cite[\S2.1]{KL1} of $1_{\bl(\la,\bc)} \psi_{h_{\la}}$ does not have any crossings of two strings with the same label and the relations (\ref{ypsi}), (\ref{psiy}). 
\end{proof}

\begin{Lemma}\label{LCyc}
For any $(\la,\bc)\in \La^{\col}(n,d)$, we have:
\begin{enumerate}
\item[{\rm (i)}] $\ga^{\la,\bc} \hat C_{d\de} 1_{\om\de} = 
u_{\la,\bc} \hat C_{d\de} 1_{\om\de}
=u_{\la,\bc} \ga^{\om} \hat C_{d\de} 1_{\om\de};$
\item[{\rm (ii)}] $\ga^{\la,\bc} \hat C_{\la\de} 1_{\om\de} =u_{\la,\bc} \hat C_{\om\de}$.
\end{enumerate}
\end{Lemma}

\begin{proof}
By Lemma~\ref{LStyagIdemp}, 
$u_{\la,\bc} \ga^{\om} \hat C_{d\de} 1_{\om\de}=u_{\la,\bc} \hat C_{d\de} 1_{\om\de}\subseteq 
\ga^{\la,\bc} \hat C_{d\de} 1_{\om\de}$, 
so for (i) it only remains to prove the  inclusion
$\ga^{\la,\bc} \hat C_{d\de} 1_{\om\de}  \subseteq u_{\la,\bc} \ga^{\om} \hat C_{d\de}1_{\om\de}$. Moreover, for (ii) we may assume that $n=1$ and prove only that $\ga^{d,j}\hat C_{d\de}1_{\om\de}\subseteq u_{d,j}\hat C_{\om\de}$. 

The word $  \bl(\la,\bc) \in I^{d\de}$ is the concatenation of $ne$ words of the form $(i^s)\in I^{s\al_i}$ for various $i\in I$ and 
$s\in \Z_{\ge 0}$. 
We denote the corresponding integer multiples  $s\al_i\in Q_+$ of simple roots
 by $\theta_1,\dots,\theta_{ne}$, listed in the order of concatenation, 
i.e.\ $\theta_{e(t-1)+q}=\la_t \al_{l_{c_t, q}}$ for all $t=1,\dots,n$ and $q=1,\dots,e$, cf.~\eqref{Eli}.
By Lemma~\ref{LMackey}, we have  
$$1_{\theta_1,\dots,\theta_{ne}}\hat C_{d\de} 1_{\om\de} = 
\sum_{g\in {}^{\la^{\{e\}}} {\!\D}^{(e^d)} } 
R_{\theta_1,\dots,\theta_{ne}} \psi_g \hat C_{\om \de}$$
as $(R_{\theta_1,\dots,\theta_{ne}}, \hat C_{\om \de})$-bimodules, so
\[
\ga^{\la,\bc} \hat C_{d\de} 1_{\om\de}=
\sum_{g\in {}^{\la^{\{e\}}} {\!\D}^{\vphantom{\la^{e}}(e^d)} } 
\ga^{\la,\bc} R_{\theta_1,\dots,\theta_{ne}} \psi_g \hat C_{\om \de}.
\]
Consider an element $g\in {}^{\la^{\{e\}}} {\!\D}^{(e^d)}$
such that the summand 
$U:=\ga^{\la,\bc} R_{\theta_1,\dots,\theta_{ne}} \psi_g \hat C_{\om \de}$ on the right hand side is non-zero. 
Then there exists $\bi \in I^{d\de}$ such that $g\bi =   \bl(\la,\bc)$ and 
$1_{\bi} 1_{\om\de}\ne 0$ in $\hat C_{\om\de}$, whence $\bi = \bi^{(1)}\ldots \bi^{(d)}$ for some $\bi^{(1)},\dots, \bi^{(d)} \in I^{\de}_{\scusp}$. 
Hence, by Lemma~\ref{LStyagBlPerm},
we have $g=h_{\la} b$ for some $b\in B_d$. Moreover, in the case when $n=1$, needed for part (ii), we have $b=1$ by Lemma~\ref{LExcellent}. We may assume that preferred reduced decompositions for the elements of $\Si_{de}$ are chosen in such a way that $\psi_g = \psi_{h_{\la}} \psi_b$, so 
$U=\ga^{\la,\bc} R_{\theta_1,\dots,\theta_{ne}} \psi_{h_{\la}} \psi_b 
\hat C_{\om \de}$. 

Let $P\subseteq R_{d\de}$ be the subalgebra generated by $y_1,\dots,y_{de}$. Then 
\begin{align*}
U&=\ga^{\la,\bc} R_{\theta_1,\dots,\theta_{ne}} \psi_{h_{\la}} \psi_b 
\hat C_{\om \de}
\\
&=\ga^{\la,\bc} \psi_{w_{0,\la}}P\psi_{h_{\la}} \psi_b 
\hat C_{\om \de}
\\
&=\ga^{\la,\bc} \psi_{w_{0,\la}}\psi_{h_{\la}}P \psi_b 
\hat C_{\om \de}
\\
&=u_{\la,\bc} \ga^\om P \psi_b 
\hat C_{\om \de},
\end{align*}
where we have used Lemma~\ref{LNHecke} for the second  equality, Lemma~\ref{LHLa} for the third equality, the definition of $u_{\la,\bc}$ and Lemma~\ref{LStyagIdemp} for the fourth  equality. 
Part (i) now follows since $
u_{\la,\bc} \ga^\om P \psi_b 
\hat C_{\om \de}\subseteq u_{\la,\bc} \ga^{\om} \hat C_{d\de} 1_{\om\de}$, whereas part (ii) follows since $\psi_b=1$ in the case $n=1$ and $P\hat C_{\om\de}=\hat C_{\om\de}$. 
\end{proof}

Multiplying the equality in Lemma~\ref{LCyc}(i) by $\ga^\om $ on the right, we obtain $\ga^{\la,\bc} \hat C_{d\de} \ga^\om = 
u_{\la,\bc} \ga^{\om} \hat C_{d\de} \ga^\om$. In particular:

\begin{Corollary}\label{CCyc} 
As a right $\ga^\om \hat C_{d\de}\ga^\om$-module, $\ga^{\la,\bc} \hat C_{d\de} \ga^\om$ is generated by  $u_{\la,\bc}$. 
\end{Corollary}

Note that by Lemmas~\ref{LPos} and~\ref{LDegU} and Corollary~\ref{CCyc}, we have $\ga^{\la,\bc} \hat C_{d\de} \ga^\om\subseteq \hat C_{d\de}^{\ge a_\la}$.
Recall the left $\hat C_{d\de}$-modules $N_{\bj}$ defined by~\eqref{ENk}.

\begin{Lemma}\label{LDegN}
For any $(\la,\bc)\in \La^{\col} (n,d)$ and $\bj\in J^d$, we have 
$\ga^{\la,\bc} N_{\bj} \subseteq \hat C_{d\de}^{>a_{\la}}$. 
\end{Lemma}

\begin{proof} Recall that $u_{\la,\bc} =u_{\la,\bc} \ga^\om$. 
We have 
\[
\ga^{\la,\bc} N_{\bj} = 
\ga^{\la,\bc} \hat C_{d\de} (\hat C_{\om\de}^{>0}) e_{\bj}
= u_{\la,\bc} \hat C_{d\de}  (\hat C_{\om\de}^{>0}) e_{\bj}
= u_{\la,\bc} e_{\la,\bc} \hat C_{d\de}  (\hat C_{\om\de}^{>0}) e_{\bj},
 \] 
where the second equality comes from Lemma~\ref{LCyc}(i) since $1_{\om\de}$ is the identity element of $\hat C_{\om\de}$, and the last equality holds by Lemma~\ref{LStyagIdemp}. Now $e_{\la,\bc} \hat C_{d\de} 1_{\om\de}$ is non-negatively graded by Lemma~\ref{LPos}(i), and $\deg(u_{\la})=a_{\la}$ by Lemma~\ref{LDegU}, so the lemma follows. 
\end{proof}

Recalling  (\ref{EMlac}) and the homomorphism $\Theta$ from Theorem~\ref{THomZZ}, 
define the $\hat C_{d\de}$-submodule
\[
 \hat Z_{\la,\bc}:= 
 \{ v\in \hat T_{\la,\bc} \mid v \Theta(g) =
 \eps_{\la,\bc} (g)v \text{ for all } g\in \Si_{\la} \} \subseteq \hat T_{\la,\bc} .
\]
Clearly, $\hat Z_{\la,\bc} \subseteq \Zt_{\la,\bc}$, cf.~\eqref{EZt}. 

\begin{Lemma}\label{LUinZhat}
We have $u_{\la,\bc} \in \hat Z_{\la,\bc}$. 
\end{Lemma}

\begin{proof}
By Lemma~\ref{LStyagIdemp} and the definition of $u_{\la,\bc}$,  we have $u_{\la,\bc}\in \ga^{\la,\bc} \hat C_{\la\de}e_{\la,\bc}$. Hence, by Lemma~\ref{LZt}, we get $u_{\la,\bc} \in \Zt_{\la,\bc}$. So for any $g\in \Si_{\la}$, we have $u_{\la,\bc}\Theta(g) - \eps_{\la,\bc}(g) u_{\la,\bc}\in \ga^{\la,\bc} N_{\la,\bc}$. By  Lemma~\ref{LDegU}, 
$u_{\la,\bc}\Theta(g) - \eps_{\la,\bc}(g) u_{\la,\bc}$ is homogeneous of degree $a_\la$. But $\ga^{\la,\bc} N_{\la,\bc}^{a_{\la}}=0$ by Lemma~\ref{LDegN}, so $u_{\la,\bc}\Theta(g) -\eps_{\la,\bc}(g)u_{\la,\bc}=0$.
\end{proof}

\begin{Lemma}\label{LLeftMult}
Let $(\la,\bc), (\mu,\bb) \in \La^{\col}(n,d)$. 
If $v\in \ga^{\mu,\bb}  \hat Z_{\la,\bc}$ is a homogeneous element of degree $a_{\mu}$, then $v=x u_{\la,\bc}$ for some $x\in \ga^{\mu,\bb} \hat C_{d\de} \ga^{\la,\bc}$. 
\end{Lemma}

\begin{proof}
By Lemma~\ref{LZt}, we have 
$v\in \ga^{\mu,\bb} \hat C_{d\de} \ga^{\la,\bc} \hat C_{\la\de}e_{\la,\bc}+\ga^{\mu,\bb} N_{\la,\bc}$. By Lemma~\ref{LDegN}, $\ga^{\mu,\bb} N_{\la,\bc}^{a_{\mu}}=0$, so 
$v\in \ga^{\mu,\bb} \hat C_{d\de} \ga^{\la,\bc} \hat C_{\la\de}e_{\la,\bc}$. Hence,  by Lemma~\ref{LCyc}(ii), we have 
$v\in  \ga^{\mu,\bb} \hat C_{d\de} u_{\la,\bc} \hat C_{\om\de}e_{\la,\bc}$. We know that $\hat C_{\om\de}$ is non-negatively graded, so we have $v=v_1+v_2$ for some homogeneous elements 
$v_1 \in \ga^{\mu,\bb} \hat C_{d\de} u_{\la,\bc} \hat C_{\om\de}^0 e_{\la,\bc}$
and 
$v_2 \in \ga^{\mu,\bb} \hat C_{d\de} u_{\la,\bc} \hat C_{\om\de}^{>0} e_{\la,\bc}$, with $\deg(v_1)=\deg(v_2)=\deg(v)=a_{\mu}$. 
By definition~\eqref{ENLaC} of $N_{\la,\bc}$, we have  $v_2 \in \ga^{\mu,\bb} N_{\la,\bc}$, whence $v_2=0$ by Lemma~\ref{LDegN}. On the other hand,
by Lemma~\ref{LParTr}, we have $u_{\la,\bc} \hat C_{\om\de}^0 e_{\la,\bc}=\Z u_{\la,\bc}$, so $v=v_1 \in \ga^{\mu,\bb} \hat C_{d\de} u_{\la,\bc}$, and the result follows by Lemma~\ref{LStyagIdemp}. 
\end{proof}

\section{RoCK blocks and generalized Schur algebras}\label{SRoCKSchur}

As in \S\ref{SSRoCK},  we fix $d\in\Z_{>0}$ and a $d$-Rouquier core $\rho$ of residue $\kappa$.

\subsection{Identifying $W_d$ with $\ga^{\om} C_{\rho,d} \ga^{\om}$}\label{SSTruncIso}
Recall from (\ref{ECRhoD}) that we have the natural surjection
$$
\Pi\colon \hat C_{\rho,d}\to \hat C_{d\de}/\ker\Om=C_{\rho,d}.
$$ 
This yields the surjections
\begin{equation}\label{ETruncSurjection}
\Pi_{\om} \colon \ga^\om \hat C_{d\de}\ga^\om \to \ga^\om C_{\rho,d}\ga^\om, \quad \Pi_{\la,\bc} \colon \ga^{\la,\bc} \hat C_{d\de} \ga^\om \to \ga^{\la,\bc} C_{\rho,d} \ga^\om
\end{equation}
for $(\la,\bc)\in \La^\col(n,d)$. 
For any $x\in \hat C_{d\de}$, we often denote by $x$ its image $\Pi(x)$ in $C_{\rho,d}$.

We define the algebra homomorphism
\begin{equation}\label{EXi}
\Xi:=\Pi_\om \circ \Theta \colon W_d \to \ga^\om C_{\rho,d} \ga^\om 
\end{equation}
where 
$\Theta\colon W_d \to \ga^{\om} \hat C_{\rho,d}\ga^{\om}$ is as in  Theorem~\ref{THomZZ}. 
Our aim is to prove that $\Xi$ is an isomorphism by generalizing the arguments
of~\cite[Section 7]{Evseev}, where a similar statement is proved over a field containing an element of quantum characteristic $e$.  
We begin with the case where $d=1$, when $W_d=\Zig$. 

\begin{Lemma}\label{LXi01}
 For $d=1$ and each $a\in \{0,1\}$, the map\, $\Xi$\, restricts to a $\Z$-module isomorphism of graded components $\Zig^a\iso\ga^{\om} C_{\rho,1}^a \ga^{\om}$. 
\end{Lemma}

\begin{proof}
By Theorem~\ref{THomZZ}(ii), $\Theta$ restricts to an isomorphism $\Zig^a\iso\ga^{\om} \hat C_{\rho,1}^a \ga^{\om}$, whence $\Xi$\, restricts to a surjection $\Zig^a\to\ga^{\om} C_{\rho,1}^a \ga^{\om}$.
 Moreover, by (\ref{EDimZZ}) and Lemmas~\ref{LCFree} and \ref{LDimD=1}, we have that $\Zig^a$ and $\ga^{\om} C_{\rho,1}^a \ga^{\om}$ are free $\Z$-modules of the same rank, which completes the proof. 
\end{proof}

\begin{Lemma}\label{LD=1Deg2}  
Let $d=1$ and $j\in J$. Then $C_{\rho,1}^2e_j=\Z(y_1-y_e)e_j$. 
\end{Lemma}
\begin{proof}
By Lemmas~\ref{LCFree} and \ref{LDimD=1}, the $\Z$-module $C_{\rho,1}^2e_j=e_jC_{\rho,1}^2e_j$ is free of rank $1$. It suffices to prove that $y:=(y_1-y_e)e_j\otimes 1$ generates $C_{\rho,1,\k}^2e_j$ over any field $\k$, i.e. that $y\neq 0$ for any field $\k$, cf. Remark~\ref{RCrazy}. This is proved in \cite[Proposition 7.2]{Evseev} for any field $\k$ containing an element of quantum characteristic $e$, in particular for $e=2$. So we may assume that $e>2$. By Corollary~\ref{CCSymmetric}, the algebra $C_{\rho,1,\k}$ has a symmetrizing form $F$ of degree $-2$. Since $e>2$, the element $j$ has a neighbor $k\in J$. In the rest of the proof, we write $x:=x\otimes 1\in C_{\rho,1,\k}$ for $x\in C_{\rho,1}$. Recalling the elements of $\hat C_\de$ introduced in \S\ref{SSExplicit}, note that $a^{k,j}\neq 0$ in $C_{\rho,1,\k}^1$ by Lemma~\ref{LXi01}. So there must exist an element $x\in C_{\rho,1,\k}^1$ such that $F(xa^{k,j})\neq 0$. Using Theorem~\ref{THomZZ} and Lemma~\ref{LXi01} again, we may assume that 
$x=\Xi (\za^{j,k})$, and hence
 $(y_1-y_e)e_j=\pm \Xi(\zc \ze_j)=\pm\Xi(\za^{j,k}\za^{k,j})=\pm a^{j,k}a^{k,j}\neq 0$. 
\end{proof}

Now we return to the case when $d\in \Z_{>0}$ is arbitrary.
By Lemma~\ref{L030216}, we have an embedding
\begin{equation}\label{EIotaChat}
\iota\colon \hat C_{\de} \to \hat C_{d\de},\; 
x\mapsto x\otimes 1_{(d-1)\de} \in \hat C_{\de} \otimes \hat C_{(d-1)\de} = \hat C_{\de,(d-1)\de} \subseteq \hat C_{d\de}. 
\end{equation}
In view of Theorem~\ref{THomZZ}(i), for any $j\in J$, we have
\begin{equation}\label{E120216}
\Theta(\ze_j[1])=\iota(e_j).
\end{equation}

\begin{Corollary}\label{CY1}
The element $y_1 \ga^\om\in C_{\rho,d}$ belongs to the image of $\Xi$.
\end{Corollary}

\begin{proof}
We have the (non-unital) algebra homomorphisms
$\Om_1 \colon \hat C_{\de} \to R^{\La_0}_{\cont(\rho)+\de}$ and 
$\Om_d \colon \hat C_{d\de} \to R^{\La_0}_{\cont(\rho)+d\de}$
defined as in~\eqref{EOm}. 
Recall the algebra homomorphism 
\[
\zeta:=\zeta_{\cont(\rho)+\de,\,(d-1)\de}\colon 
R^{\La_0}_{\cont(\rho)+\de}\to R^{\La_0}_{\cont(\rho)+\de,\,(d-1)\de}
\subseteq  R^{\La_0}_{\cont(\rho)+d\de}
\]
defined by~\eqref{EZetaHom}. It follows easily from the definitions that 
$\zeta\circ \Om_1 = \Om_d\circ \iota\colon \hat C_{\de} \to R^{\La_0}_{\cont(\rho)+d\de}$, whence $\iota(\ker \Om_1) \subseteq \ker \Om_d$.

Let $j\in J$. Identifying $\hat C_{\de}\otimes\hat C_{(d-1)\de}$ with $\hat C_{\de,(d-1)\de}\subseteq \hat C_{d\de}$ as usual, we have in $\hat C_{d\de}$: 
\begin{align*}
y_1 e_{j} \otimes 1_{(d-1)\de} =
\iota (y_1 e_j) &\in \iota( \Z(y_1-y_e)e_j + \ker \Om_1) \\
&= \Z (y_1-y_e)e_j \otimes 1_{(d-1)\de} + \iota(\ker \Om_1) \\
&\subseteq \Z (y_1-y_e) e_j \otimes 1_{(d-1)\de} + \ker\Om_d,
\end{align*}
where we have used Lemma~\ref{LD=1Deg2} for the first inclusion. Multiplying by $\ga^\om$, we get
\begin{align}
y_1\ga^\om=\sum_{j\in J}(y_1 e_{j} \otimes 1_{(d-1)\de} )\ga^\om
&\in \sum_{j\in J}(\Z (y_1-y_e) e_j \otimes 1_{(d-1)\de})\ga^\om + \ker\Om_d \notag
\\
&= \sum_{j\in J} \pm \Theta(\zc[1] \ze_j [1]) + \ker \Om_d,
\label{EY1GAOM}
\end{align}
where the last equality holds by Theorem~\ref{THomZZ}(i) and (\ref{E120216}). Now the lemma follows on applying $\Pi$.
\end{proof}

\begin{Theorem}\label{TXiIso}
The map $\Xi\colon W_d \to \ga^\om C_{\rho,d}\ga^\om$ is an  isomorphism of graded algebras. 
\end{Theorem}

\begin{proof}
By Theorem~\ref{THomZZ}(iii), the algebra $\ga^\om C_{\rho,d}\ga^\om$ is generated by 
$\Xi(W_d)$ together with the element $y_1 \ga^\om$. 
But $y_1 \ga^\om\in \Xi(W_d)$ by Corollary~\ref{CY1}, so $\Xi$ is surjective. 
By Corollary~\ref{CDimOmOm}, the algebras $W_d$ and 
$\ga^\om C_{\rho,d}\ga^\om$ are $\Z$-free of the same rank, and the result follows. 
\end{proof}

\begin{Lemma}\label{LProjCRhoD}
Let $\k$ be a field with $\charact \k =0$ or $\charact \k>d$.
The left module $C_{\rho,d,\k} \ga^\om$ is a projective generator for the algebra $C_{\rho,d,\k}$. 
\end{Lemma}

\begin{proof}
As $C_{\rho,d,\k}\ga^\om$ is projective, it is enough  to show that 
for every simple $C_{\rho,d,\k}$-module $L$ we have 
$\ga^\om L\cong \Hom_{C_{\rho,d,\k}} (C_{\rho,d,\k}\ga^\om, L) \ne 0$. 
But $L$ may also be viewed as a simple $\hat C_{d\de,\k}$-module via the natural surjection 
$\hat C_{d\de,\k} \twoheadrightarrow C_{\rho,d,\k}$. By Lemma~\ref{LProjGen}, the module $\hat C_{d\de,\k} \ga^\om$ is a projective generator for $\hat C_{d\de,\k}$, whence 
$\ga^\om L \cong \Hom_{\hat C_{d\de,\k}} (\hat C_{d\de,\k}\ga^\om, L) \ne 0$.
\end{proof}

\begin{Corollary}\label{CFaithful}
The $C_{\rho,d}$-module
$C_{\rho,d} \ga^\om$ is faithful. 
\end{Corollary}

\begin{proof}
By Lemma~\ref{LCFree}, the algebra $C_{\rho,d}$ is $\Z$-free, so it is enough to show that the $C_{\rho,d,\Q}$-module 
$C_{\rho,d,\Q} \ga^\om$ is faithful. By Lemma~\ref{LProjCRhoD}, this module is a projective generator for $C_{\rho,d,\Q}$, and the result follows.
\end{proof}

\subsection{Identifying $\ga^{\la,\bc} C_{d\de} \ga^{\om}$ with $M_{\la,\bc}$}

Let $n\in \Z_{>0}$ and $(\la,\bc) \in \La^{\col} (n,d)$. 
By Theorem~\ref{THomZZ}, the right 
$\ga^{\om} \hat C_{d\de} \ga^{\om}$-module
$\ga^{\la,\bc} \hat C_{d\de}\ga^{\om}$ becomes a right $W_d$-module via the map $\Theta$. Moreover, by Theorem~\ref{TXiIso},  the right 
$\ga^{\om} C_{\rho,d} \ga^{\om}$-module
$\ga^{\la,\bc} C_{\rho,d}\ga^{\om}$ becomes a right 
$W_d$-module via the map $\Xi$. In other words:
\begin{align}
  v  z &:= v \Theta(z) \qquad (v\in  
  \ga^{\la,\bc} \hat C_{d\de}\ga^{\om},\ z \in W_d), 
  \label{EVThetaZ}
  \\
  v z &:= v \Xi(z) \qquad (v\in  \ga^{\la,\bc} C_{\rho,d}\ga^{\om},\ z \in W_d).
  \label{EVZ}
\end{align}
It is clear from the definitions that 
$\Pi_{\la,\bc}\colon \ga^{\la,\bc} \hat C_{d\de} \ga^{\om}\to \ga^{\la,\bc} C_{\rho,d} \ga^{\om}$ is a surjective homomorphism of $W_d$-modules. 

Recall the colored permutation $W_d$-module $M_{\la,\bc}$ with generator $m_{\la,\bc}=1_{\la,\bc}\otimes \ze_{\la,\bc}$ defined by~\eqref{EMLaC}
and the
element $u_{\la,\bc} \in \ga^{\la,\bc} C_{d\de} \ga^\om$ introduced in~\S\ref{SSLaOm}. 

\begin{Lemma}\label{LModHatHom}
There is a degree-preserving $W_d$-module homomorphism
\[
\theta_{\la,\bc} \colon M_{\la,\bc} \to q^{-a_\la}\ga^{\la,\bc} \hat C_{d\de} \ga^{\om}, \; m_{\la,\bc} \mapsto u_{\la,\bc}.
\]
\end{Lemma}

\begin{proof}
By~\eqref{EThLaC} and Lemma~\ref{LStyagIdemp}, we have 
$u_{\la,\bc}\Theta(\ze_{\la,\bc}) =u_{\la,\bc}e_{\la,\bc}=u_{\la,\bc}$. By Lemma~\ref{LUinZhat}, 
for any $g\in \Si_{\la}$ we have 
$u_{\la,\bc} \Theta(g) = \eps_{\la,\bc} (g) u_{\la,\bc}$. Using Lemma~\ref{LDegU}, we deduce that there is a degree-preserving $W_{\la,\bc}$-module homomorphism 
$\alt_{\la,\bc} \to q^{-a_\la}\ga^{\la,\bc} \hat C_{d\de}\ga^\om,\ 1_{\la,\bc} \mapsto u_{\la,\bc}$. This map induces a $W_d$-module homomorphism $\theta_{\la,\bc}$ as in the statement of the lemma. 
\end{proof}

From now on, we write $\bar u_{\la,\bc}:=\Pi_{\la,\bc}(u_{\la,\bc})\in C_{\rho,d}$. 

\begin{Theorem}\label{TModIso}
For any $(\la,\bc) \in \La^\col (n,d)$,
there is an   isomorphism of graded $W_d$-modules: 
\[
\eta_{\la,\bc} \colon M_{\la,\bc} \iso q^{-a_\la}\ga^{\la,\bc} C_{\rho,d} \ga^{\om}, \; m_{\la,\bc}  \mapsto \bar u_{\la,\bc}.
\]
\end{Theorem}

\begin{proof}
Let $\theta_{\la,\bc}$ be as in Lemma~\ref{LModHatHom}. We have a homomorphism of $W_d$-modules 
$$\eta_{\la,\bc}:=\Pi_{\la,\bc}\circ \theta_{\la,\bc}\colon M_{\la,\bc} \to q^{-a_\la}\ga^{\la,\bc} C_{\rho,d} \ga^{\om}, \; m_{\la,\bc}  \mapsto \bar u_{\la,\bc}.$$ 
By Corollary~\ref{CCyc}, the right 
$\ga^\om C_{\rho,d} \ga^\om$-module $q^{-a_\la}\ga^{\la,\bc} C_{\rho,d} \ga^{\om}$ 
is generated by $\bar u_{\la,\bc}$. Using Theorem~\ref{TXiIso}, we conclude that $\eta_{\la,\bc}$ is surjective.
By Corollary~\ref{CDimEqual}, the $\Z$-modules $M_{\la,\bc}$ and 
$q^{-a_\la}\ga^{\la,\bc} C_{\rho,d} \ga^{\om}$ are free of the same (ungraded) rank, and the theorem follows. 
\end{proof}

\subsection{The algebra $E(n,d)$ and the double $D_Q (n,d)$}\label{SSEnd}

Fix $n\in \Z_{>0}$.
Recall the tuple $\bc^0 \in J^{n(e-1)}$, the modules $M^\la=M_{\la,\bc^0}$ and the algebra $S^\Zig(n,d)$ from \S\ref{SSTDSA}.
Let $\la \in \La(n(e-1),d)$. Define the idempotent
\[
 f^\la:= \ga^{\la,\bc^0}.
\]
Recall the integer $$a_\la =-e\sum_{t=1}^{n(e-1)}\la_t(\la_t-1)/2= \deg(u_{\la,\bc^0}),$$ 
see~\eqref{Eala} and Lemma~\ref{LDegU}. 
In the sequel, we abbreviate 
\begin{align*}
u_{\la}&:=u_{\la,\bc^0}, \quad \bar u_\la:=\bar u_{\la,\bc^0},\quad 
e_{\la}:=e_{\la,\bc^0},\quad \ze_{\la}:=\ze_{\la,\bc^0}, \quad \eps_{\la}:=\eps_{\la,\bc^0}, \\
\theta_{\la}&:=
\theta_{\la,\bc^0}\colon M^\la \to q^{-a_\la}f^\la \hat C_{d\de} \ga^\om,\quad
\eta_{\la}
:=
\eta_{\la,\bc^0}\colon M^\la \iso q^{-a_\la}f^\la C_{\rho,d} \ga^\om,
\end{align*}
where $\theta_{\la,\bc^0}$ is the homomorphism of Lemma~\ref{LModHatHom} and $\eta_{\la,\bc^0}$ is the isomorphism of 
Theorem~\ref{TModIso}.

Define the left $C_{\rho,d}$-module
\[
\Ga (n,d):= \bigoplus_{\la\in \La(n(e-1),d)} q^{a_\la} C_{\rho,d} f^\la
\]
and the algebra
\begin{equation}\label{EDefE}
 E(n,d):= \End_{C_{\rho,d}} (\Gamma (n,d))^{\op}.
\end{equation}
Let $\la,\mu\in \La(n(e-1),d)$. We identify the (graded) 
$\Z$-module $q^{a_{\la}-a_{\mu}} f^\mu C_{\rho,d} f^\la$ with the $\Z$-submodule of $E(n,d)$ consisting of the endomorphisms that send the summand $q^{a_\mu}C_{\rho,d} f^\mu$ to $q^{a_\la}C_{\rho,d} f^\la$ and send the other summands to zero. Specifically, an element $x\in q^{a_{\la}-a_{\mu}} f^\mu C_{\rho,d} f^\la$ corresponds to the homomorphism given by the right multiplication:
\[
q^{a_\mu}C_{\rho,d} f^\mu \to q^{a_\la}C_{\rho,d} f^\la,\ v\mapsto vx.
\]
Thus,
\begin{equation}\label{EEDec}
 E(n,d) = \bigoplus_{\la,\mu\in \La(n(e-1))} q^{a_\la-a_\mu} f^\mu C_{\rho,d} f^\la.
\end{equation}

 Let
$x\in q^{a_\la-a_\mu} f^\mu C_{\rho,d} f^\la$. Recalling the right $W_d$-module structure (\ref{EVZ}),  we have a $W_d$-module homomorphism 
\[
q^{-a_\la}f^\la C_{\rho,d} \ga^\om \to q^{-a_\mu}f^\mu C_{\rho,d} \ga^\om,\ 
v \mapsto xv. 
\] 
Identifying $q^{-a_\la}f^\la C_{\rho,d} \ga^\om$ with $M^\la$ and 
$q^{-a_\mu}f^\mu C_{\rho,d} \ga^\om$ with $M^\mu$ via the isomorphisms of Theorem~\ref{TModIso}, we obtain an element $\Phi(x)\in \Hom_{W_d} (M^\la,M^\mu)$. In other words,
\[
\Phi(x)\colon M^\la \to M^\mu, \ v \mapsto 
\eta_{\mu}^{-1} ( x\eta_\la (v)) \qquad (v\in M^\la). 
\]
Recall from \S\ref{SSTDSA} that $\Hom_{W_d} (M^\la,M^\mu)$ is identified with $\xi_\mu S^{\Zig}(n,d)\xi_\la$. 
The assignments $x\mapsto \Phi(x)$ for all $\la,\mu\in \La(n(e-1))$ and all $x\in q^{a_\la-a_\mu} f^\mu C_{\rho,d} f^\la $ extend uniquely to a $\Z$-linear map 
\[
\Phi \colon E(n,d) \to S^{\Zig} (n,d). 
\]

\begin{Lemma}\label{LPhiHom}
The map $\Phi\colon E(n,d) \to S^\Zig (n,d)$ is a homomorphism of graded algebras. 
\end{Lemma}

\begin{proof}
That $\Phi$ is a homomorphism of ungraded algebras follows easily from the definitions. Let $x\in q^{a_\la-a_\mu} f^\mu C_{\rho,d} f^\la$ be a homogeneous element for some $\la,\mu\in \La(n(e-1),d)$. 
Then, by definitions, 
$\Phi(x)\colon m^\la \mapsto m^\mu z$ for some homogeneous $z\in W_d$ such that $x \bar u_\la = \bar u_\mu \Xi(z)$ in $q^{-a_\mu}f^\mu C_{\rho,d}\ga^\om$. Hence, $\Phi(x)$ is homogeneous  of degree 
\[
\deg(z) =\deg(\Xi(z))= \deg(x)-(a_\la-a_\mu) + \deg(\bar u_\la) - \deg(\bar u_\mu) = \deg(x),
\]
where the last equality is due to Lemma~\ref{LDegU}. 
\end{proof}

\begin{Corollary}\label{CPhiInj}
The algebra homomorphism $\Phi \colon E(n,d) \to S^\Zig (n,d)$ is injective.
\end{Corollary}

\begin{proof}
If not, then there exist $\la,\mu\in\La(n(e-1),d)$ and 
$0\ne x \in f^\mu C_{\rho,d}f^\la$ such that $x f^\la C_{\rho,d} \ga^\om=0$, whence $xC_{\rho,d}\ga^{\om}=0$. But this is impossible by Corollary~\ref{CFaithful}.
\end{proof}

\begin{Lemma}\label{LPhiDeg0}
We have $\Phi(E(n,d)) \supseteq S^\Zig(n,d)^0$.
\end{Lemma}

\begin{proof}
Suppose that $\la,\mu\in \La(n(e-1),d)$ and $h\in \Hom_{W_d} (M^\la,M^\mu)^0$. Then there exists $z\in W_d^0$ such that $h(m^\la) = m^\mu z$. 
Hence, $m^\mu z \ze_{\la} = m^\mu z$ and $m^\mu z g= \eps_{\la} (g) m^\mu z$ for all $g\in \Si_\la$. By Lemma~\ref{LModHatHom}, (\ref{EVThetaZ}) and~\eqref{EThLaC}, 
it follows that the element 
$$v:=\theta_\la(m^\mu z)=u_\mu \Theta(z) \in q^{-a_\mu}f^\mu \hat C_{d\de} \ga^\om$$
has degree zero and satisfies $v= v e_{\la}$ and 
$v \Theta(g) = \eps_\la (g) v$ for all $g\in \Si_\la$. 
By Lemma~\ref{LLeftMult}, there exists $x\in f^\mu \hat C_{d\de}f^\la$ such that $v= x u_\la$. Applying the surjection $\Pi$ to this equality and writing 
$\bar x:=\Pi(x)\in q^{a_\la-a_\mu}f^\mu C_{\rho,d}f^\la$, we have $\bar u_{\mu} \Xi(z) = \bar x \bar u_{\la}$, cf.~\eqref{EXi}.
But $\bar u_\la = \eta_\la(m^\la)$ and $\bar u_\mu \Xi(z) 
= \eta_{\mu} (m^\mu z)$, so 
the map $\Phi(\bar x)$ sends $m^\la$ to $m^\mu z$. Thus, 
$\Phi(\bar x) = h$, so $h\in \Phi(E(n,d))$. The lemma follows. 
\end{proof}

Recall the algebra homomorphisms  
$\itt^\la\colon \Zig \to S^\Zig (n,d)$ from~\eqref{EILa}.

\begin{Lemma}\label{LPhiSpecial}
For any $\la \in\La((n-1)(e-1),d-1)$, we have $\itt^\la (\Zig) \subseteq \Phi(E(n,d))$. 
\end{Lemma}

\begin{proof}
Let $z\in \ze_j \Zig \ze_k$ for some $k,j\in J$. Recall the embedding $\iota \colon \hat C_{\de} \to \hat C_{d\de}$ from~\eqref{EIotaChat}. 
It follows from Theorem~\ref{THomZZ}(i) and (\ref{E120216}) that 
there exists $x\in e_j \hat C_{\de}e_k$ such that 
$\Theta(z[1])= \iota(x)$. Note that 
\[
u_{\hat\la^k} = e_k \otimes u_\la\in \hat C_{\de} \otimes \hat C_{(d-1)\de}= \hat C_{\de,(d-1)\de} \subseteq \hat C_{d\de},
\]
and $u_{\hat\la^j}$ is described similarly. Hence, $\iota(x) u_{\hat\la^k} = u_{\hat\la^j} \iota(x)= u_{\hat\la^j}\Theta(z[1])$.
Writing $\bar x:=\Pi(\iota(x)) f^{\hat\la^k} = f^{\hat\la^j} \Pi(\iota(x))\in f^{\hat\la^j} C_{\rho,d}f^{\hat\la^k}$, we have 
\[
\bar x \bar u_{\hat\la^k} = \bar u_{\hat\la^j} \bar x = \bar u_{\hat\la^j} \Xi(z[1]),
\]
whence 
\begin{align*}
\Phi(\bar x)(m^{\hat\la^k})&=\eta^{-1}_{\hat\la^j}(\bar x \eta_{\hat\la^k}(m^{\hat\la^k}))=\eta^{-1}_{\hat\la^j}(\bar x \bar u_{\hat\la^k})
=\eta^{-1}_{\hat\la^j}( \bar u_{\hat\la^j} \Xi(z[1]))=m^{\hat\la^j} z[1]
\\
&=\itt^\la (z)(m^{\hat\la^k}).
\end{align*}
So $\Phi(\bar x) = \itt^\la (z)$, and the lemma follows. 
\end{proof}

\begin{Lemma}\label{LESym}
For every field $\k$, the $\k$-algebra $E(n,d)_{\k}$ is symmetric. 
\end{Lemma}

\begin{proof}
By Corollary~\ref{CCSymmetric}, the algebra $C_{\rho,d,\k}$ is symmetric. 
It follows from~\eqref{EDefE} that 
$E(n,d)_{\k} \cong \End_{C_{\rho,d,\k}} 
(\Gamma(n,d)_{\k} )^{\op}$. Since 
$\Gamma(n,d)_{\k}$ is a projective $C_{\rho,d,\k}$-module, the lemma follows by~\cite[Proposition IV.4.4]{SY}. 
\end{proof}

Recall the subalgebra $T^\Zig (n,d)\subseteq S^\Zig (n,d)$ from 
\S\ref{SSTDSA}.

\begin{Theorem}\label{TMainIso}
Suppose that $n\ge d$. Then
we have an isomorphism of graded algebras $\Phi\colon E(n,d) \iso T^\Zig (n,d)$. 
\end{Theorem}

\begin{proof}
By Lemma~\ref{LPhiHom} and Corollary~\ref{CPhiInj}, the map $\Phi\colon E(n,d) \to S^\Zig (n,d)$ is an injective homomorphism of graded algebras, so 
$E(n,d) \cong \Phi(E(n,d))$.
By Lemmas~\ref{LPhiDeg0} and~\ref{LPhiSpecial},
we have $T^\Zig(n,d) \subseteq \Phi(E(n,d))$. 
By Lemma~\ref{LESym}, for every prime $p$, the algebra 
$\Phi(E(n,d))\otimes_{\Z} \F_p$ is symmetric. An application of Theorem~\ref{TSymmetricity} completes the proof.
\end{proof}

\begin{Corollary} \label{CKey} 
Let $n\geq d$. Then $E(n,d) \cong D_Q (n,d)$. 
\end{Corollary}
\begin{proof}
This follows from Theorems~\ref{TGeneration} and \ref{TMainIso}.
\end{proof}

\begin{Example}
Recall the idempotents 
$\xi_\la \in S^\Zig (n,d)$ defined in~\S\ref{SSTDSA} for any $\la\in \La(n(e-1),d)$. It follows from the definitions that for all $\la,\mu\in \La(n(e-1),d)$, the homomorphism $\Phi$ maps the component 
$q^{a_\la-a_\mu} f^\mu C_{\rho,d} f^\la$ of the decomposition~\eqref{EEDec} of $E(n,d)$ into the component 
$\xi_\mu S^\Zig (n(e-1),d) \xi_\la= \Hom_{W_d} (M^\la, M^\mu)$ of $S^\Zig(n,d)$. In this example, 
we consider the case when $e=2$, $d=2$, $n=2$ and $\la=(2,0)$, and 
we identify 
 $\Phi(f^\la C_{\rho,d} f^\la)$ as an explicit subalgebra of 
 $\End_{W_2} (M^\la)$. 

Let $x_1, x_2\in \End_{W_2} (M^\la)$ be the endomorphisms defined by the properties that $x_1 (m^\la) = m^\la (\zc[1]+\zc[2])$ and $x_2 (m^\la) = m^\la \zc[1] \zc[2]$. 
Then $\{ 1:=\xi_\la, x_1,x_2\}$ is a $\Z$-basis of the commutative algebra $\End_{W_2} (M^\la)$, and $x_1^2 = 2 x_2$, $x_1x_2=0$. Moreover, 
it is easy to see as in~\cite[Example 4.27]{EK}
that 
$\xi_\la T^\Zig (n,d) \xi_\la$ is the $\Z$-span of $\{1,x_1,2x_2\}$, so 
$\xi_\la T^\Zig (n,d) \xi_\la$ is isomorphic to the truncated polynomial algebra $\Z[z]/(z^3)$, with $x_1$ corresponding to $z$. Thus, Theorem~\ref{TMainIso} asserts, in particular, that 
$\Phi(f^\la C_{\rho,d} f^\la) = \Z 1 \oplus \Z x_1 \oplus 2\Z x_2$. 
This assertion can also be verified by direct calculations using 
 (1) the defining relations of the affine zigzag algebra $1_{0101} \hat C_{2\de} 1_{0101}$, see~\cite[Definition 3.3]{KM2}; and 
 (2) the fact that $y_1 \ga^\om = y_1 1_{0101} = a (y_1-y_2) 1_{0101}$ in $C_{\rho,2}$ for some $a\in \Z$, see~\eqref{EY1GAOM}.

\end{Example}

\subsection{Morita equivalences}\label{SSMorita}
Let $A$ and $B$ be graded $\Z$-algebras. A {\em graded} functor $\mod{A}\to\mod{B}$ is a functor $\funF$ equipped with an equivalence between $q\circ \funF$ and $\funF\circ q$. A graded functor $\funF$ is a {\em graded equivalence} if it is an equivalence of categories (in the usual sense). The graded algebras $A$ and $B$ are {\em graded Morita equivalent} if there is a graded equivalence between $\mod{A}$ and $\mod{B}$. As noted for example in \cite[\S II.5.3]{Van} the graded analogue of Morita theory holds. In particular, $A$ is graded Morita equivalent to $B$ if and only if there exists a graded projective left $A$-module $P$ which is a projective generator and such that $B\cong \End_A(P)^\op$.  

For a graded algebra $A$, recall the notation $\ell(A)$ from~\S\ref{SSAlg}.

\begin{Lemma}\label{LMoritaField}
Let $A$ be a graded $\Z$-algebra which is finitely generated as a $\Z$-module, and let $\eps\in A$ be a homogeneous idempotent. 
Suppose that for every prime $p$ we have 
$\ell(A_{\bar \F_p})=\ell((\eps\otimes 1) A_{\bar \F_p}(\eps\otimes 1))$. Then the algebras $A$ and $\eps A \eps$ are graded Morita equivalent. 
\end{Lemma}

\begin{proof}
We write $\eps:=\eps\otimes 1\in A_{ \bar \F_p}$ for each prime $p$.
It suffices to show that the left $A$-module $A\eps$ is a projective generator for $A$ or, equivalently, that $A\eps A=A$. 
Assume that $A\eps A \ne A$. Then there exists a prime $p$ 
such that $A_{\bar \F_p}\eps A_{\bar \F_p} \ne A_{\bar \F_p}$. If $L$ is a composition factor of $A_{\bar \F_p}/A_{\bar \F_p}\eps A_{\bar \F_p}$, then $\eps L=0$, which contradicts the assumption that 
$\ell(A_{\bar \F_p})=\ell(\eps A_{\bar \F_p}\eps)$, for example by \cite[Theorem 6.2(g)]{Green}. 
\end{proof}

Let $\la \in \La(n(e-1),d)$.
It follows from the definitions in~\S\ref{SSGG} that 
$\bl (\la,\bc^0)$ is obtained from $\hat \bl (\la,\bc^0)$ by replacing
each subword of the form $i^m$ that is not preceded by or followed by $i$ with $i^{(m)}$. 
Therefore, for any $\la,\mu \in \La(n(e-1),d)$, we either have $f^\la = f^\mu$ or $f^\la f^\mu = f^\mu f^\la =0$. We have an equivalence relation on $\La(n(e-1),d)$, with $\la$ being equivalent to $\mu$ if and only if $f^\la = f^\mu$. 
Let $\mathcal X\subseteq \La(n(e-1),d)$ be a set of representatives of equivalence classes.
Define 
$
 f:= \sum_{\la\in \mathcal X} f^\la\in C_{\rho,d}.
$
 Then $f^2=f$ is a homogeneous idempotent. 

\begin{Lemma}\label{LProjIdMorita}
The algebra $E(n,d)$ is graded Morita equivalent to $f C_{\rho,d} f$. 
\end{Lemma}

\begin{proof}
Consider the left $f C_{\rho,d}f$-module 
\[
f \Ga(n,d) = 
\bigoplus_{\la\in \La(n(e-1),d)} q^{a_\la} f C_{\rho,d} f^\la.
\]
There is a surjective 
$f C_{\rho,d}f$-module homomorphism $f \Ga(n,d)\to f C_{\rho,d}f$ which is the identity on the summands $f C_{\rho,d} f^\la$ for $\la\in \mathcal X$ and zero on the other summands. 
Hence, 
$f \Ga(n,d)$ is a projective generator for $fC_{\rho,d} f$. 
It is easy to see that 
$E(n,d) \cong \End_{f C_{\rho,d} f} (f \Ga(n,d))^\op$, 
since for all $\la,\mu\in \La(n(e-1),d)$ we have
$$\Hom_{f C_{\rho,d} f}(q^{a_\mu} f C_{\rho,d} f^\mu, q^{a_\la} f C_{\rho,d} f^\la)\cong q^{a_\la-a_\mu}f^\mu f C_{\rho,d} f^\la
=q^{a_\la-a_\mu}f^\mu C_{\rho,d} f^\la.
$$ 
The lemma follows by graded Morita theory. 
\end{proof}

Write $\al= \cont(\rho)+d\de$, so that $R^{\La_0}_{\al}$ is the RoCK block of~\S\ref{SSRoCK}.
For any $m,h\in\Z_{\ge 0}$, 
we denote by $S(m,h)$ the usual Schur algebra over $\Z$, as in~\cite{Green}. 

\begin{Theorem}\label{TMainMorita}
Suppose that $n\ge d$. Then the $\Z$-algebras $R^{\La_0}_{\al}$ and $D_Q (n,d)$ are graded Morita equivalent.
\end{Theorem}

\begin{proof}
By Remark~\ref{RCrazy}, there is a homogeneous idempotent 
$\mathbf{e} \in R^{\La_0}_{\al}$ such that 
$C_{\rho,d} \cong \mathbf{e} R^{\La_0}_{\al} \mathbf{e}$. 
Hence, by Lemma~\ref{LProjIdMorita}, 
there exists a homogeneous idempotent $\eps\in R^{\La_0}_{\al}$ such that $E(n,d)$ is graded Morita equivalent 
to $\eps R^{\La_0}_{\al} \eps$. 
By Corollary~\ref{CKey}, we have 
$E(n,d)\cong D_Q (n,d)$, so $\eps R^{\La_0}_{\al} \eps$ is graded Morita equivalent to $D_Q (n,d)$. So it suffices to show that $\eps R^{\La_0}_{\al} \eps$ is graded Morita equivalent to $R^{\La_0}_{\al}$. 

Let $p$ be a prime, and write $\eps:=\eps\otimes 1 \in R^{\La_0}_{\al, \bar \F_p}$.
By the first paragraph, the algebras $\eps R^{\La_0}_{\al,  \bar \F_p}\eps$ and 
$D_{Q} (n,d)_{\bar \F_p}$ 
are graded Morita equivalent. 
In particular, $\ell (\eps R^{\La_0}_{\al, \bar \F_p} \eps) 
= \ell(D_{Q} (n,d)_{\bar \F_p})$.
By Lemma~\ref{LMoritaField}, it remains to show only that 
$\ell(R^{\La_0}_{\al, \bar\F_p})=\ell(D_{Q} (n,d)_{\bar \F_p})$. 

Since the algebra $D_{Q} (n,d)_{\bar \F_p}$ is non-negatively graded, we have 
$\ell(D_{Q} (n,d)_{\bar \F_p})=\ell(D_{Q}(n,d)_{\bar \F_p}^0)$.
By~\cite[(7.2) and Lemma 7.3]{EK} together with Theorem~\ref{TGeneration}, 
\[
D_Q(n,d)^0 \cong \bigoplus_{(d_1,\dots,d_{e-1})\in \La(e-1,d)} S(n, d_1) \otimes \dots \otimes S(n,d_{e-1}). 
\]
By~\cite[Theorem 3.5(a)]{Green}, for all $h\le n$ we have  $\ell(S(n,h)_{\bar \F_p})= |\Par(h)|$. It follows that 
$\ell(D_{Q} (n,d)_{\bar \F_p}) = |\Par^J (d)|$.
On the other hand, by Theorem~\ref{TNSimples}, we have $\ell(R^{\La_0}_{\al, \bar \F_p})=|\Par^J(d)|$, and 
the proof is complete.
\end{proof}

Thus, we have proved Theorem~\ref{TA}. In conclusion, we consider the case where we work over a field of sufficiently large characteristic, cf.~the discussion in Section~\ref{SIntro}.

\begin{Proposition}\label{PFourAlgebras}
Suppose that $n\ge d$ and $\k$ is a field with $\charact \k =0$ or $\charact \k>d$. Then the RoCK block $R^{\La_0}_{\al,\k}$, the Turner double $D_Q(n,d)_{\k}$ and the wreath products $W_{d,\k}$ and $(R^{\La_0}_{\de,\k})^{\otimes d} \rtimes \k \Si_d$ are all graded Morita equivalent to each other. 
\end{Proposition}

\begin{proof}
We write $x:=x\otimes 1\in A_{\k}$ for any algebra $A$ and any $x\in A$. 
By (the proof of) Theorem~\ref{TMainMorita}, the algebras 
$R^{\La_0}_{\al,\k}$, $C_{\rho,d,\k}$ and $D_Q(n,d)_{\k}$ are graded Morita equivalent. By Lemma~\ref{LProjCRhoD}, the module $C_{\rho,d,\k} \ga^\om$ is a projective generator for $C_{\rho,d,\k}$, so $C_{\rho,d,\k}$ is graded Morita equivalent to 
\[
\End_{C_{\rho,d,\k}} (C_{\rho,d,\k} \ga^\om)
\cong \ga^\om  C_{\rho,d,\k}\ga^{\om} \cong W_{d,\k}.
\]
where the second isomorphism comes from Theorem~\ref{TXiIso}. 
Recall the idempotent $e_J$ from~\eqref{EEJ}.
By the $d=1$ case of Lemma~\ref{LProjCRhoD}, the module $R^{\La_0}_{\de,\k}e_J$ is a projective generator for $R^{\La_0}_{\de,\k}$. Hence, setting $\xi:=e_J^{\otimes d}$, we have that $((R^{\La_0}_{\de,\k})^{\otimes d} \rtimes \k \Si_d) \xi$ is a projective generator for $(R^{\La_0}_{\de,\k})^{\otimes d} \rtimes \k \Si_d$. So $(R^{\La_0}_{\de,\k})^{\otimes d} \rtimes \k \Si_d$ is graded Morita equivalent to 
\[
\xi ((R^{\La_0}_{\de,\k})^{\otimes d} \rtimes \k \Si_d) \xi \cong 
(e_J R^{\La_0}_{\de,\k} e_J)^{\otimes d} \rtimes \k\Si_d 
\cong (\Zig_{\k})^{\otimes d} \rtimes k\Si_d \cong W_{d,\k},
\]
where for the second isomorphism we use the fact that 
$e_J R^{\La_0}_{\de,\k} e_J \cong \Zig_\k$, see~\cite[Theorem 4.17]{KM2}. 
\end{proof}

\end{document}